\documentclass[a4paper,11pt,reqno]{amsart}
\usepackage{amsmath,amssymb,amsthm} 
\usepackage{graphics} 
\usepackage{enumerate}
\usepackage[colorlinks,citecolor=blue]{hyperref}

\usepackage[capitalize]{cleveref}

\usepackage{epsfig}
\usepackage{color}
\usepackage[english]{babel}

\numberwithin{equation}{section}

\newtheorem{theorem}{Theorem}[section]
\newtheorem{lemma}[theorem]{Lemma}
\newtheorem{proposition}[theorem]{Proposition}

\newtheorem{corollary}[theorem]{Corollary}

\theoremstyle{definition}

\newtheorem{remark}[theorem]{Remark}

\newcommand{\restr}{\mathop{\raisebox{-.127ex}{\reflectbox{\rotatebox[origin=br]{-90}{$\lnot$}}}}}

\newcommand{\R}{\mathbb{R}}
\newcommand{\N}{\mathbb{N}}

\newcommand{\eps}{\varepsilon}

\newcommand{\be}{\begin{equation}}
\newcommand{\ee}{\end{equation}}

\DeclareMathOperator{\supp}{supp}

\newcommand\lt{\left}
\newcommand\rt{\right}

\def\les{\lesssim}

\newcommand{\cF}{\mathcal{F}}

\newcommand{\cI}{\mathcal{I}}

\renewcommand{\Cap}{\operatorname{Cap}}

\newcommand{\TV}{\dist_{TV}}

\def\EE{\mathbb{E}}
\def\PP{\mathbb{P}}
\def\diam{\operatorname{diam}}

\newcommand{\cM}{\mathcal{M}}
\newcommand{\cS}{\mathcal{S}}
\renewcommand{\c}{\mathsf{c}}

\newcommand{\bra}[1]{\left( #1 \right)}
\newcommand{\sqa}[1]{\left[ #1 \right]}
\newcommand{\cur}[1]{\left\{ #1 \right\}}

\newcommand{\abs}[1]{\left| #1 \right|}
\newcommand{\nor}[1]{\left\| #1 \right\|}
\def\fref{f}

\newcommand{\T}{\mathbb{T}}
\newcommand{\dist}{\mathsf{d}}

\newcommand{\Lip}{\operatorname{Lip}}

\newcommand{\cR}{\mathcal{R}}

\newcommand{\cQ}{\mathcal{Q}}

\newcommand{\dil}{\operatorname{dil}}
\newcommand{\tras}{\operatorname{trans}}

\usepackage[hyperref,doi,url=false,doi=false, isbn=false, giveninits=true, backref,style=numeric,maxbibnames=99, backend=biber]{biblatex}

\bibliography{OT.bib}

\title[Wasserstein Asymptotics for Brownian Motion on the Torus and Interlacements]{Wasserstein Asymptotics for  Brownian Motion on the Flat Torus and Brownian Interlacements}
\author[M. Mariani]{Mauro Mariani}
\address{M.M.: Faculty of Mathematics, HSE University, 101000 Moscow, Russia}
\email{}
\author[D. Trevisan]{Dario Trevisan}
\address{D.T.: Dipartimento di Matematica, Università degli Studi di Pisa, 56125 Pisa, Italy  }
\email{dario.trevisan@unipi.it}
\date{\today}
\subjclass[2010]{60D05, 90C05, 39B62, 60F25, 35J05}
\keywords{optimal transport, geometric probability, Brownian interlacements}
\newcounter{proof-step}

\usepackage{tikz}
\begin{document}

\maketitle

\begin{abstract}
We study the large time behavior of the optimal transportation cost towards the uniform distribution, for the occupation measure of a stationary Brownian motion on the flat torus in $d$ dimensions, where the cost of transporting a unit of mass is given by a power of the flat distance. We establish a global upper bound, in terms of the limit for the analogue problem concerning the occupation measure of the Brownian interlacement on $\R^d$. We conjecture that our bound is sharp and that our techniques may allow for similar studies on a larger variety of problems, e.g.\ general diffusion processes on weighted Riemannian manifolds.
\end{abstract}

\section{Introduction}

Given a stochastic process $(X_t)_{t \ge 0}$, its occupation measure up to a time $T> 0$ can be defined as  the (random) measure
\begin{equation}
\mu_T^X = \int_0^T \delta_{X_s} ds, \quad \mu_T^X(A) = \int_0^T I_{\cur{X_s \in A}} ds,
\end{equation}
 If renormalized to be a probability measure, it is also known as the empirical measure of $X$. From the simplest case of pure jump process associated to i.i.d.\ random variables $(Y_n)_{n =0}^\infty$, i.e., $X_t = Y_{\lfloor t \rfloor}$, to that of diffusion processes on manifolds, the occupation measure has a plethora of applications, ranging from non-parametric statistics, to Monte Carlo methods and mean field theory.

Under natural assumptions on $X$, such as stationarity and ergodicity, limit theorems can be established for the empirical measure, as the time horizon $T$ increases, showing convergence towards its invariant measure. It is then a relevant question in applications to quantify such convergence, using a suitable metric between measures defined on the state space of the process. A particularly compelling choice, assuming that the state spaces is already equipped with a distance, e.g.\ if $X$ takes values in a Riemannian manifold, is given by  the optimal transport (Wasserstein) cost of order $p$, for some $p>0$. 
Also known as the earth mover's distances in the case $p=1$, the metric is defined as the minimum total cost of moving a source mass distribution $\mu$ towards the target distribution $\lambda$, where the cost of transporting a unit of mass is given by the $p$-th power of the distance in the underlying state space. Classically, the case $p=1$ has been the preferred choice, also because of the celebrated Kantorovich dual formulation, which represents  the minimum transport cost as a maximum discrepancy between the integrals  with respect to $\mu$ and $\lambda$ over all the $1$-Lipschitz functions. In recent years, other choices of $p$, particularly $p=2$, have been the subject of intense investigation, see the monographs \cite{AGS, VilOandN}. Also the ``concave'' case $p<1$ yields interesting features, as studied in \cite{MR1440931, mccann1999exact}.

The problem of establishing rates of convergence for the Wasserstein cost of the occupation measure for  pure jump process associated to i.i.d.\ random variables $(Y_n)_{n =0}^\infty$  is  strongly related to the so-called random assignment (or bipartite matching) problem, of combinatorial nature.  Indeed, if instead of a single process one considers two (independent) families of variables $(Y_n)_{n =i}^\infty$, $(Z_n)_{n =i}^\infty$. Then, for every $T = n$ one can study the assignment cost
\begin{equation}
 \min_{\sigma \in\cS_n} \sum_{i=1}^n \dist(Y_i, Z_{\sigma(i)})^p
 \end{equation}
where $\cS_n$ denotes the set of permutations over $n$ elements. By Birkhoff's theorem on doubly stochastic matrices, it is well-known that the cost above equals  the Wasserstein cost of order $p$ between the two occupation measures
\begin{equation}
\mu_n = \sum_{i=1}^n \delta_{Y_i}, \quad \lambda_n =\sum_{i=1}^n \delta_{Z_i}.
\end{equation}
From this point of view, it is natural to expand the techniques developed for  the assignment problem to the study of asymptotic rates of convergence for more general stochastic processes, e.g.\ diffusions on Riemannian manifolds. Indeed, the literature on the random assignment problem is vast and growing: stemming from the seminal works \cite{dudley1969speed}, it stimulated powerful functional analytic techniques \cite{talagrand1992matching, talagrand2014upper} and combinatorial/geometrical ones \cite{AKT84, dobric1995asymptotics, BoutMar, BaBo}. A renewed interest due to powerful predictions by the statistical physics community \cite{caracciolo2015scaling, caracciolo2019anomalous, sicuro_euclidean_2017, caracciolo2020dyck, BeCa}  lead recently to the development of novel methods \cite{AST, ledoux2019optimal, goldman2021quantitative, goldman2021convergence}   that have found several  applications \cite{bobkov2020transport, AmGlau, AGT19, goldman2018large, ambrosio2022quadratic, goldman2022fluctuation, huesmann2021there, goldman2023concave}, even beyond the case of i.i.d.\ points \cite{jalowy2021wasserstein, clozeau2023annealed} and also for other combinatorial optimization problems \cite{capelli2018exact,goldman2022optimal}. 

In the case of continuous processes on manifolds, F.-Y.~Wang and collaborators pioneered systematic exploitation of the PDE tools from \cite{AST,ledoux2019optimal} to the exploration of asymptotic rates for occupation measure \cite{wang2019limit, wang2021convergence, wang2020convergence, wang2021precise, wang2021wasserstein, wang2022wasserstein, wang2023convergence}. Let us mention that similar techniques have been also extended to non-Markov processes such as the fractional Brownian motion \cite{huesmann2022wasserstein, li2023wasserstein}. 

However, a known issue that afflicts the original PDE methods \cite{AST,ledoux2019optimal} (but also in some sense  \cite{goldman2021quantitative}) is that when the dimension of the underlying manifold grows (i.e., for $d \ge 3$ in the random assignment problem and $d \ge 5$ for diffusion processes) the upper and lower bounds resulting from a ``global'' application of the methods become less precise and yield (conjectured) non-optimal constants, although they still match the correct rates. To overcome this, in the setting of the assignment problem, it was first put forward in \cite{goldman2021convergence} and later developed in \cite{ambrosio2022quadratic, goldman2022optimal} the need for further ``localize'' the problem, using geometric decompositions of Whitney-type.

\subsection{Main results}

Aim of this paper is to show for the first time that it is possible to adapt these localization arguments in the setting of diffusion processes on Riemannian manifolds, obtaining (conjectured) sharp results. We focus on the simplest of all cases, namely that of Brownian motion on the flat torus, where the  occupation measure converges towards the uniform (Lebesgue) distribution. It turns out that the limit of the Wasserstein cost of the occupation measure on sufficiently small scales is related to the limit of a suitably defined notion of occupation measure associated to the Brownian (continuous time) interlacement process on $\R^d$, as introduced in \cite{sznitman2013scaling}. Such a connection is quite natural, in view of similar results about scaling limits of the support of a Brownian  process in the torus (for an exposition in the discrete time case, see \cite{drewitz2014introduction}). Intuitively, the Brownian interlacement plays here the role of a Poisson point process in the case of i.i.d.\ points.


%
%


Our first main result can be stated as follows. All the notation and basic notions (including the Newtonian capacity $\Cap(\Omega)$ and normalized equilibrium measure for a set $\Omega\subseteq \R^d$) will be precisely given in \Cref{sec:notation}. We only anticipate that the notation $W_{\Omega}^p(\mu)$ denotes the Wasserstein cost of order $p$ between the restriction of a measure $\mu$ on $\Omega$ towards the uniform measure on $\Omega$, with the same total mass $\mu(\Omega)$.

\begin{theorem}\label{thm:main-bm-iid}
Let $d \in \cur{ 3, 4}$ and $p \in (0, (d-2)/2)$, or $d \ge 5$ and $p>0$. Then, there exists a constant $\c(\cI, d,p) \in (0, \infty)$ such that the following holds. Let $\Omega\subseteq \R^d$ be a bounded connected domain with $C^2$ boundary (or $\Omega = Q$ be a cube) and let $(B^{i})_{i \ge 1}$ be independent Brownian motions on $\R^d$, with initial law given by the normalized equilibrium  measure on $\Omega$. Then, it holds
\begin{equation}\label{eq:limit-n-bm-equilibrium-initial-law}
\lim_{n \to \infty} \EE\sqa{ W_{\Omega}^p \bra{ \sum_{i=1}^n \int_0^\infty \delta_{B^i_s} ds }}/\bra{n/ \Cap(\Omega)}^{1-p/(d-2)} = \c(\cI, d, p) |\Omega|.
\end{equation}
\end{theorem}

The notation $\cI$ in $\c(\cI, d,p)$ stands for ``interlacement''. Although strictly speaking in the statement above there is no Brownian interlacement process, its existence is deduced by showing first (\Cref{thm:poi-inter}) that a similar limit holds for the Brownian interlacement, by exploiting its stronger self-similarity properties. 

Our second main result links the constant $\c(\cI ,d,p)$ with the occupation measure of a stationary  Brownian motion on the flat torus $\T^d = \R^d/ \mathbb{Z}^d$.

\begin{theorem}\label{thm:main-torus}
Let $d \in \cur{3,4}$ and $p \in (0, (d-2)/2)$, or $d \ge 5$ and $p>0$. Let $(B_t)_{t \ge 0}$ be a stationary Brownian motion on $\T^d$, i.e.\ with uniform initial law. Then,  
\begin{equation}\label{eq:main-torus}
\limsup_{T \to \infty} \EE\sqa{  W^p_{\T^d}\bra{ \int_0^T \delta_{B_t} dt}}/  T^{1-p/(d-2)} \le \c(\cI, d,p).
\end{equation}
with $\c\bra{\cI, d,p}$ as in \Cref{thm:main-bm-iid}.
\end{theorem}

Although both our main results investigate the expected value of the Wasserstein cost of order $p$, we are able to give a concentration result in \Cref{prop:concentration} which can be used to improve to a.s.\ convergence for a certain range of $p$'s. This could be relevant in view of applications, e.g.\ in Monte Carlo methods, where one usually simulates a single sample path instead of averaging over an independent family of paths. 

\subsection{Comments on the proof technique} As already mentioned, our results adapt, for the first time in the study of occupation measures for diffusion processes, the combination of geometric and analytic techniques developed for the assignment problem in the recent works \cite{goldman2021convergence, ambrosio2022quadratic, goldman2022optimal}. In particular, in order to establish \Cref{thm:main-bm-iid}, we actually prove a rather general result, \Cref{theo:domain},  that applies to stationary random measures under suitable concentration assumptions, and we believe could be of interest also in other settings. However, since the proof is indeed a generalization of the arguments already employed for the assignment problem, we defer it to \Cref{app:asymptotics}. 

The main technical novelty instead, we believe, comes from its application to prove \Cref{thm:main-torus}. Indeed, we are able to argue that,  as $T$ increases, on a sufficiently small scale, but larger than the critical scale $T^{-1/(d-2)}$, the ``local'' Wasserstein cost of order $p$  converges exactly (after a suitable renormalization) towards the same constant $\c(\cI, d,p)$ of the Brownian interlacement. This is the content of \Cref{prop:local-torus}, which constitutes the main step in the proof of \Cref{thm:main-torus}. However, when transferring this exact limit from the local problems  to that on the whole $\T^d$, we obtain only an inequality, due to the sub-additivity properties of the transportation cost.

In the proof of \Cref{prop:local-torus}, which is split into small steps for the reader's convenience, we employ several technical tools, not only from the random assignment problem literature, but also from the random interlacement theory \cite{sznitman2013scaling, drewitz2014introduction}. In particular, we also make use of quantitative mixing rates in the total variation distance for Brownian motion on $\T^d$, in order to efficiently split a single trajectory into a large number of independent ones, which will constitute an approximation of Brownian interlacements. However, it turns out that we also need some fine asymptotics for the hitting distribution of a small ball (in addition to the hitting probability). This is achieved in \Cref{prop:hitting-times}, via a combination of the functional analytic methods for estimating the Wasserstein cost, with classical potential-theoretic arguments.

\subsection{Further questions and conjectures}

Our results raise some further natural questions:
\begin{enumerate}
\item Is the $\limsup$ actually a limit  (and equals the right hand side) in \eqref{eq:main-torus}? When compared with the assignment problem, the situation looks very similar, for sharp (conjectured)  constants associated to upper and lower limits are known to exist, see e.g.\ \cite{BaBo, goldman2022optimal}, but in general it is not known whether they coincide. A few exceptions are given by the case $d=p=2$ \cite{ambrosio2022quadratic}, and the concave case $d=1$, $0<p<1/2$ \cite{goldman2023concave}, but the proofs rely on rather special properties of the assignment in such cases.
%
 
\item Does \Cref{thm:main-torus} extend to the case of more general diffusion processes? Let us consider for example the setting of a weighted Riemannian manifolds $(M, g, \sigma)$, where $\sigma$ is a (smooth, bounded from above and below) probability density with respect to the Riemannian volume measure, and $B$ denotes the diffusion generated by the corresponding weighted Laplacian (so that the invariant measure is $\sigma$). Then, we conjecture that the right hand side in \eqref{eq:main-torus} should be  multiplied by an additional integral term:
\begin{equation}
\limsup_{T \to \infty} \EE\sqa{  W^p_{\T^d}\bra{ \int_0^T \delta_{B_t} dt}}/  T^{1-p/(d-2)} \le \c(\cI, d,p) \int_{M} f^{1-p/(d-2)},
\end{equation}
where
\begin{equation}
f(x) = \lim_{\eps \to 0} \frac{ \Cap_M\bra{ B_\eps(x) } }{\Cap_{\R^d} (B_\eps(0))}.
\end{equation}
 It would be already interesting to prove (or disprove) this conjecture in the case of $M = \R^d$ endowed with a Gaussian weight $\sigma$, so that the Ornstein-Uhlenbeck process is the resulting diffusion, and $p$ is small enough, e.g.\ $0<p<d-2$.

\item How should our results be modified in the cases $d \le 2$, or $d \in \cur{3,4}$ but $p \ge (d-2)/2$? Let us mention that, for a stationary Brownian motion on $\T^d$ it is proved in \cite{huesmann2022wasserstein} that, for any $p \ge 1$,
\begin{equation}\label{eq:mattesini}
\EE\sqa{  W^p_{\T^d}\bra{ \int_0^T \delta_{B_t} dt}} \sim \begin{cases} T \cdot T^{-p/2}& \text{if $d  \le 3$,}\\
T^{1-p/2} \cdot \bra{\log T }^{p/2} & \text{if  $d=4$,}\\
T^{1-p/(d-2)} & \text{if $d \ge 5$,}
\end{cases}
\end{equation}
where the notation $f(T)\sim g(T)$ means that for every $T$ (large enough), it holds $c^{-1} \le f(T)/g(T) \le c$ for some constant $c \in (0, \infty)$. For the concave case $0<p<1$, we conjecture, by analogy with the assignment problem \cite{bobkov2020transport, goldman2023concave},  that
\begin{equation}
\EE\sqa{  W^p_{\T^d}\bra{ \int_0^T \delta_{B_t} dt}}  \sim T \cdot T^{-p/2}\quad \text{if $d \le 2$ or $d=3$ and $p> 1/2$.}
\end{equation}
In the  case $d=3$, $p=1/2$, we conjecture instead
\begin{equation}
\EE\sqa{  W^{1/2}_{\T^3}\bra{ \int_0^T \delta_{B_t} dt}}  \sim \sqrt{ T \log T}.
\end{equation}

\item In \Cref{prop:concentration}, we show that the random variable 
\begin{equation}
W^p_{\T^d}\bra{ \int_0^T \delta_{B_t} dt}/T^{1-p/(d-2)} 
\end{equation}
concentrates around its expectation if $p$ is sufficiently small (in particular, if $p<(d-2)/3$). It is natural to conjecture that such concentration holds for every $p>0$ (if $d\ge 5$). Let us point out that a similar issue arises for the assignment problem, where the corresponding concentration results are known to hold for  $0<p<d$ (if $d\ge 3$) but one expects that it should be true for every $p$.
\end{enumerate}

\subsection{Structure of the paper} In \Cref{sec:notation}, we introduce all the relevant notation and recall some basic facts, in particular concerning optimal transportation theory. In \Cref{sec:bm} we focus on Brownian motion on $\R^d$ and the torus $\T^d$, establishing, among various useful bounds, \Cref{prop:hitting-times} about small-ball hitting probabilities and hitting distributions. \Cref{sec:bi} is devoted to the proof of \Cref{thm:main-bm-iid} although, as already mentioned, most of the general tools from the assignment problem literature are actually recalled and elaborated in \Cref{app:asymptotics}, \Cref{app:depoisson} and \Cref{sec:upper}.  The main body of the article ends with  \Cref{sec:torus}, devoted to the proof of \Cref{thm:main-torus} and  \Cref{sec:concentration}, focusing on the concentration properties of the random Wasserstein cost and (non-sharp) lower bounds.

\subsection*{Acknowledgements} D.T.\ thanks the MIUR Excellence Department Project awarded to the Department of Mathematics, University of Pisa, CUP I57G22000700001,  the HPC Italian National Centre for HPC, Big Data and Quantum Computing - Proposal code CN1 CN00000013, CUP I53C22000690001, the PRIN 2022 Italian grant 2022WHZ5XH - ``understanding the LEarning process of QUantum Neural networks (LeQun)'', CUP J53D23003890006, and the INdAM-GNAMPA project 2023 ``Teoremi Limite per Dinamiche di Discesa Gradiente Stocastica: Convergenza e Generalizzazione''. Research also partly funded by PNRR - M4C2 - Investimento 1.3, Partenariato Esteso PE00000013 - ``FAIR - Future Artificial Intelligence Research" - Spoke 1 ``Human-centered AI'', funded by the European Commission under the NextGeneration EU programme.

D.T.\ thanks A.\ Chiarini for useful conversations on the theory of random interlacements and J.-X.\ Zhu for noticing a missing factor in \Cref{prop:concentration} which lead to a different range of exponents. M.\ Mariani acknowledges support by  the 2023 Visiting Fellow program of University of Pisa. Both authors thank an anonymous referee for pointing out a gap in the proof of \Cref{thm:main-torus} in an earlier version of the paper, which led to a revised version that includes \Cref{sec:upper}.

\section{Notation and Basic facts}\label{sec:notation}

\subsection{General notation} Given two real-valued functions $f$, $g$ (e.g.\ of a real variable $x$), we write $f(x) \les g(x)$ if there exists a strictly positive and finite constant $c$ (independent of $x$) such that $f(x) \le c g(x)$ for every $x$ (notice that we do not require $f$, $g$ to be positive) and $f(x) \gtrsim g(x)$ if $g(x) \les f(x)$. We also write $f(x) \sim g(x)$ if $f(x) \les g(x)$ and $g(x) \les f(x)$. The notation $f(x) \ll g(x)$ means that $\lim_{x} f(x)/g(x) = 0$, for $x \to \infty$ or $x \to 0$ (it will be clear from the context). Since we are going to use functions with several parameters, such as exponents $p$, $q$ or additional positive parameters $\eps$, $\delta$ etc., we may stress the dependence of the constant $c$ with respect to these parameters by adding them as subscripts, e.g.\ $f_p(x)\les_p g_p(x)$ means that the constant $c = c(p)$ may depend on $p$. Despite this, we warn the reader that in order to keep the notation as light as possible, we will not stress such dependence in many statements. Again, to keep the statements simple, we do not always stress the fact that the various constants appearing  are strictly positive and finite.

 Given  a set $S$, we write $\sharp S$ for the number of its elements and $\chi_S$ for its indicator function.  In case $S \in \mathcal{A}$ is an event on a probability space $(\Omega, \mathcal{A}, \mathbb{P})$, we use the notation $I_A = \chi_A$. For $L>0$, we let $Q_L=[-L/2,L/2]^d$  for the cube of side length $L$ and $D_L(y)= \cur{x\, : \, \abs{x-y} \le L} \subseteq \R^d$ for the ball of radius $L$ centred at $y$, where we denote by $|x|$ the  Euclidean norm of a vector $x\in \R^d$. We consider $\T^d = \R^d/\mathbb{Z}^d$ endowed with the flat distance
\begin{equation}
\dist_{\T^d}(x,y) = \inf_{z \in \mathbb{Z}^d} \abs{ x-y-z},
\end{equation}
 and identify it (as a set) with $Q_1 = [-1/2, 1/2]^d \subseteq \R^d$. We use the same notation
 \begin{equation}
  D_L(y) = \cur{x\, : \, \dist_{\T^d}(x,y) \le L},
 \end{equation}
 so that the ball $D_L(0)$ on $\T^d$ is naturally identified (as a set) with the ball in $\R^d$. 

We use the notation $|A|$ for the Lebesgue measure of a Borel set $A \subseteq \R^d$ or $A \subseteq \T^d$, $\dist(\cdot, A)$ for the distance function from $A$, $\diam(A)$ for its diameter, and $\int_A f =\int_A f(x) dx$ for the Lebesgue integral of a function $f$ on $A$.  Given a Borel measure $\mu$ on $\R^d$ or $\T^d$, we write $\tras_x \mu$ for its translation by $x \in \R^d$, i.e., $\tras_x \mu (A) = \mu(A -x)$. In case of $\R^d$, we write $\dil_{\rho} \mu$ for the dilation by a factor $\rho \in (0, \infty)$, i.e., $\dil_\rho \mu (A) = \mu( A/\rho)$.

For a (sufficiently smooth) function $\phi$ on $\R^d$ or $\T^d$, we use the notation $\nabla \phi$ for its gradient, $\nabla\cdot \phi$ for the divergence, $\Delta \phi$ for the Laplacian.
 
\subsection{Occupation measures and hitting times} Given a continuous curve $x = (x_t)_{t \ge 0}$ with values in $\R^d$ or $\T^d$, we write
 \begin{equation}
 \mu_T^x = \int_0^T \delta_{x_t} dt
 \end{equation}
 for its occupation measure up to time $T \in [0, \infty]$. For a (compact) set $K \subseteq \R^d$, define its first hitting time
\begin{equation}
 \tau_K x := \inf \cur{t > 0\, : x_t \in K}.
\end{equation}
When $x$ is understood, we simply write $\tau_K$, in particular we use often the notation $\mu_{\tau_K}^x = \mu_{\tau_K x}^x$. We also write $\theta = (\theta_s)_{s \ge 0}$ for the shift operator acting on continuous curves as $\theta_s x = (x_{t+s})_{t \ge 0}$ for every $s \ge 0$. We apply $\theta$ also when $s = s(x)$, i.e., $\theta_{s(x)}x = (x_{t+s(x)})_{ t\ge 0}$. This notation is particularly useful in order to define the first hitting time of a set $K'$ after $\tau_K$, which is simply given by $\tau_K +\tau_{K'} \circ \theta_{\tau_K} $. We notice the identity, valid for any $x$ and $K$,
\begin{equation}
\label{eq:restriction-occupation-measure}
\mu_T^x\restr{K} =  \mu_{T-\tau_K}^{\theta_{\tau_K}x} \restr{K},
\end{equation}
with $\tau_K = \tau_K x$ (valid also for $T = \infty$ with the convention that $\infty - \infty = 0$).

\subsection{Negative Sobolev norms}
Given a bounded domain $\Omega \subseteq \R^d$ (or $\Omega \subseteq \T^d$) with Lipschitz boundary and $p \in (1,\infty)$, with H\"older conjugate $p' = p/(p-1)$, we write $\| f \|_{L^p(\Omega)}$ for the Lebesgue norm of $f$, and 
\begin{equation}
 \|f\|_{W^{-1,p}(\Omega)}=\sup_{\nor{\nabla \phi}_{L^{p'}(\Omega)}\le 1} \int_{\Omega} f \phi
\end{equation}
for the negative Sobolev norm. Notice in particular that if $\|f\|_{W^{-1,p}(\Omega)}<\infty$ then $\int_\Omega f=0$. In this case we may and will restrict the supremum to functions $\phi$ having also average zero. 

Let us recall that we can bound the $W^{-1,p}$ norm by the $L^p$ norm via the following Poincaré inequality.

\begin{lemma}\label{lem:poincare}
 Let $\Omega$ be a bounded domain with Lipschitz boundary and let $f:\Omega\to \R$ such that $\int_\Omega f=0$. Then, for every $p>1$,
 \begin{equation}\label{eq:Lp}
  \|f\|_{W^{-1,p}(\Omega)}\les |\Omega|^{\frac{1}{d}} \|f\|_{L^p(\Omega)}.
 \end{equation}
Moreover, the implicit constant is invariant by dilations of $\Omega$.
\end{lemma}

 \subsection{Total variation distance}

Given two random variables $X$, $Y$ taking values in a measurable space $(E, \mathcal{E})$ (possibly defined on different spaces) with laws $\mathbb{P}_X$, $\mathbb{P}_Y$, their total variation distance is
\begin{equation}
 \TV(\mathbb{P}_X,\mathbb{P}_Y) = \sup_{A \in \mathcal{E}} \mathbb{P}(X \in A) - \PP\bra{Y \in A} \in [0,1]
\end{equation}
Recall that, for any $f: E \to \R$ measurable and bounded, it holds
\begin{equation}
 \abs{ \EE\sqa{f(X)} - \EE\sqa{f(Y)} }\le \sup_{x \in E}\abs{f(x)} \TV(\mathbb{P}_X,\PP_Y).
\end{equation}
We use the following well-known contraction property for the total variation distance: if $N$ is a Markov kernel from $E$ to a measurable space $(F, \mathcal{F})$, then
\begin{equation}\label{eq:contraction-tv}
\TV( N_\sharp \PP_X, N_\sharp \PP_Y ) \le \TV(\mathbb{P}_X,\PP_Y),
\end{equation}
where $N_\sharp \PP_X$ denotes the (push-forward) probability given by $N_\sharp \PP_X(A) = \EE\sqa{ N(X, A)}$ for $A \in \mathcal{F}$.

In particular, we will use the following \emph{data-processing} inequality (for the total variation distance): if $(U, V, W)$ are random variables that form a Markov chain
\begin{equation}
U \rightarrow V \rightarrow W
\end{equation}
(possibly taking values in different measurable spaces $(E, \mathcal{E})$, $(F, \mathcal{F})$, $(G, \mathcal{G})$ with transition kernels $N_{F\to E}$, $N_{F \to G}$ and $\tilde{V}$ is any random variable taking values in $F$, then
\begin{equation}\label{eq:markov-chain-tv}
 \TV\bra{  \mathbb{P}_{(U,V,W)}, N_{F \to E} \otimes \mathbb{P}_{\tilde V} \otimes N_{F \to G} } \le \TV(V, \tilde{V}),
\end{equation}
where $N_{F \to E} \otimes \mathbb{P}_{\tilde V} \otimes N_{F \to G} $ is a notation for the probability measure corresponding to a Markov chain 
\begin{equation}
\tilde{U} \rightarrow \tilde{V} \rightarrow \tilde{W}
\end{equation}
with $\tilde{U}$ having the kernel $N_{F \to E}$ as conditional law given $\tilde{V}$, and $\tilde{W}$ having kernel $N_{F \to G}$ as conditional law given $\tilde{V}$. 

\subsection{Optimal transport}


Given $p >0$, a Borel subset $\Omega \subseteq \R^d$  (or $\Omega \subseteq \T^d$) and two positive Borel measures $\mu$, $\lambda$ with $\mu(\Omega) = \lambda(\Omega) \in (0, \infty)$ and finite $p$-th moments, the optimal transportation (or Wasserstein) cost of order $p$ between $\mu$ and $\lambda$ is defined as the quantity
\begin{equation} W_{\Omega}^p(\mu, \lambda) =  \min_{\pi\in\mathcal{C}(\mu,\lambda)} \int_{\R^d\times\R^d} \dist(x,y)^p d \pi(x,y)
\end{equation}
where $\dist$ denotes the Euclidean distance in case $\Omega \subseteq \R^d$, and the flat distance in the case $\Omega \subseteq \T^d$. The set $\mathcal{C}(\mu, \lambda)$ is the set of couplings between $\mu$ and $\lambda$ i.e., positive Borel measures on $\Omega\times \Omega$ such that the first and second marginals are respectively $\mu$ and $\lambda$. 
Moreover, if $\mu(\Omega) = \lambda(\Omega) = 0$, we define $W^p_{\Omega}(\mu, \lambda) = 0$, while if $\mu(\Omega) \neq \lambda(\Omega)$, we let $W_{\Omega}^p(\mu, \lambda) = \infty$.  
Let us notice that if, $\Omega \subseteq \T^d$ and $\diam(\Omega) < 1/2$, the Wasserstein cost can be equivalently computed on $\T^d$ or on $\R^d$, by identify $\Omega$ with a subset of $[0,1]^d$.

Next, we collect some basic facts related to the Wasserstein cost. Proofs of the simpler ones can be found in any of the standard references on the subject \cite{AGS,VilOandN, peyre2019computational}.

For every constant $a>0$,
\begin{equation}
W_{\Omega}^p( a \mu, a \lambda) = a W_{\Omega}^p(\mu, \lambda),
\end{equation}
which always allows to reduce to the case of probability measures $\mu(\Omega) = \lambda(\Omega)=1$. For $r \ge 1$, by H\"older inequality,
\begin{equation}\label{eq:higer-r-wass}
 W_\Omega^p(\mu, \lambda) \le \mu(\Omega)^{1-1/r} \bra{ W_\Omega^{pr} (\mu, \lambda)}^{1/r}.
\end{equation}

If $\mu(\Omega) = \lambda(\Omega)$,
\begin{equation}\label{eq:trivial-wass}
W_{\Omega}^p(\mu, \lambda) \le \diam(\Omega)^p \mu(\Omega),
\end{equation}
which can be slightly improved to
\begin{equation}\label{eq:less-trivial-wass}
W^p_\Omega(\mu, \lambda) \le \diam(\Omega)^p \TV (\mu, \lambda).
\end{equation}
Notice also the lower bound
\begin{equation}
\label{eq:trivial-lower-bound}
W_{\Omega}^p(\mu, \lambda) \ge \int_{\Omega} \dist(x, \operatorname{\supp}(\lambda) )^p d\mu(x),
\end{equation}
where  $\dist(x, \supp(\lambda))$  denotes the distance between $x$ and the support of the measure $\lambda$.
For $p \ge 1$, the quantity $( W_\Omega^p(\mu, \lambda))^{1/p}$ is a distance,  while for $p \in (0,1)$, $W_{\Omega}^p(\mu, \lambda)$ is already a distance, hence enjoying the triangle inequality.  We will also  use the following sub-additivity inequality 
\begin{equation}\label{eq:sub} W_\Omega^p\bra{ \int_E \mu_z d \sigma(z), \int_{E} \lambda_z d \sigma(z)} \le \int_{E} W^p_\Omega(\mu_z, \lambda_z) d \sigma(z)\end{equation}
where $(E, \mathcal{E}, \sigma)$ is a measure space and $(\mu_z)_{z \in E}$, $(\lambda_z)_{z\in E}$ are measurable families. Notice that, for the right hand side to be finite, it is necessary that $\mu_z(\Omega) = \lambda_z(\Omega)$ for $\sigma$-a.e.\ $z \in E$. 

To keep the notation simple, given a Borel subset $\tilde{\Omega} \subseteq \Omega$, we write
\begin{equation}
W_{\tilde{\Omega}}^p(\mu, \lambda) = W_{\tilde{\Omega}}^p(\mu \restr \tilde{\Omega}, \lambda \restr \tilde{\Omega}),
\end{equation}
which also coincides with $W_{\Omega}^p( \chi_{\tilde{\Omega}} \mu, \chi_{\tilde{\Omega}} \lambda )$.
Moreover, if a measure is absolutely continuous with respect to Lebesgue measure, we only write its density. For example,
\begin{equation}
 W^p_{\Omega} \bra{ \mu,  \mu(\Omega)/|\Omega| }
\end{equation}
denotes the transportation cost between $\mu \restr \Omega$ to the uniform measure on $\Omega$ with total mass $\mu(\Omega)$. To further simplify the notation, in this special case of transporting towards the uniform measure, we write
\begin{equation}
W^p_{\Omega} \bra{ \mu} = W^p_{\Omega} \bra{ \mu,  \mu(\Omega)/|\Omega| }.
\end{equation}

%
%

Let us recall some lemmas from \cite{goldman2022optimal}. The first one is a ``geometric'' sub-additivity result, which follows straightforwardly combining the triangle inequality, \eqref{eq:sub} and the elementary inequality stating that, for $p>0$, there exists $c = c(p)$ such that
\begin{equation}\label{eq:elementary}
 (a+b)^p\le (1+\eps) a^p+ \frac{c}{\eps^{(p-1)^+}} b^p \qquad \forall a,b\ge 0 \textrm{ and } \eps\in (0,1),
\end{equation}
where we use the notation $z^+ = \max\cur{0, z}$.

\begin{lemma}\label{lem:sub}
Let $(E, \mathcal{E}, \sigma)$ be a measure space, let  $\Omega\subset \R^d$ (or $\Omega \subseteq \T^d$)  and $(\Omega_z)_{z \in E}$ be a measurable collection of Borel subsets of $\Omega$. Let $\mu$, $\lambda$ be measures on $\Omega$ such that
 \begin{equation}
 \int_E \chi_{\Omega_z} d \sigma(z)  = 1, \quad \text{$\mu+\lambda$-a.e.}
 \end{equation}
 Then, for every  $p > 0$, there exists $c = c(p)$ such that, for every $\eps\in [0,1)$, 
 \begin{equation}\label{eq:mainsub-convex}
 \begin{split}
  W^p_{\Omega}\lt(\mu, \frac{\mu(\Omega)}{\lambda(\Omega)}\lambda\rt) & \le (1+\eps)\int_E W^p_{\Omega_z}\lt(\mu,\frac{\mu(\Omega_z)}{\lambda(\Omega_z)}\lambda\rt)d \sigma(z) \\
  & \qquad + \frac{c}{\eps^{(p-1)^+}}W^p_\Omega\lt(\int_E \frac{\mu(\Omega_z)}{\lambda(\Omega_z)}\chi_{\Omega_z} d \sigma(z)\lambda ,\frac{\mu(\Omega)}{\lambda(\Omega)}\lambda\rt).
  \end{split}
 \end{equation}
\end{lemma}

Letting in particular $\lambda$ be the Lebesgue measure, we have, for $p>0$,
 \begin{equation}\label{eq:mainsub-lebesgue}
  W^p_{\Omega}\lt(\mu\rt)\le (1+\eps) \int_{E} W^p_{\Omega_z}\lt(\mu\rt)d \sigma(z)+ \frac{C}{\eps^{(p-1)^+}}W^p_\Omega\lt( \int_E \frac{\mu(\Omega_z)}{|\Omega_z|}\chi_{\Omega_z} d \sigma(z)\rt).
 \end{equation}

The second lemma relies on a PDE argument and applies to sufficiently nice domains.

 \begin{lemma}\label{lem:peyre} Assume that $\Omega \subseteq \R^d$  (or $\Omega \subseteq \T^d$) is  bounded, connected and with  Lipschitz boundary. If $\mu$ and $\lambda$ are measures on $\Omega$ with $\mu(\Omega) = \lambda(\Omega)$,  absolutely continuous with respect to the  Lebesgue measure and $\inf_\Omega \lambda>0$, then, for every $p> 1$,
 \begin{equation}\label{eq:estimCZ}
  W_{\Omega}^p(\mu,\lambda)\les \frac{1}{\inf_{\Omega} \lambda^{p-1}}\nor{ \mu - \lambda}_{W^{-1,p}(\Omega)}^p.
 \end{equation}
 \end{lemma}
In the lemma above, $\inf_{\Omega} \lambda$ is the (essential) infimum of the density of $\lambda$ with respect to Lebesgue measure.

Combining the above result with  \Cref{lem:poincare}, we obtain  that, for $p > 1$,
\begin{equation}\label{eq:estimCZ-Lp}
  W_{\Omega}^p(\mu,\lambda)\les \frac{|\Omega|^{p/d}}{\inf_{\Omega} \lambda^{p-1}}\nor{ \mu - \lambda}_{L^p(\Omega)}^p,
 \end{equation}
 where the implicit constant  is invariant with respect to dilations of $\Omega$.
 
 \begin{remark}[The case of a manifold]\label{rem:peyre-sphere}
 In fact, \Cref{lem:peyre} holds as well in case of compact smooth connected Riemannian manifolds, see \cite{AGT19}. In particular, given the unit sphere $\Omega = \partial  D_1 \subseteq \R^d$ and letting $\lambda= \tilde{e}_{D_1}$ be the uniform (probability) distribution on the sphere (the notation will be clarified below), it holds
 \begin{equation}
  W_{\partial D_1}^p(\mu,\tilde{e}_{D_1})\les \sup \cur{ \int_{\partial D_1} \phi d \mu\, : \, \int_{\partial D_1} \phi = 0, \quad \int_{\partial D_1} |\nabla \phi|^{p'} \le 1  }.
 \end{equation}
 A simple scaling argument shows also that, for $\ell>0$,
 \begin{equation}  W_{\partial D_\ell}^p\bra{\mu, \tilde{e}_{D_\ell} }\les \ell^p \sup \cur{ \int_{\partial D_\ell} \phi(x/\ell) d \mu(x)\, : \, \int_{\partial D_1} \phi  = 0, \quad \int_{\partial D_1} |\nabla \phi|^{p'} \le 1  }.
 \end{equation}
 where $\tilde{e}_{D_\ell}$ denotes the uniform probability distribution on $\partial D_\ell$.
 \end{remark}

 \begin{remark}[Duality and the concave case]
  The case $p=1$ of \Cref{lem:peyre} is particularly simple via Kantorovich duality. Considering for simplicity the case of $\T^d$ only, the dual formulation reads
\begin{equation}
 W^1_{\mathbb{T}^d}\bra{ \mu, \lambda } = \sup\cur{ \int_{\mathbb{T}^d} f d (\mu - \lambda)  \, : \, |f(x)-f(y)| \le \dist_{\mathbb{T}^d}(x,y) },
\end{equation}
i.e., $f$ is Lipschitz continuous with $\Lip(f) \le 1$. Then, writing $\mu - \lambda = \Delta \Delta^{-1} (\mu - \lambda)$, we find after an integration by parts
\begin{equation} \begin{split}
 W^1_{\mathbb{T}^d}\bra{ \mu, \lambda } & = \sup\cur{ \int_{\mathbb{T}^d} \nabla f \bra{  \nabla \Delta^{-1} (\mu - \lambda)}  \, : \,  \Lip(f) \le 1}\\
 & \le \nor{   \nabla \Delta^{-1} (\mu - \lambda) }_{L^1},
 \end{split}
\end{equation}
having used that $|\nabla f| \le 1$ a.e.\ if $\Lip(f) \le 1$. This argument suggests also a similar -- but slightly weaker -- inequality in the case $0<p<1$. Indeed, duality reads
\begin{equation}
 W^p_{\mathbb{T}^d}\bra{ \mu, \lambda } = \sup\cur{ \int_{\T^d} f d (\mu - \lambda)  \, : \,  |f(x)-f(y)| \le \dist_{\T^d}(x,y)^p },
\end{equation}
i.e., $f$ is H\"older continuous with exponent $p$ and constant $[f]_{C^p} \le 1$. We then recall that, for every $q<p$, one has the inequality
\begin{equation}
 \nor{ \Delta^{q/2} f }_{L^2 } \les [f]_{C^p},
\end{equation}
where the norm of the fractional Laplacian is defined as a Fourier multiplier or equivalently as the norm of a fractional Sobolev space (see e.g. \cite{Hitchhiker}) and the implicit constant depends on $p$, $q$, $d$ only. Writing $\mu - \lambda = \Delta^{q/2} \Delta^{-q/2}(\mu- \lambda)$ we obtain the upper bound
\begin{equation}\label{eq:peyre-concave}
 W^p_{\T^d}\bra{ \nu, \lambda} \les_{p,d,q} \nor{ \Delta^{-q/2} (\mu - \lambda) }_{L^2 }.
\end{equation}
 \end{remark}

Let us notice that, by \eqref{eq:sub}, it always holds
\begin{equation}\label{eq:eq-to-be-reversed}
W_{\Omega}^p\bra{ \lambda+ \mu, \tilde{\lambda}+\mu } \le  W_{\Omega}^p\bra{ \lambda, \tilde{\lambda}} + W_{\Omega}^p(\mu,\mu) = W_{\Omega}^p\bra{ \lambda, \tilde{\lambda}}.
\end{equation}
It is natural to ask whether under suitable smoothness assumptions on $\mu$ one can improve this bound.  The following result is a special case of \cite[Proposition 2.9]{goldman2022optimal} for the case where $\mu$ and $\tilde{\lambda}$ have  constant densities. 

\begin{proposition}\label{prop:density-helps}
Let $\Omega \subseteq \R^d$ be a bounded Lipschitz domain, $\lambda$ be any finite measure on $\Omega$ and $h > 0$. Then, for every $p>d/(d-1)$, it holds
\begin{equation}\label{claim} W_\Omega^p\bra{ \lambda + \frac{h}{|\Omega|}} \les  \diam(\Omega)^p \lambda(\Omega) \bra{ \frac{ \lambda(\Omega)}{h}}^{p/d},  
\end{equation}
where the implicit constant depends on $\Omega$, but is invariant by rescaling of $\Omega$.
\end{proposition}

The condition $p>d/(d-1)$ appears in \eqref{prop:density-helps} because of an application of the Sobolev embedding in $\Omega$. If $p\le d/(d-1)$ one can deduce suitable modifications of \eqref{claim}, as discussed in \cite[Remark 2.10]{goldman2022optimal}. However, we may dispense of these cases by a simple application of \eqref{eq:higer-r-wass}: if $0<p\le d/(d-1)$ and $r$ is chosen sufficiently large so that $pr>d/(d-1)$, then
\begin{equation}\begin{split}\label{eq:wasserstein-h-p-r}
W_\Omega^p\bra{ \lambda + \frac{h}{|\Omega|}} & \les \bra{ \lambda(\Omega) + h}^{1-1/r} \diam(\Omega)^p \frac{ \lambda(\Omega)^{1/r+p/d}}{h^{p/d}}\\
& \les \diam(\Omega)^p \lambda(\Omega) \bra{ \frac{ \lambda(\Omega)}{h}}^{p/d} + \diam(\Omega)^p h \bra{ \frac{\lambda(\Omega)}{h}}^{1/r+p/d}.
\end{split}
\end{equation}

It is also natural to ask whether \eqref{eq:sub} can be reversed. Indeed, given two measures $\mu$, $\lambda$, and setting
\begin{equation}
 u := \frac{\lambda(\Omega)}{\mu(\Omega)+\lambda(\Omega)},
 \end{equation} 
then applying the triangle inequality, \eqref{eq:elementary}, \eqref{eq:sub} and \eqref{eq:trivial-wass} easily yield
\begin{equation}\begin{split}\label{eq:same-asymptotics-trivial}
W_{\Omega}^p(\mu) &\le  (1+\eps) W_\Omega^p\bra{(1-u)(\mu+ \lambda)} + \frac{c}{\eps^{(p-1)^+}} W_{\Omega}^p\bra{\mu, (1-u)(\mu + \lambda) } \\
& \le (1+\eps) W_\Omega^p\bra{\mu+ \lambda}  + \frac{c}{\eps^{(p-1)^+}} \diam(\Omega)^p \lambda(\Omega). 
\end{split} 
\end{equation}

Combining  this derivation with \Cref{prop:density-helps} leads to the following  bounds.

\begin{lemma}\label{lem:same-asymptotics}
Let $\Omega \subseteq \R^d$ be a bounded Lipschitz domain, let $p >d/(d-1)$ and $\eps \in (0,1)$. Then there exists $c = c(\Omega, d, p,\eps) $ such that the following holds. If $\mu$ and $\lambda$ are measures on $\Omega$, then
\begin{equation}
 W_{\Omega}^p(\mu+\lambda) \le (1+\eps) W_{\Omega}^p(\mu)  + c\diam(\Omega)^p \lambda(\Omega)  \bra{ \frac{ \lambda(\Omega)}{\mu(\Omega) }}^{p/d},
\end{equation}
and
\begin{equation}
W_{\Omega}^p(\mu) \le (1+c\delta)(1+\eps) W_{\Omega}^p(\mu+\lambda) + c\diam(\Omega)^p \lambda(\Omega)  \bra{ \frac{ \lambda(\Omega)}{\delta \mu(\Omega) }}^{p/d} 
\end{equation}
provided that $\delta$ satisfies
\begin{equation}
 \lambda(\Omega)/\mu(\Omega) \le \delta \le 1/c.
 \end{equation}
Moreover, the constant $c$  is invariant by rescaling of $\Omega$.
%
\end{lemma}
\begin{proof}
The first inequality is straightforward from the triangle inequality and \eqref{eq:elementary},
\begin{equation}
W_{\Omega}^p(\mu+\lambda) \le (1+\eps)W_{\Omega}^p(\mu+\lambda,\mu(\Omega)/|\Omega| + \lambda) + c(p,\eps) W_{\Omega}^p(\mu(\Omega)/|\Omega| + \lambda),
\end{equation}
 and an application of  \Cref{prop:density-helps}. Let us focus on the second one. Without loss of generality, we can assume that  $c \ge 2$, hence $\delta\le 1/2$. Write
\begin{equation}
u := \frac{\lambda(\Omega)}{\mu(\Omega)+\lambda(\Omega)} \le \frac{\lambda(\Omega)}{\mu(\Omega)} \le \delta.
\end{equation}
%
By the triangle inequality and \eqref{eq:elementary}, we have, for $c = c(p,\eps)$,
\begin{equation}
 W_{\Omega}^p(\mu)  \le (1+\eps)W_\Omega^p( (1-u) (\mu+\lambda) ) + c W_{\Omega}^p\bra{ \mu, (1-u)(\mu+\lambda)}.
 \end{equation}
 Since 
 \begin{equation}
  W_\Omega^p( (1-u) (\mu+\lambda) ) = (1-u)W_\Omega^p( \mu+\lambda)\le W_\Omega^p( \mu+\lambda),
 \end{equation}
 we focus on the second term. By \eqref{eq:sub} and repeated applications of the triangle inequality and \eqref{eq:elementary} with $\eps=1/2$, we find
 \begin{equation}
  \begin{split}
  W_{\Omega}^p\bra{ \mu, (1-u)(\mu+\lambda)} & \le W_{\Omega}^p((1-2\delta) \mu, (1-2\delta) \mu) + W_{\Omega}^p\bra{  2\delta \mu, (2\delta-u)\mu + (1-u) \lambda } \\
  & \le  W_{\Omega}^p\bra{  2\delta\mu, (2\delta-u)\mu + (1-u) \lambda }\\
  & \les W_{\Omega}^p\bra{  2\delta\mu, (2\delta-u)\mu(\Omega)/|\Omega|+ (1-u) \lambda }\\
  & \quad  +  W_{\Omega}^p\bra{  (2\delta-u)\mu(\Omega)/|\Omega|+ (1-u) \lambda, (2\delta-u)\mu + (1-u) \lambda }\\
  & \les  W_{\Omega}^p\bra{  2\delta \mu} +W_{\Omega}^p\bra{  (2\delta-u)\mu(\Omega)/|\Omega|+ (1-u) \lambda } \\
  & \quad  + W_{\Omega}^p\bra{ (2\delta-u)\mu },
  \end{split}
  \end{equation} 
where the implicit constants depend on $p$ only. By \Cref{prop:density-helps}, we obtain
\begin{equation}\begin{split}
W_{\Omega}^p\bra{  \frac{(2\delta-u)\mu(\Omega)}{(1-u)|\Omega|}+\lambda } & \les \diam(\Omega)^p \lambda(\Omega) \bra{ \frac{ \lambda(\Omega) (1-u)}{(2\delta - u) \mu(\Omega)}}^{p/d}\\
& \les  \diam(\Omega)^p \lambda(\Omega) \bra{ \frac{ \lambda(\Omega)}{\delta\mu(\Omega)}}^{p/d}.
\end{split}
\end{equation}
where the implicit constants depend here also on $\Omega$ (but are invariant with respect to rescaling). 
Collecting all the terms, for some constant $\bar c = \bar c(\Omega, d,p,\eps)$
  \begin{equation}
   W_{\Omega}^p(\mu)  \le (1+\eps)W_\Omega^p( \mu+\lambda ) + \bar c \delta W_{\Omega}^p\bra{ \mu} +  \bar c \diam(\Omega)^p \lambda(\Omega) \bra{ \frac{ \lambda(\Omega)}{\delta\mu(\Omega)}}^{p/d}.
  \end{equation}
  Subtracting both sides the term $c\delta W_{\Omega}^p\bra{ \mu}$ yields the thesis, by estimating $1/(1-\bar c\delta) \le 1+2 \bar c \delta$ provided that $\delta \bar c$ is sufficiently small, which can be easily achieved by choosing a final constant $c$ in the thesis that is sufficiently large.
%
\end{proof}
%

\begin{remark}\label{rem:density-helps}
If $0<p\le d/(d-1)$, we cannot apply directly \Cref{prop:density-helps}, but we may rather use \eqref{eq:wasserstein-h-p-r}, by choosing $r$ sufficiently large such that $pr>d/(d-1)$. This leads to modified bounds, respectively by adding the additional terms
\begin{equation}\label{eq:corrections-density-helps}
 \diam(\Omega)^p \mu(\Omega) \bra{ \frac{\lambda(\Omega)}{\mu(\Omega)}}^{1/r+p/d}  \quad \text{and} \quad \diam(\Omega)^p (\delta \mu(\Omega) ) \bra{ \frac{\lambda(\Omega)}{\delta \mu(\Omega)}}^{1/r+p/d} 
\end{equation}
in the right hand sides.
\end{remark}

%
%
%

\subsection{Concentration inequalities}



We recall the classical Rosenthal's inequalities for i.i.d.\ variables (see e.g.\ \cite{rosenthal1970subspaces, ibragimov2001best, osekowski2012note}): for $q\ge 2$ and centered variables
\begin{equation}\label{eq:rosenthal-centered}
\nor{  \sum_{i=1}^n Z_i}_{L^q} \les_q n^{1/q} \nor{Z_1}_{L^q} + n^{1/2}\nor{Z_1}_{L^2},
\end{equation}
while for $q \ge 1$ and  positive variables,
\begin{equation}\label{eq:rosenthal-positive}
\nor{  \sum_{i=1}^n Z_i}_{L^q} \les_q n^{1/q}  \nor{ Z_1}_{L^q} + n \nor{ Z_1 }_{L^1}.
\end{equation}
These in particular yield, for a binomial random variable $N$ with parameters $(n,q)$ and $\EE\sqa{N} = np \ge 1$, the bounds
\begin{equation}
 \nor{ N - np}_{L^q } \les_q (np)^{1/2}, \quad\text{and} \quad \nor{N}_{L^q }\les_q np.
\end{equation}
and in the limit $n \to \infty$, $p \to 0$ with $np \to \lambda \ge 1$, so that $N \to  N_\lambda$  a  Poisson variable with mean $\lambda$:
\begin{equation}
 \nor{ N_{\lambda} - \lambda}_{L^q } \les_q \lambda^{1/2}, \quad\text{and} \quad \nor{N}_{L^q }\les_q \lambda.
\end{equation}
Finally, if we allow $n = N_\lambda$ in \eqref{eq:rosenthal-centered} to be a random Poisson variable with parameter $\lambda \ge c$ for some constant $c>0$ and independent of a i.i.d.\ family $(Z_i)_{i=1}^\infty$ of random variables with mean $m$, we obtain
\begin{equation}\label{eq:rosenthal-poisson}\begin{split}
\nor{  \sum_{i=1}^{N_\lambda} Z_i -\EE\sqa{\sum_{i=1}^{N_\lambda} Z_i }}_{L^q} & \les  \nor{  \sum_{i=1}^{N_\lambda} (Z_i  -m)}_{L^q} + \nor{   (N_{\lambda} - \lambda) m }_{L^q} \\
& \les \EE\sqa{ N_\lambda} ^{1/q} \nor{Z_1-m}_{L^q} + \EE\sqa{ N_{\lambda}^{q/2}}^{1/q} \nor{Z_1-m}_{L^2} + \lambda^{1/2} |m|\\
& \les \lambda^{1/q} \nor{Z_1}_{L^q} + \lambda^{1/2}\nor{Z_2}_{L^2}
\end{split}
\end{equation}
where the implicit constant depends on $q$ and $c$ only.

\section{Brownian motion}\label{sec:bm}

\subsection{Brownian motion on $\R^d$}  Throughout this section, we let $d \ge 3$.  Given $x \in \R^d$  write $P_x$ for the law  of a $d$-dimensional Brownian motion  $(B_t)_{t \ge 0}$ staring at $x$, i.e., such that  $B_0=x$. Similarly, write $E_x$ for the expectation with respect to $P_x$.  We extend such notations for a probability measure $\nu$ on $\R^d$, writing
\begin{equation}
P_\nu = \int_{\R^d} P_x d\nu(x), \quad E_\nu = \int_{\R^d} E_x d \nu(x).
\end{equation}
The Brownian motion process $(B_t)_{t \ge 0}$ is (strong) Markov with generator $\frac 1 2 \Delta$, and transition density
\begin{equation}
 p_t(x, y)   = \exp\bra{ - \frac{|x-y|^2}{2t} } \frac{1}{(2 \pi t)^{d/2}}.
 \end{equation}
The associated Green function is
\begin{equation}\label{eq:green}
g(x,y) = \int_0^\infty p_{t}\bra{x,y} dt = \frac{c}{|x-y|^{d-2}},
\end{equation}
for a suitable constant $c = c(d)\in (0, \infty)$. It is elementary to check that for every $\rho>0$, given a Brownian motion $(B_t)_{t \ge 0}$ starting at $0$, its rescaled process $B^\rho := (\rho B_{t/\rho^2})_{ t \ge 0}$ has the same law of $(B_t)_{t \ge 0}$. This yields the following identity in law between occupation measures:
\begin{equation}
\mu_T^{B} \stackrel{law}{=} \dil_{\rho} \mu_{T/\rho^2}^B.
\end{equation}

%

For a  compact $K \subseteq \R^d$ with Lipschitz boundary, denote with $e_K$ its equilibrium measure, which is  defined as the  Borel measure $\mu$ concentrated on $K$  such that its Newtonian potential
\begin{equation}
u_\mu (x) := \int_{\R^d} g(x, y) d \mu(y),
\end{equation}
is identically $1$ for every $x \in K$, where $g$ is Green's function \eqref{eq:green}. It is well-known \cite{port2012brownian} that such a measure exists and is unique, also for sets $X$ with less regular boundaries (but it requires the notion of regular points). Its total mass $e_K(\R^d)$ is called the capacity $\Cap(K)$ of $K$. We write $\tilde{e}_K = \Cap(K)^{-1} e_K$ for its normalization to a probability measure, which we call the normalized equilibrium measure. With this notation, we have the following result.
%

\begin{proposition}\label{lem:sweeping}
Let $K \subseteq \R^d$. Then, $\Cap( u K ) = u^{d-2} \Cap(K)$ for every $u >0$, and  
\begin{equation}\label{eq:expected-occupation-measure}
 E_{e_K}\sqa{\mu_\infty^B(K)} = |K|.
\end{equation}
Let $\tilde{K} \subseteq K \subseteq \R^d$. Then,
\begin{equation}
P_{e_{K}} ( \tau_{\tilde{K}} < \infty, B_{\tau_{\tilde{K}}} \in A ) = P_{e_{\tilde{K}}}(B_0 \in A).
\end{equation}
In particular, 
\begin{equation}
P_{ \tilde{e}_{K}} ( \tau_{\tilde{K}} < \infty ) = \Cap(\tilde{K})/\Cap(K).
\end{equation}
\end{proposition}

We will use throughout that in case of balls $\tilde K = D_\ell$, ${K} = D_L$, with $\ell < L$, the normalized equilibrium measures are the uniform probability distributions on the respective boundaries. Moreover, given any $x \in \partial D_L$,  one can prove that
\begin{equation}\label{eq:prob-hitting-ball}
P_x (\tau_{D_\ell} < \infty) = \bra{\ell/L}^{d-2}.
\end{equation}

We end this section collecting some elementary upper and lower bounds on the occupation measures.

\begin{lemma}\label{lem:transient-bm-integral-bounds}
Let $d \ge 3$. For every bounded Borel set $A\subseteq \R^d$ and $x \in \R^d$, it holds
\begin{equation}\label{eq:transient-bm}
E_x \sqa{ \mu_\infty^B(A)} \les \min\cur{ \diam(A)^2,  \frac{|A|}{\dist(x, A)^{d-2} }}.
\end{equation}
For every $q \ge 1$,
\begin{equation} \label{eq:bm-p-power} E_x\sqa{ \mu_\infty^B(A) ^q } \les_{q} \diam(A)^{2q}. 
\end{equation}
\begin{proof}
Without loss of generality, we can assume that $A$ is closed. Consider first the case $q=1$. Assume first that $\dist(x,A)=0$. Then, $A \subseteq \cur{y: |x-y| \le \diam(A)}$, hence
\begin{equation} \int_0^\infty P_x(B_t \in A)  dt = \int_A \frac{c(d)}{|x-y|^{d-2}} dy \les \int_0^{\diam(A)} \frac{ r^{d-1}}{r^{d-2}} dr = \diam(A)^2.\end{equation}
If instead $\dist(x,A)>0$, write
\begin{equation}
\mu_\infty^B = \mu_\infty^{\theta_{\tau_A} B}
\end{equation}
and use that $\theta_{\tau_A} B$ is a Brownian motion (by the strong Markov property) with initial law $\nu$, concentrated on $A$ (actually, on $\partial A$). Hence,
\begin{equation}
E_x\sqa{ \mu_\infty^ B(A)} = \int_{\partial A} E_y\sqa{ \mu_\infty^{B}(A)} d\nu(y) \les \diam(A)^{2}.
\end{equation}
To obtain the estimate
\begin{equation}
E_x\sqa{ \mu_\infty^B(A) } \les   \frac{|A|}{\dist(x, A)^{d-2} },
\end{equation}
simply bound from above the green function $g(x,y) \les \dist(x, A)^{2-d}$ for $y \in A$. 

For the general case, it is sufficient to assume that $q \in \mathbb{N}$. Arguing by induction, write
\begin{equation}
E_x\sqa{ \mu_\infty^B(A) ^q } = \int_{(0, \infty)^q}   P_{x}(B_{t_i} \in  A\,\,   \forall i) dt_1 dt_2 \ldots dt_q
\end{equation}
and let $S_q := \cur{(t_i)_{i=1}^q \in (0, \infty)^q\, :\,  t_1 \le \ldots \le t_q}$. Then,
\begin{equation}
\int_{[0, \infty)^q} P_{x}(B_{t_i} \in  A\,\,   \forall i) dt_1 dt_2 \ldots dt_q = q! \int_{S_q} P_{x}(B_{t_i} \in  A\,\,  \forall i).
\end{equation}
Using the Markov property, we have
\begin{equation} \begin{split}
 \int_{S_q} P_{x}(B_{t_i} \in  A\,\,   \forall i ) & = \int_{S_{q-1}} \bra{  \int_{t_{q-1}}^\infty dt_q P_{y}(B_{t_q} \in A | B_{t_{q-1}} =y )} P_{x}(B_{t_i} \in  A\,\,   \forall i )\\
&  \les \diam(A)^2  \int_{S_{q-1}}  P_{x}(B_{t_i} \in  A\,\,   \forall i). \qedhere
\end{split}\end{equation}
\end{proof}
\end{lemma}

We can  also easily estimate from below the time spent in $D_1$ before exiting $D_L$, i.e., the variable $\mu_{\tau_{D_L^c}}^B(D_1)$ for a Brownian motion starting at $x \in D_1$.

\begin{lemma}\label{lem:expected-tau-before-L}
If $d \ge 3$, there exists $c=c(d)$ such that, if $\min\cur{L, T} \ge c$ it holds
\begin{equation}
\inf_{x \in {D}_1} E_x\sqa{ \mu_{T \land \tau_{D_L^c}}^B( D_1 )}  \ge 1/c.
\end{equation}
\end{lemma}
\begin{proof}
Notice first that, for every $x \in {D}_1$,
\begin{equation}
E_x\sqa{ \mu_{\infty}^B( D_1 )} = \int_0^\infty P_t(B_t \in D_1) dt =\int_{D_1} \frac{c(d)}{|x-y|^{d-2}} dy \gtrsim \int_{1/2}^{3/2} \frac{ r^{d-1}}{r^{d-2}} dr \gtrsim 1,
\end{equation}
where the first lower bound follows by intersecting $D_1$ with a suitable cone with vertex at $x$. Using the strong Markov property \eqref{eq:prob-hitting-ball} and \eqref{eq:bm-p-power}, we also find that
\begin{equation}
E_x \sqa{ \mu_{\infty}^{\theta_{\tau_{D_L^c}}B}( D_1 )} \les L^{2-d} \sup_{y \in \partial D_1}  E_y \sqa{ \mu_{\infty}^{B}( D_1 ) }\to 0
\end{equation}
as $L \to \infty$, and convergence is uniform with respect to $x \in \bar{D}_1$.
Moreover, denoting with $Z$ a standard Gaussian variable on $\R^d$,  by \eqref{eq:transient-bm}, we have
\begin{equation}\begin{split}
E_x\sqa{ \mu_{\infty}^{\theta_T B}(D_1)} & = \EE\sqa{ E_{x+\sqrt{T} Z}\sqa{ \mu_{\infty}^{B}(D_1)} }\\
& \les \EE\sqa{ \min\cur{1, \dist(x+\sqrt{T} Z, D_1)^{2-d} }} \to 0
\end{split}
\end{equation}
as $T \to \infty$, and convergence is uniform with respect to $x \in {D}_1$. Thus, we trivially bound from above
\begin{equation}
E_x \sqa{ \mu_{\infty}^{\theta_{ T \land \tau_{D_L^c}  }B}( D_1 )} \le E_x \sqa{ \mu_{\infty}^{\theta_{D_L^c}B}( D_1 )}+ E_x\sqa{ \mu_{\infty}^{\theta_T B}(D_1)} \to 0.
\end{equation}
By difference,
\begin{equation}
E_x \sqa{  \mu_{\tau_{D_L^c} \land T}^B( D_1 ) }  = E_x\sqa{ \mu_{\infty}^B( D_1 )} - E_x \sqa{ \mu_{\infty}^{\theta_{ \tau_{D_L^c} \land T }B}( D_1 )} \gtrsim 1   
\end{equation}
if $L$ and $T$ are both sufficiently large.
\end{proof}

\begin{remark}\label{rem:scaling}
A simple scaling argument, yields that, if $0<\ell< L<\infty$, $T>0$ are such that $\min\cur{L/\ell, T/\ell^2 } \ge c$, then
\begin{equation}
\inf_{x \in {D}_\ell} E_x \sqa{ \mu_{T \land \tau_{D_L^c}} ^B (D_\ell)} \ge \ell^2/c.
\end{equation}
\end{remark}

\subsection{Brownian motion on $\T^d$}

The above notions have their counterparts on Riemannian manifolds or even more generally on certain classes of metric measure spaces. We focus here on the flat torus $\T^d$, where the transition density  is given by
\begin{equation}
p_t(x, y) = \sum_{z \in \mathbb{Z}^d} \exp\bra{- \frac{ |x-y-z|^2}{2t}} \frac{1}{(2 \pi t)^{d/2}}.
\end{equation}
We write $P_x$, $E_x$ as in the case of $\R^d$. The (uniform) Lebesgue measure on $\T^d$ is the invariant measure for the Brownian motion, and we say that $(B_t)_{t \ge 0}$ is stationary if $B_0$ has uniform law. For every $p \ge 1$, there exists $c=c(p,d)$ such that
\begin{equation}
E_x \sqa{ \bra{ \dist_{\T^d} (x, B_t)}^p} \le C (t\land 1)^{p/2}, \quad \text{for $t \ge 0$.} 
\end{equation}
Moreover, the ultra-contractivity inequality holds: for every $t \ge 0$,
\begin{equation}\label{eq:ultra}
\sup_{x \in \T^d} E_x \sqa{ f(B_t) } \les t^{-d/2} \int_{\T^d} f .
\end{equation}
for every non-negative Borel $f: \T^d \to [0, \infty)$. We also have, for some $c=c(d)>0$,
\begin{equation}\label{eq:convergence-tv}
 \sup_{x \in \T^d} \TV(p_t(x, \cdot), 1) \les e^{- c t} \quad \text{for all $t>0$,}
\end{equation}
where with a slight abuse of notation write densities with respect to Lebesgue measure on $\T^d$ instead of measures. This implies, by \eqref{eq:contraction-tv} that for a Brownian motion $(B_t)_{t \ge 0}$  on $\T^d$ with any initial law,
\begin{equation}\label{eq:convergence-tv-any-law}
\TV\bra{ \PP_{(B_0, B_t)}, \PP_{B_0} \otimes 1 } \les e^{-ct} \quad \text{for all $t>0$. }
\end{equation}
The validity of \eqref{eq:convergence-tv} can be seen in many ways, e.g.\ by using the identity
\begin{equation}
p_t(x,y) = \sum_{z \in \mathbb{Z}^d} e^{i 2 \pi (x-y)\cdot z - t |z|^2/2 }
\end{equation}
and estimating, for $t \ge 1$ (the case $t<1$ is trivial, since $\TV \le 1$ anyway)
\begin{equation}
\begin{split}
\sup_{x \in \T^d} \TV(p_t(x, \cdot),  1) & \le \sum_{z \in \mathbb{Z}^d\setminus \cur{0}} e^{- c t |z|^2 } \les \int_{1}^\infty r^{d-1} e^{- c t r^2} dr\\
& \les t^{-d/2} \int_{\sqrt{t}}^\infty u^{d-1} e^{-c u^2} du \les  e^{-c' t}.
\end{split}
\end{equation}
A kind of counterpart of \Cref{lem:transient-bm-integral-bounds} is the following.

\begin{lemma}\label{lem:moment-bm-uniform-torus}
Let $d \ge 1$ and let  $B$ be a stationary Brownian motion on $\T^d$. Then, for every Borel $A \subseteq \T^d$, it holds
\begin{equation}
\EE\sqa{\mu_T^B(A) } = |A| T, \quad \text{for $T \ge 0$,}\label{lem:expectation-bm-stationary}
\end{equation}
and, for every $q \ge 1$, there exists $c=c(d,q)$ such that, for every $T \ge 1$,
\begin{equation} \label{eq:recurrent-bm-p-power}   \EE\sqa{ \abs{ \mu_T^B(A) - |A|T }^q }^{1/q}  \le  C |A|^{1/q^*} T^{1/2},\end{equation}
and $1/q^* := \min\cur{1, 1/p+1/d}$.
\end{lemma}
\begin{proof}
The first statement is straightforward, since
\begin{equation}
\EE\sqa{\mu_T^B(A) } =\int_0^T \PP(B_t \in A) dt = \int_0^T |A| dt = |A|T.
\end{equation}
To prove \eqref{eq:recurrent-bm-p-power}, since $L^q(\PP)$ norms are increasing with respect to $q$, we may assume that $q \ge d/(d-1)$, so that $\min\cur{1, 1/q+1/d} = 1/q+1/d$. Without loss of generality, we can also assume that $A$ is closed. Given a smooth function $g$ on $\T^d$, with $\int_{\T^d} g = 0$, let $f$ denote the solution to the elliptic equation
\begin{equation}
\frac 1 2 \Delta f = g, \quad \text{on $\T^d$.}
\end{equation}
By standard regularity theory, also $f$ is smooth, hence It\^o's formula applies yielding
\begin{equation}
\int g d \mu_T^B =\frac 1 2  \int_0^T \Delta f (B_s) ds = f(B_T) - f(B_0) - \int_0^T \nabla f( B_t) dB_t.
\end{equation}
Taking the $L^q(\PP)$ norm, we obtain by the triangle inequality and Burkholder-Davis-Gundy inequality that
\begin{equation}\begin{split}
\nor{ \int g d \mu_T^B}_{L^q(\PP)} & \les \nor{ f(B_T)}_{L^q(\PP)}+\nor{ f(B_0)}_{L^q(\PP)}
+ \nor{\int_0^T \nabla f(B_t) dB_t }_{L^q(\PP)}\\
&  \les \nor{ f(B_T)}_{L^q(\PP)} +  \nor{  \sqrt{ \int_0^T \abs{\nabla f}^2(B_t) dt }}_{L^p(\PP)}\\
& \les \nor{ f}_{L^q(\T^d)} +  \sqrt{T} \nor{ \nabla f}_{L^q(\T^d)}.
\end{split}
\end{equation}
By considering a sequence of smooth functions $(g_n)_n$ with $\int_{\T^d} g_n = 0$ for every $n \ge 1$ and such that $g_n \to \chi_A - |A|$ in $L^q(\T^d)$ but also $\int g_n d\mu_T^B \to \mu_T(A) - T|A|$ $\PP$-a.s., then by Calderon-Zygmund theory, we have that the induced solutions $f_n$ converge in the Sobolev space $W^{2,q}(\T^d)$ to the solution $f$ to $\frac 1 2 \Delta f = \chi_A - |A|$, and 
\begin{equation}
\nor{\nabla^2 f}_{L^q(\T^d)} \les \nor{ \chi_A - |A|}_{L^q(\T^d)} \les |A|^{1/q}.
\end{equation}
 By the Sobolev embedding on $\T^d$, we obtain
\begin{equation}
\nor{ f }_{L^p(\T^d)} + \nor{ \nabla f }_{L^p(\T^d)} \les |A|^{1/q},
\end{equation}
which leads to the thesis.
\end{proof}
%
%
%
%
%
%

\subsection{Hitting probabilities}

In this section, we consider a stationary Brownian motion $B$ on $\T^d$ and we estimate the hitting probability
\begin{equation}
 \PP( \sigma < \tau_{D_\ell} \le \rho ) 
\end{equation}
as well as the conditional hitting law $\nu_{\rho}$, defined as
\begin{equation}
\nu_{\rho} (A) := \PP( B_{\tau_{D_\ell}} \in A | 0 < \tau_{D_\ell} \le \rho ),
\end{equation}
for $\ell \to 0$ and $\rho \to \infty$ in the regime $\rho \ll \ell^{2-d}$, and $0\le \sigma < \rho$. These assumptions will entail that the event has a small probability, and that $\nu_{\sigma, \rho}$ is very close to the  uniform distribution on $\partial D_\ell$ (i.e., the normalized equilibrium measure $\tilde{e}_{D_\ell}$ for $D_\ell \subseteq \R^d$).
%

The problem of estimating the hitting probabilities of small sets much studied on manifolds, but we are not aware of bounds for the hitting law, hence we review and expand some results from \cite{GRIGORYAN2002115, drewitz2014introduction}. Consider a Brownian motion $(B_t)_{t \ge 0}$ on $M=\R^d$ or $\T^d$ and transition density $p(t,x,y)$.

Given a compact $K \subseteq M$ with smooth boundary, write $p_{K^c}(t,x,y)$ for the transition density of the Brownian motion killed upon hitting $K$. Recall that $p_{K^c}$ is symmetric, i.e., $p_{K^c}(t,x,y) = p_{K^c}(t,y,x)$ for every $x$, $y \in M$. In particular, on $\T^d$ this yields the identity
\begin{equation}\label{eq:key-identity-torus-hitting-time}
\int _{\T^d} p_{K^c}(t,x,y) dx = \int_{\T^d} p_{K^c}(t,y,x) dx = P_y(  \tau_{K} >t).
\end{equation}

Given a bounded Borel function on $\partial K$, we are interested in estimating the function
\begin{equation}
(t,x) \mapsto \Phi(t,x) := E_x\sqa{ \phi(B_{\tau_K}) I_{\cur{ 0< \tau_K \le t}} },
\end{equation}
which is a solution to the equation
\begin{equation}
\begin{cases}   \partial_t \Phi =\frac 1 2 \Delta \Phi &  \text{in $(0, \infty)\times K^c$,}\\
  u(0,x) = 0 &  \text{for $x \in K^c$,}\\
  u(t,x) = \phi(x) & \text{for $(0, \infty) \times \partial K$.}
  \end{cases}
  \end{equation}
Moreover, one has the representation
\begin{equation}
\partial_t \Phi(t,x) = \int_{\partial K} \phi(y) \partial_{\mathsf{n}} p_{K^c}(t,x, y)  \sigma(dy),
\end{equation}
where $\sigma$ denotes the surface measure on $\partial K$ and $\mathsf{n}$ the inward normal. Combining these facts,  the following representation theorem holds (see also \cite[Lemma 3.1]{GRIGORYAN2002115}).

\begin{lemma}\label{lem:fat-boundary}
Let $K \subseteq K' \subseteq M$ have smooth boundaries and let $\chi$ be (sufficiently smooth) and such that
\begin{equation}
\chi = \phi \quad \text{on $\partial K$,} \quad \text{and}\quad   \partial_{\mathsf{n}} \chi = 0 \quad \text{on $\partial K'$.}
\end{equation}
Then, for every $t>0$, $x \in K^c$,
\begin{equation}\label{eq:fat-boundary}
 \partial_t \Phi(t,x) = \int_{K' \setminus K} p_{K^c}(t,x,y) \Delta \chi(y) dy  ds -   \int_{K' \setminus K}  \partial_t p_{K^c}(t,x,y) \chi(y) dy.
\end{equation}
\end{lemma}

Assume that $M = \T^d$ and that $B$ is a stationary Brownian motion. Then, integrating  \eqref{eq:fat-boundary} over $x \in \T^d$ and using \eqref{eq:key-identity-torus-hitting-time} yields
\begin{equation}
 \partial_t \EE\sqa{ \phi(B_{\tau_K}) I_{\cur{ 0< \tau_K \le t}}} = \int_{K' \setminus K} P_y( \tau_{K} >t) \Delta \chi (y)  dy -  \int_{K'\setminus K}  \partial_t P_y(  \tau_{K} >t) \chi (y) dy.
\end{equation}
Further integration with respect to $t\in (\sigma, \rho]$ leads to the identity
\begin{equation}\label{eq:integrated-very-useful}
\begin{split}
 \EE\sqa{ \phi(B_{\tau_K}) I_{\cur{ \sigma < \tau_K \le \rho}}} & = \int_{\sigma}^\rho \int_{K' \setminus K} P_y( \tau_{K} >t) \Delta \chi (y)  dy \\
 & \quad + \int_{K'\setminus K}  P_y( \sigma < \tau_{K} \le \rho) \chi (y) dy 
\end{split}\end{equation}
which we crucially use to establish the following result.

\begin{proposition}\label{prop:hitting-times}
Let $d \ge 3$, $0< \gamma<d-2$ 
and for assume that $0\le \sigma<\rho$, with $\rho \sim \ell^{-\gamma}$ as $\ell \to 0$. 
Given a stationary Brownian motion $B$ on $\T^d$, it holds
\begin{equation}\label{eq:estimate-hitting-time}
 \abs{ \PP( \sigma <\tau_{D_\ell}(B) \le \rho ) - (\rho-\sigma) \ell^{d-2} \Cap(D_1) }\les (\rho- \sigma)\rho \ell^{2(d-2)} \abs{\log \ell} + \ell^d.
 \end{equation} 
Moreover,
\begin{equation}\label{eq:estimate-wasserstein-hitting-distribution}
W_{D_\ell}^p(\nu_{\sigma, \rho}, \tilde{e}_{D_\ell} ) \les \ell^{p}  \cdot \bra{ \rho \ell^{d-2} \abs{\log \ell} + \ell^2/(\rho-\sigma)},
\end{equation}
where $\tilde{e}_{D_\ell}$ denote the uniform probability measure on $\partial D_\ell$, and the implicit constants depend on $d$, $\gamma$ and the implicit constant in the condition $\rho \sim \ell^{-\gamma}$ only.
\end{proposition}

\begin{proof}


We start from \eqref{eq:integrated-very-useful} in the case $K = D_\ell$, $K' = D_{2 \ell}$.  Recall that we identify $D_\ell \subseteq \T^d$ with a subset of $[-1/2,1/2)^d$, and similarly $y\in \T^d$ with $y \in [-1/2, 1/2)^d$. We further define
\begin{equation}
\tilde{D}_\ell = \bigcup_{z \in \mathbb{Z}^d} D_\ell(z) \subseteq \R^d
\end{equation}
and write $\tilde{P}_y$ for the law of a Brownian motion on $\R^d$ starting at $y$. With this notation, we notice that
\begin{equation}
P_y( \tau_{D_\ell } >t)  = \tilde{P}_y (\tau_{\tilde{D}_\ell } >t) =  \tilde{P}_y (\tau_{D_\ell }>t) - \tilde{P}_y ( \tau_{D_\ell}>t,  \tau_{\tilde{D}_\ell \setminus D_\ell} \le t). 
\end{equation}
Therefore,
\begin{equation}
\abs{ P_y( \tau_{D_\ell} >t) -   \tilde{P}_y (\tau_{D_\ell} >t)} \le \tilde{P}_y (\tau_{\tilde{D}_\ell \setminus D_\ell} \le t),
\end{equation}

Our next aim is to  bound from above the probability in the right hand side above. By the exponential maximal inequality for martingales, it holds, for every $M>0$ and $t >0$,
\begin{equation}
\tilde{P}_y\bra{ \sup_{0\le s \le t} \abs{ B_s -y} \ge \sqrt{ t M} } \les \exp\bra{ - M^2/2}.
\end{equation}
Therefore, we focus on bounding from above the probability
\begin{equation}
\tilde{P}_y \bra{ \tau_{\tilde{D}_\ell \setminus D_\ell} \le t, \quad  \sup_{s \le t} \abs{B_s-y} \le \sqrt{ t M}}.
\end{equation}
By the triangle inequality, in the event above it must hold that $\tau_{D_\ell(z)} < \infty$ for some $z \in \mathbb{Z}^d \setminus \cur{0}$, such that
\begin{equation}
|z| \le 3 \ell + \sqrt{tM}.
\end{equation}
Notice that, as $\ell\to 0$, we have that the right hand side is $<1$ e.g.\ if $tM<1/2$. In such a case the event has null probability. In any case, we define 
\begin{equation}
\bar{k} :=  \lfloor 3 \ell +\sqrt{t M} \rfloor,
\end{equation}
then, the probability is bounded from above by
\begin{equation}\begin{split}
\sum_{0< |z| \le \bar{k}} P_y( \tau_{D_\ell(z)} < \infty) &  \les \sum_{k=1}^{\bar{k}} \bra{ \frac{k}{\ell}}^{2-d} \sharp \cur{ z \in \mathbb{Z}^d\, :\,  k-1 < |z| \le k }\\
&\les \sum_{k=1}^{\bar{k}} \bra{ \frac{k}{\ell}}^{2-d} k^{d-1} \les \ell^{d-2} \bar{k}^2 \les \ell^d + \ell^{d-2}t M
\end{split}
\end{equation}
We find therefore the estimate
\begin{equation}
\sup_{y \in D_{2 \ell}\setminus D_{\ell} } \abs{ P_y( \tau_{K} >t) -   \tilde{P}_y (\tau_{K} >t)} \les 
 \exp\bra{-M/2} + \ell^{d-2} t M + \ell^d  
\end{equation}
Letting $M=- 2\log(\ell^{d-2}\rho ) \sim | \log \ell|$, we obtain
\begin{equation}
\sup_{y \in D_{2 \ell}\setminus D_{\ell} } \abs{ P_y( \tau_{K} >t) -   \tilde{P}_y (\tau_{K} >t)} \les  \ell^{d-2} \rho |\log \ell| + \ell^d.
\end{equation}
 Using this bound in \eqref{eq:integrated-very-useful} for $\sigma \le t \le \rho$, we obtain
\begin{equation}\label{eq:almost-there}
\begin{split}
  \Bigg| \EE\sqa{ \phi(B_{\tau_{D_\ell}}) I_{\cur{ \sigma < \tau_{D_\ell}(B) \le \rho }}} & - \int_{\sigma }^\rho \int_{D_{2 \ell} \setminus D_{\ell}} \tilde P_y( \tau_{D_\ell} >t) \Delta \chi (y)  dy   \Bigg| \\
 &  \les \rho \ell^{d-2} \abs{ \log \ell} (\rho - \sigma)  \int_{D_{2\ell}\setminus D_{\ell}}  |\Delta \chi(y)|dy + \int_{D_{2 \ell} \setminus D_\ell}   \abs{ \chi (y)}dy. 
\end{split}\end{equation}

We are now in a position to establish \eqref{eq:estimate-hitting-time}. We set $\phi(x)=1$ for every $x \in \partial D_\ell$ and we let $\chi$ be any smooth cut-off function $v$ on $D_2$ such that $v=1$ on $\partial D_1$ and $v = \partial _\mathsf{n} v = 0$ on $\partial D_2$, 
and letting $\chi(x) = v(x/\ell)$. We obtain that $\abs{\chi(y)} \les 1$, $\abs{ \Delta \chi } \les \ell^{-2}$, so that
\begin{equation}
\int_{D_{2 \ell} \setminus D_\ell}  \abs{ \chi (y)}dy \les \ell^d, \quad \int_{D_{2 \ell} \setminus D_\ell}  \abs{ \Delta \chi (y)}dy \les \ell^{d-2},
\end{equation}
so the second line in \eqref{eq:almost-there} is bounded from above:
\begin{equation}\label{eq:error-term-good}
\rho (\rho -\sigma) \ell^{d-2} \abs{ \log \ell}   \int_{D_{2\ell}\setminus D_{\ell}}  |\Delta \chi(y)|dy + \int_{D_{2 \ell} \setminus D_\ell}  \abs{ \chi (y)}dy \les \rho(\rho -\sigma) \ell^{2(d-2)} \abs{\log \ell} + \ell^d
\end{equation}
For the first line, we use the scaling properties of Brownian motion on $\R^d$ so that
\begin{equation}\begin{split}
\int_{0}^\rho \int_{D_{2 \ell} \setminus D_{\ell}} \tilde P_y( \tau_{D_\ell } >t) \Delta \chi (y)  dy  dt &= \int_{0}^\rho \int_{D_{2 \ell} \setminus D_{\ell}} \tilde P_{y/\ell}( \tau_{D_1} >t/\ell^2) \Delta \chi (y)  dy dt\\
& = \ell^d \int_0^\rho \int_{D_{2} \setminus D_{1}}\tilde P_{z}( \tau_{D_1} >t/\ell^2) \Delta \chi (\ell z) dz dt\\
& = \rho \ell^{d-2} \int_0^1 \int_{D_{2} \setminus D_{1}}\tilde P_{z}( \tau_{D_1} >s \rho /\ell^2) \Delta v (z) dz ds.
\end{split}
\end{equation}
As $\ell \to 0$, we have by dominated convergence that
\begin{equation}
 \int_0^1 \int_{D_{2} \setminus D_{1}}\tilde P_{z}( \tau_{D_1} >s \rho /\ell^2) \Delta v (z) dz \to  \int_{D_{2} \setminus D_{1}}\tilde P_{z}( \tau_{D_1} =\infty) \Delta v (z) dz.
\end{equation}
Finally, using Green's identity and the properties of $v$ and $\tilde P_{z}( \tau_{D_1} =\infty)$ we obtain that
\begin{equation}\label{eq:error-term-good-final}
 \int_{D_{2} \setminus D_{1}}\tilde P_{z}( \tau_{D_1} =\infty) \Delta v (z) dz = \int_{\partial D_1} \partial_\mathsf{n} P_z( \tau_{D_1} = \infty) d z = \Cap(D_1).
\end{equation}
The convergence above can be made quantitative e.g.\ by a classical result by Port \cite[Theorem 2.3]{port2012brownian} (see also \cite{GRIGORYAN2002115} for similar bounds on manifolds):
\begin{equation}
\tilde P_{z}( t< \tau_{D_1}< \infty) \les  t^{1-d/2}\land 1.
\end{equation}
as $t \to \infty$, uniformly on $z \in D_2 \setminus D_1$. As a consequence, we obtain
\begin{equation}\begin{split}
 \Bigg| \int_0^1 \int_{D_{2} \setminus D_{1}}\tilde P_{z}( \tau_{D_1} >s \rho /\ell^2) \Delta v (z) dzds & -  \int_{D_{2} \setminus D_{1}}\tilde P_{z}( \tau_{D_1} =\infty) \Delta v (z) dz \Bigg|\\
 & \les \int_0^1 (s \rho/\ell^2)^{1-d/2} \land 1 ds \int_{D_2 \setminus D_1} \abs{\Delta v(z)} dz\\
 & \les \frac{\ell^2}{\rho} \int_{D_2 \setminus D_1} \abs{\Delta v(z)} dz \les \frac{\ell^{2}}{\rho}.
 \end{split}
\end{equation}
If $\sigma = 0$, this concludes the proof of \eqref{eq:estimate-hitting-time}. Otherwise, we argue similarly and obtain
\begin{equation}
\abs{ \int_{0}^\sigma \int_{D_{2 \ell} \setminus D_{\ell}} \tilde P_y( \tau_{D_\ell } >t) \Delta \chi (y)  dy  dt - \sigma \ell^{d-2}\Cap(D_1) } \les  \ell^d,
\end{equation}
and by difference \eqref{eq:estimate-hitting-time}.

Next, we address the proof of \eqref{eq:estimate-wasserstein-hitting-distribution}. In view \eqref{eq:estimate-hitting-time} and using \Cref{rem:peyre-sphere}, it is sufficient to consider $\phi \in H^{1,p'}(\partial D_1)$ with
\begin{equation}
\int_{\partial D_1} \phi d\sigma = 0, \quad \text{and} \quad \int_{\partial D_1} \abs{\nabla \phi}^{p'} d \sigma \le 1,
\end{equation}
and establish the inequality
\begin{equation}
\EE\sqa{ \phi( B_{\tau_\ell} /\ell ) I_{\cur{ 0 < \tau_{D_\ell} \le \rho}}} \les  \bra{\rho\ell^{d-2}}^2 \abs{\log \ell} + \ell^d
\end{equation}
(where the implicit constant does not depend on $\phi$). To this aim, we choose a variant of the function $\chi$ in \eqref{eq:almost-there} by letting again $\chi(x) = v(x/\ell)$, where in this case $v$ (defined on $D_2$) enjoys the following properties:
\begin{equation}\label{eq:properties-magical}
v = \phi \quad \text{on $\partial D_1$}, \quad  v = \partial_n v = 0 \quad \text{on $\partial D_2$,} \quad  \text{and} \quad  \int_{D_2 \setminus D_1} \abs{v} + \abs{\Delta v} \les 1.
\end{equation}
%
Granted that such a function indeed exists, we obtain that the argument goes exactly in the same way starting from \eqref{eq:almost-there} as in the previous case, in particular the second line in \eqref{eq:almost-there} is bounded from above as in \eqref{eq:error-term-good} and we are finally lead to \eqref{eq:error-term-good-final}. In this case, however, integrating by parts we conclude that
\begin{equation}\begin{split}
 \int_{D_{2} \setminus D_{1}}\tilde P_{z}( \tau_{D_1} =\infty) \Delta v (z) dz& = \int_{\partial D_1} \phi(z) \partial_{\mathsf{n}} P_z( \tau_{D_1} = \infty) d z \\
 & =\Cap(D_1) \int_{\partial D_1} \phi(z) dz = 0.
 \end{split}
\end{equation}
Thus, we only need to show that a $v$ satisfying \eqref{eq:properties-magical} exists. We build it in the following way: first, we consider an extension of $\phi$ in $D_2$ by letting
\begin{equation}
\tilde{\phi}(x) = \phi(x/|x|)
\end{equation}
This extension has zero average $\int_{D_2 \setminus D_1} \tilde{ \phi} = 0$ and satisfies
\begin{equation}
\int_{D_2 \setminus D_1} | \nabla \tilde{\phi}|^{p'} \les \int_{\partial D_1} \abs{ \nabla \phi }^{p'} d \sigma \les 1
\end{equation}
by integrating in radial coordinates. In particular, by Poincaré inequality on $D_2 \setminus D_1$,
\begin{equation}
\int_{D_2 \setminus D_1} \tilde{\phi}  \les 1. 
\end{equation}
Next, we solve the problem
\begin{equation}
\begin{cases}
\Delta u = \operatorname{div} (\nabla \tilde{\phi} ) &  \text{ in $D_2 \setminus D_1$}\\
u = 0 & \text{on $\partial D_1 \cup \partial D_2$}
\end{cases}
\end{equation}
obtaining a function $u$ with zero average $\int_{D_2 \setminus D_1} u = 0$ and (by global Calderon-Zygmund theory) such that
\begin{equation}
 \int \abs{\nabla u}^{p'} \les \int_{D_2 \setminus D_1} | \nabla \tilde{\phi}|^{p'} \les 1.
 \end{equation}
 We can then extend $u$ to a Sobolev function identically null on $D_1$. Therefore, by Poincaré inequality for functions on $D_2$ that are null on $D_1$, we obtain
 \begin{equation}
  \int_{D_2 \setminus D_1} \abs{u} \les \int \abs{\nabla u} \les 1.
  \end{equation} 
  We finally define the function
  \begin{equation}
  v = (\tilde{\phi} - u) \eta,
  \end{equation}
  where $\eta (x) = \eta(|x|) \in [0,1]$ is a cut-off function with $\eta(1) =1$, $\eta(2) = \eta'(2) = 0$. Clearly $v = \phi$ on $\partial D_1$,  $v = \partial_n v = 0$ on $\partial D_2$ and 
  \begin{equation}
   \int_{D_2 \setminus D_1} \abs{v} \les \int_{D_2 \setminus D_1} |\tilde{\phi}| + \abs{u} \les 1.
   \end{equation} 
   Finally, using that $\Delta \tilde{\phi} = \Delta u$ (in distributional sense), we obtain that
   \begin{equation}
   \begin{split}
   \Delta v & = \eta \Delta (\tilde{\phi} - u) +  (\tilde{\phi} - u)  \Delta \eta +  2 \nabla (\tilde{ \phi} - u) \nabla \eta\\
   & =  (\tilde{\phi} - u)  \Delta \eta +  2 \nabla (\tilde{ \phi} - u) \nabla \eta,
   \end{split}
   \end{equation}
   which eventually leads to 
   \begin{equation}
   \int_{D_2 \setminus D_1} \abs{\Delta v} \les 1,
   \end{equation}
   hence the thesis is settled.
\end{proof}

Next, we recall the following simple bound for exit times (for Brownian motion on $\T^d$ or equivalently on $\R^d$)
 
\begin{lemma}\label{lem:exit-time}
Let $d \ge 1$ and $0<\ell < L/2<1/2$. Then, there exists $c=c(d)>0$ such that
\begin{equation}
\sup_{x \in D_{\ell}} P_x \bra{ \tau_{D_{L}^c}(B) > t } \les e^{-c t/L^2}.
\end{equation}
\end{lemma}
\begin{proof}
A simple union bound yields that the probability in dimension $d$ is estimated from above (a sum) of the similar probabilities in dimension $d=1$. But then this is a well-known consequence of the gambler's ruin problem.
\end{proof}

We end this section with estimates for iterated hitting times of $D_\ell$. 
Precisely, given $0<\ell < L <1/2$ and a Brownian motion $B$ on $\T^d$, we define $\tau^1_\ell = \tau_{D_\ell} B$, and iteratively
\begin{equation}\begin{split}
\tau_{k,L} & := \inf \cur{ t \ge \tau_{k-1,\ell} \, : \, B_t \in D_L^c }\\
\tau_{k+1,\ell}  & := \inf \cur{ t\ge \tau_{k,L} \, : \, B_t \in D_\ell}.
\end{split}
\end{equation}
%
With this notation, we have the following result.

\begin{corollary}\label{cor:hitting-twice}
Let $0\le \gamma_L <1$ and $0<\gamma_\rho<\gamma_L(d-2)$ and set for $\ell \in (0,1/2)$, $\rho \sim \ell^{-\gamma}$, $L\sim \ell^{\gamma_L}$. Given a stationary Brownian motion  $B$ on $\T^d$, it holds, for every $k \ge 1$, 
\begin{equation}\label{eq:estimate-hitting-time-twice}
 \PP( \tau_{k,\ell} \le \rho ) \les_k  (\rho \ell^{d-2})^k, \quad \text{as $\ell \to 0$.}
 \end{equation}
 \end{corollary} 
 \begin{proof}
 With the notation of the previous proof, we first argue that, if $\sigma \sim \ell^{-\gamma_\sigma}$ for some $0<\gamma_\sigma< \gamma_\rho$, then
\begin{equation}\label{eq:upper-bound-hitting-time}
\sup_{y \in \partial D_L} P_y(\tau_{D_\ell} B \le \sigma ) \les \rho \ell^{d-2}.
\end{equation}
Indeed, for every $\alpha>0$,
\begin{equation}
\tilde{P}_y\bra{ \sup_{0\le s \le \sigma} \abs{ B_s -y} \ge \sqrt{ \alpha \sigma \abs{\log \sigma}} } \les \sigma^{- \alpha^2/2},
\end{equation}
If $B_0 =y \in \partial D_L$, in the event \begin{equation}\label{eq:good-event-lemma-multiple-hittings}
\cur{ \tau_{\tilde{D}_\ell} \le \sigma, \quad  \sup_{s \le \sigma} \abs{B_s-y} \le \sqrt{ \alpha \sigma \abs{\log \sigma}}},
\end{equation}
by the triangle inequality, it must hold that $\tau_{D_\ell(z)} < \infty$ for some $z \in \mathbb{Z}^d$, such that
\begin{equation}
|z| \le \ell + L+ \sqrt{\alpha \sigma \abs{ \log \sigma}}.
\end{equation}
We define 
\begin{equation}
\bar{j} :=  \lfloor \ell +L +\sqrt{\alpha \sigma \abs{ \log \sigma }} \rfloor,
\end{equation}
then, the probability of the event \eqref{eq:good-event-lemma-multiple-hittings} is bounded from above by
\begin{equation}\begin{split}
\sum_{|z| \le \bar{j}} P_y( \tau_{D_\ell(z)} < \infty) &  \les P_y( \tau_{D_\ell(0)} < \infty) +  \sum_{j=1}^{\bar{j}} \bra{ \frac{j}{\ell}}^{2-d} \sharp \cur{ z \in \mathbb{Z}^d\, :\,  j-1 < |z| \le j }\\
&\les \bra{\frac{ L}{\ell}}^{2-d}+ \sum_{k=1}^{\bar{k}} \bra{ \frac{k}{\ell}}^{2-d} k^{d-1} \les \ell^{d-2}(L^{2-d}+ \bar{k}^2) \\
& \les   \ell^{d-2}\bra{L^{2-d} +  \alpha \sigma \abs{\log \sigma}}.
\end{split}
\end{equation}
We find therefore the estimate
\begin{equation}
\sup_{y \in \partial D_L} P_y (\tau_{K} >t) \les \sigma^{-\alpha^2/2}+ \ell^{d-2} \bra{ L^{2-d}  +  \alpha \sigma \abs{\log \sigma} } \les \ell^{d-2}\bra{ L^{2-d} + \sigma \abs{\log \sigma}}
\end{equation}
provided that we chose $\alpha$ sufficiently large. Having settled \eqref{eq:upper-bound-hitting-time}, we argue by induction upon $k$, the case $k=1$ being already settled in \Cref{prop:hitting-times}. We write, for $k>1$, the inequality
\begin{equation}\label{eq:split-hitting-times-induction}
 \PP( \tau^k_\ell \le \rho )  \le  \PP( \tau^k_\ell - \tau^{k-1}_L \le \sigma, \tau^{k}_\ell \le \rho ) +  \PP(  \tau^k_\ell - \tau^{k-1}_L > \sigma,  \tau^k_\ell \le \rho ).
\end{equation}
for some $\sigma = \ell^{-\gamma_\sigma}$ with $0<\gamma_\sigma<\gamma_\rho$. For the first term, we apply \eqref{eq:upper-bound-hitting-time} and argue that
\begin{equation}\begin{split}
\PP( \tau^k_\ell - \tau^{k-1}_L \le \sigma,  \tau^k_\ell \le \rho ) & \le \PP( \tau^k_\ell - \tau^{k-1}_L \le \sigma,  \tau^{k-1}_\ell \le \rho ) \\
&  \le \EE\sqa{ I_{\cur{ \tau^{k-1}_\ell \le \rho }}  \PP\bra{ \tau^k_\ell - \tau^{k-1}_L \le \sigma \mid \cF_{\tau^{k-1}_\ell} }}\\
& \le \PP (\tau^{k-1}_\ell \le \rho) \sup_{y \in \partial D_L} P_y(\tau_{\ell}(B) \le \sigma )   ] \\
& \les  (\rho \ell^{d-2})^{k-1}\ell^{d-2}\bra{L^{2-d}+ \sigma \abs{ \log \sigma }} \ll (\rho \ell^{d-2})^{k}
\end{split}
\end{equation}
provided that $\gamma_\sigma < \gamma_\rho$ and using also the condition $\gamma_\rho<(d-2)\gamma_L$. For the second term in the right hand side of \eqref{eq:split-hitting-times-induction}, we write
\begin{equation}
\PP\bra{ \tau^k_\ell - \tau^{k-1}_L > \sigma,  \tau^k_\ell \le \rho } \le \PP\bra{ \tau^{k-1}_\ell \le \rho, \, \tau_{D_\ell} \theta_{\tau^{k-1}_L + \sigma} B \le \rho}.
\end{equation}
%
%
We introduce the variables 
\begin{equation}
U= (B_{t \land \tau_L^{k-1}})_{t \ge 0}, \quad V = (B_{\tau_L^{k-1}}, B_{\tau_L^{k-1}+\sigma})\quad \text{and} \quad W = (B_{\tau_L^{k-1}+\sigma+t})_{t \ge 0},
\end{equation}
which define a Markov chain by the strong Markov property. By \eqref{eq:markov-chain-tv} and \eqref{eq:convergence-tv}, we deduce that
\begin{equation}
\abs{ \PP( \tau_{\ell}^{k-1} \le \rho, \, \tau_{D_\ell} \theta_{\tau^{k-1}_L + \sigma} B) - \PP( \tau_{\ell}^{k-1} \le \rho, \tau_{D_\ell} \tilde{B} \le \rho )} \les  e^{-c \sigma} \ll_k (\rho \ell^{d-2})^k
\end{equation}
where $\tilde{B}$ is a stationary Brownian motion on $\T^d$, independent from $B$, and we used the fact that $\sigma =\ell^{-\gamma_\sigma}$ and $\gamma_\sigma>0$. By induction and \Cref{prop:hitting-times},
\begin{equation}
\PP( \tau_{\ell}^{k-1} \le \rho, \tau_{D_\ell} \tilde{B} \le \rho ) = \PP( \tau_\ell^{k-1} \le \rho) \PP( \tau_{D_\ell}\tilde B \le \rho) \sim (\rho \ell^{d-2})^k.
\end{equation}
This settles the thesis \eqref{eq:estimate-hitting-time-twice}.
%
 \end{proof}

\section{Brownian interlacement occupation measure}\label{sec:bi}

Unless specified otherwise, let $d\ge 3$. Given an (intensity) parameter $u>0$ and a compact $K\subseteq \R^d$, we introduce the following random measure on the Borel subsets of $K$: we consider independent $((B^i_t)_{t \ge 0})_{i=1}^\infty$ Brownian motions, all with initial law $\tilde{e}_K$ and a further independent Poisson random variable $N$, with mean $u \Cap(K)$, and let
\begin{equation}
\cI_u\restr K = \sum_{i=1}^{N} \mu_\infty^{B^i}\restr K, 
\end{equation}
i.e., for $A \subseteq K$ Borel,
\begin{equation}
\cI_u\restr K (A) = \sum_{i=1}^{N} \mu_\infty^{B^i}(A) = \sum_{i=1}^{N} \int_{0}^\infty I_{\cur{B^i _t \in A}} dt.
\end{equation}
Notice that, by \eqref{eq:expected-occupation-measure}, it holds
\begin{equation}\label{eq:expected-occupation-measure-interlacement}
\EE\sqa{ \cI_u\restr K (K) } = u |K|.
\end{equation}

\subsection{Basic facts}
The restriction notation ``$\restr K$'' is justified by the following result.

\begin{lemma}
Let $u>0$, $\tilde{K} \subseteq K \subseteq \R^d$. Then,
\begin{equation}
 (\cI_u \restr K)  \restr \tilde{K}, \quad \text{and} \quad \cI_u \restr \tilde{K}
 \end{equation}
 have the same law. 
\end{lemma}

\begin{proof}
We have, by \eqref{eq:restriction-occupation-measure},
\begin{equation}
\begin{split}
(\cI_u \restr K)\restr \tilde{K}  & = \sum_{i=1}^{N} \mu^{B^i}_{\infty} \restr \tilde{K} = \sum_{i=1}^{N} I_{\cur{ \tau^{i} < \infty}} \, \mu^{ \theta_{\tau^i}(B^i)} _\infty \restr \tilde{K}.
\end{split}
\end{equation}
where we write for brevity $\tau^i = \tau_{\tilde{K}}^{B^i}$. 
Each variable $I_{\cur{\tau^i < \infty}}$ has  Bernoulli law with parameter $\Cap(\tilde{K})/\Cap(K)$ by \Cref{lem:sweeping} and are all independent, hence the summation is the same (in law) as performed over a Poisson variable $\tilde{N}$ with mean $u \Cap(\tilde{K})$. Moreover, by the strong Markov property, conditionally upon $\cur{ \tau^i < \infty}$, the process $\theta_{\tau^i} (B^i)$ is a Brownian motion with initial law $\tilde{e}_{\tilde{K}}$, and are all independent. Thus, 
\begin{equation}
\sum_{i=1}^{N} I_{\cur{ \tau^{i} < \infty}} \, \mu^{ \theta_{\tau^i}(B^i)} _\infty \restr \tilde{K} \quad \text{and} 
\quad
\cI_u\restr \tilde{K} = \sum_{i=1}^{\tilde{N}} \mu^{ \tilde{B}^i}_{\infty} \restr \tilde{K} 
\end{equation}
have the same law.
  \end{proof}

In view of the result above, we may ``glue'' together all the measures $\cI_u \restr K$, and define a random Borel measure $\cI_u$ on $\R^d$. This can be technically achieved by considering a sequence of compacts e.g.\ $K_n = D_n$ and letting $n \to \infty$ and considering a limit in law (we leave the details since one could actually dispense from considering the limit object by always restricting to a sufficiently large $n$). a measure plays the role of a Poisson point process with intensity $u$ in our setting.  Notice that, by \eqref{eq:expected-occupation-measure-interlacement} and the fact that $A \mapsto \EE\sqa{\cI_u(A)}$ is a measure, it follows that $\EE\sqa{ \cI_u(A) } = u |A|$ for every $A\subseteq \R^d$ Borel.

\begin{lemma}
For every $u>0$, $\rho>0$ and $x \in \R^d$, the following identities in law hold:
\begin{equation}\label{eq:invariance-mu}
\tras_x \cI_u= \cI_u, \quad \text{and} \quad \dil_\rho \cI_u = \rho^{-2} \cI_{u/\rho^{d-2}}.
\end{equation}
In particular, 
\begin{equation}\label{eq:invariance-mu-1}
\dil_{u^{-1/(d-2)}} \cI_1 =  u^{2/(d-2)} \cI_u.
\end{equation}
\begin{proof}
The first identity follows by translation invariance of the equilibrium measure, i.e., the fact that $\tras_x \tilde{e}_{K} = \tilde{e}_{x+K}$ for every compact $K \subseteq \R^d$. In particular, $\Cap(K) = \Cap(x+K)$. Moreover, if $B = (B_t)_{t \ge 0}$ is a Brownian motion with initial law $\tilde{e}_K$, the process $x+B = (x+B_t)_t$ is a Brownian motion with initial law $\tilde{e}_{x+K}$. Thus,
\begin{equation}
\tras_x (\cI_u \restr K) = \sum_{i=1}^N \tras_x \mu_\infty^{B^i} = \sum_{i=1}^N  \mu_\infty^{x+B^i},
\end{equation}
which clearly has the law of $\cI_u \restr (x+K)$.

 For the second identity, we notice first that
\begin{equation}
\tilde{e}_{\rho K} = \dil_\rho \tilde{e}_{K}, \quad \text{and} \quad \Cap( \rho K) = \rho^{d-2} \Cap(K).
\end{equation}
so that $N$ is a Poisson variable with mean
\begin{equation}
\EE\sqa{N} = u\Cap(K) = \frac{ u}{\rho^{d-2}}\Cap(\rho K).
\end{equation}
Moreover, if $B = (B_t)_{t \ge 0}$ is a Brownian motion with initial law $\tilde{e}_K$, the process $\tilde{B}^\rho = (\rho B_{t/\rho^2})_{t \ge 0}$ has the same law of a Brownian motion with initial law  $\dil_\rho \tilde{e}_{K }$. Therefore,
\begin{equation}
\begin{split}
\dil_\rho (\cI_u \restr K) & = \dil_\rho  \sum_{i=1}^N \mu_\infty^{B^i} =\sum_{i=1}^N \int_0^\infty \delta_{\rho B_t^i} dt = \rho^2 \sum_{i=1}^N \int_0^\infty \delta_{\rho B_{t/\rho^2}^i} dt
\end{split}
\end{equation}
has the same law as $\rho^2 \cI_{u/\rho} \restr \rho K$. 
%
%
%
%
\end{proof}
\end{lemma}

In the next lemma, we consider the concentration properties for $\cI_1(A)$ (the general case follows from \eqref{eq:invariance-mu-1}).

\begin{lemma}
Let $d \ge 3$. For every $q \ge 1$, there exists $C = C(d,q) <\infty$ such that, if $\diam(A) \ge 1$, 
\begin{equation}\label{eq:concentration-mu-lambda}
\|  \cI_{1}(A) - \EE\sqa{ \cI_1(A) }\|_q \le C \diam(A)^{ (d+2)/2}.
\end{equation}
\begin{proof}
Without loss of generality, we can assume $q \ge 2$. Let $K$ be a ball of radius $\diam(A)$ such that $A \subseteq  K$.
 Then, $\cI_1(A) = \sum_{i=1}^N \mu_\infty^{B^i}(A)$ is the sum of a Poisson number (with mean $\Cap(K) =\diam(A)^{d-2} \Cap(D_1) \approx \diam(A)^{d-2}$) of i.i.d.\ variables, each with finite moments of all orders and bounded from above in \Cref{lem:transient-bm-integral-bounds}. By Rosenthal's inequality in the version \eqref{eq:rosenthal-poisson}, it follows that $\cI_1(A)$ has finite moments of all orders, and
\begin{equation}\begin{split}
\nor{ \cI_1(A) - \EE\sqa{ \cI_1(A)} }_q & \les_{q,d} \bra{ \diam(A)^{(d-2)/q} +\diam(A)^{(d-2)/2}} \diam(A)^{2} \\
& \les_{q,d} \diam(A)^{(d+2)/2}. \qedhere
\end{split}
\end{equation}
%
%
%
%
\end{proof}
\end{lemma}

\subsection{Limit results}

We are now in a position to establish the convergence as $u \to \infty$ of $\EE\sqa{W_{\Omega}^p(\cI_u)}$. The arguments are a modification of those originally devised for the random matching problem in \cite{BaBo, goldman2021convergence, ambrosio2022quadratic, goldman2022optimal} and employ only basic invariance properties of $\cI_u$. Thus, we  provide a complete derivation in \Cref{app:asymptotics} in a more abstract setting and here specialize to the interlacement occupation measure.
%

\begin{theorem}\label{thm:poi-inter}
Let $d \in \cur{ 3, 4}$ and $p \in (0, (d-2)/2)$, or $d \ge 5$ and $p>0$. Then, there exists a constant $\c(\cI, d, p) \in (0, \infty)$ such that, for every bounded connected domain $\Omega$ with $C^2$ boundary (or $\Omega = Q$ a cube) it holds
\begin{equation}\label{eq:limit-interlacement}
\lim_{u \to \infty} \EE\sqa{ W_{\Omega}^p (\cI_u) }/u^{1-p/(d-2)}  = \c(\cI, d, p) |\Omega|.
\end{equation}
\end{theorem}

\begin{proof}
If we apply \Cref{theo:domain} with $\nu = \cI_1$, which satisfies the conditions i) (stationarity), ii) (integrability) and iii) (concentration) with $\alpha=2$, we obtain
 \begin{equation}
\lim_{n \to \infty}  \EE\sqa{ W_{\Omega}^p( \dil_{n^{-1/d}} \cI_1 ) }/n^{1-p/d} = \c(\cI, d, p) |\Omega|,
\end{equation}
for some constant $\c(\cI, d, p) \in [0, \infty)$. Setting $n = n(u) = u^{d/(d-2)}$ so that $n^{-1/d} = u^{-1/(d-2)}$, hence $\dil_{n^{-1/d}} \cI_1  = u^{2/(d-2)}\cI_u$ by \eqref{eq:invariance-mu-1}, we obtain \eqref{eq:limit-interlacement}. The fact that $\c(\cI, d,p)$ is strictly positive will follow as a consequence of \Cref{thm:main-torus} and \Cref{prop:lower-bound}.
\end{proof}

We then obtain \Cref{thm:main-bm-iid}, which is nothing but a \emph{de-Poissonized} version of \eqref{eq:limit-interlacement}, where the number of Brownian motions is deterministic. Since this type of arguments are also rather standard, and here we follow closely  \cite{goldman2021convergence, goldman2022optimal}, we prefer to obtain it as a consequence of a more general result, see \Cref{prop:depoisson} in  \Cref{app:depoisson}.


We end instead this section by investigating how the limit behaves if we assume that the Brownian motions do not start exactly with the equilibrium measure. We restrict ourselves to the case of a ball, although we conjecture that similar bounds should hold true for general domains. 

In the next lemma, write $\cM$ for the $\sigma$-algebra on $C([0, \infty); \R^d)^{\otimes n}$ generated by the map
\begin{equation}
((x^i_t)_{t \ge 0})_{i=1}^n \mapsto ((|x^i_t|)_{t \ge 0})_{i=1}^n,
\end{equation}

\begin{lemma}\label{lem:stability}
Let $d \ge 1$ and $\nu$, $\tilde{\nu}$ be  probability measures on $\partial D_1$.  Let, $(\tau^i)_{i=1}^n$ be $\cM$-measurable functions with values in $[0, \infty]$. Let $B = (B^i)_{i=1}^n$  be independent Brownian motions, all with initial law $\nu$ and $\tilde{B} = (\tilde{B}^i)_{i=1}^n$ be independent Brownian motions, all with initial law $\tilde{\nu}$. If $p >0$, then there exists $c=c(p)< \infty$ such that, for every $\eps \in [0, 1)$,
\begin{equation}\begin{split}\label{eq:lem-stability}
\EE\sqa{ W_{D_1}^p\bra{ \sum_{i=1}^{n} \mu_{\tau^i(B)}^{B^i}}} & \le (1+\eps) \EE\sqa{ W_{D_1}^p\bra{ \sum_{i=1}^{n} \mu_{\tau^i(\tilde{B})}^{\tilde{B}^i} }} \\
& \quad  + \frac{C}{\eps^{(p-1)^+}} \EE\sqa{ \sum_{i=1}^{n}\mu_{\tau^i({B})}^{{B}^i}(D_1) } W_{D_1}^p\bra{ \nu, \tilde{\nu}}.
\end{split}\end{equation}
Moreover, in the second line one can also replace $B$ with $\tilde B$.
\end{lemma}
\begin{proof}
Consider an optimal transport plan between $\nu$ and $\tilde{\nu}$, and assume for simplicity that it is induced by a map $\Psi: \partial D \to \partial D$. We use $\Psi$ to induce a coupling between a Brownian motion $B$ with initial law $\nu$ and a Brownian motion $\tilde{B}$ with initial law $\tilde{\nu}$ in the following way. We let $U = U(\tilde{B}_0)\in \R^{d\times d}$ be the orthogonal transformation which acts as a rotation, on the plane spanned by $\cur{B_0, \Psi(B_0)}$, mapping $B_0$ into $\Psi(B_0)$ (if $B_0 = \Psi(B_0)$ we simply let $U$ be the identity). We then define
\begin{equation}
\tilde{B}_t = U B_t, \quad \text{for $t \ge 0$,}
\end{equation}
which is a Brownian motion, with $\tilde{B}_0 = \Psi(B_0)$, so that its initial law is $\tilde{e}_{D}$. Moreover, since $(|\tilde{B}_t|)_{t \ge 0} = (\abs{B_t})_{t \ge 0}$, we have, for every $T \ge 0$,
\begin{equation}
U_\sharp \mu_T^{B} \restr D = \mu_T^{\tilde{B}} \restr D,
\end{equation}
In addition, we have $\tau^B = \tau^{\tilde{B}}$ and $\mu_T^{B}(D)= \mu_T^{\tilde{B}}(D)$ and because each process $(\abs{B_t})_{t \ge 0}$ is independent of $B_0$ (for $\abs{B_0} = 1$) we have that $(\tau^B, \mu_{\tau^B}^B(D) )$ and $\abs{B_0 -\Psi(B_0)}$ are independent random variables.
   Since
\begin{equation}
\abs{\tilde{B}_t - B_t} \le \nor{I - U}\abs{B_t} = \abs{B_0 -\Psi(B_0)}, \quad \text{for every $t \ge 0$ such that $B_t \in D$,}
\end{equation}
we conclude that, for every $T \ge 0$, 
\begin{equation}\label{eq:bound-coupling-bm}
W_{D}^p \bra{ \mu_T^{B}, \mu_T^{\tilde{B}} } \le \mu_T^{\tilde{B}}(D) \abs{B_0 - \Psi(B_0)}^p,
\end{equation}
having used the coupling induced by $U$. We apply this construction to each $B^i$ and notice that $\tau_i = \tilde{\tau}_i$ and the variables
\begin{equation}
(\tau^i, \mu_{\tau^i}^{B^i}(D) )_{i=1}^n, \quad \text{and} \quad \bra{ \abs{B^i_0 -\Psi(B^i_0)}}_{i=1}^n
\end{equation}
are independent. Using this independence, \eqref{eq:sub} and \eqref{eq:bound-coupling-bm}, we thus obtain
\begin{equation}
\begin{split}
\EE  \sqa{ W_{D_1}^p \bra{ \sum_{i=1}^{n} \mu_{\tau^i}^{B^i} ,  \sum_{i=1}^{n} \mu_{\tilde{\tau}^i}^{\tilde{B}^i} }} & \le \EE\sqa{ \sum_{i=1}^n  \mu_{\tau_i}^{B^i}(D) \abs{B^i_0 - \Psi(B^i_0)}^p}\\
 & \le \EE\sqa{ \EE\sqa{ \sum_{i=1}^n  \mu_{\tau_i}^{B^i}(D) \abs{B^i_0 - \Psi(B^i_0)}^p \bigg| (\tau^i, \mu_{\tau^i}^{B^i}(D) )_{i=1}^n }}\\
& = \EE\sqa{ \sum_{i=1}^n \mu_{\tau_i}^{B^i}(D) }  W_{D_1}^p( \nu, \tilde{\nu} ).
\end{split}
\end{equation}
The thesis then follows from the triangle inequality and an application of \eqref{eq:elementary}.
\end{proof}

\section{Occupation measure for the Brownian motion on the torus}\label{sec:torus}

In this section, we study the asymptotics for the occupation measure of a Brownian motion on $\T^d$ and establish \Cref{thm:main-torus}. Most of the argument is in fact contained in the following ``local'' result, showing that by focusing only the cost $W_{D_\ell}^p$ for a small ball $D_\ell$, with $\ell = T^{-\gamma}$ for $\gamma$ smaller (but sufficiently close) to $1/(d-2)$, the constant $\c(\cI, d,p)$ associated to the Brownian interlacement of the previous section will appear in a suitably renormalized limit.

\begin{proposition}\label{prop:local-torus}
Let $d \in \cur{3,4}$ and $p \in (0, (d-2)/2)$, or $d \ge 5$ and $p>0$. Let $B =(B_t)_{t \ge 0}$ be a stationary Brownian motion on $\T^d$. There exists $\bar{\gamma} = \bar{\gamma}(d,p) \in (0,1/(d-2))$ such that the following holds. For every $\gamma \in (\bar{\gamma}, 1/(d-2))$, letting  
 $\ell = T^{-\gamma}$, it holds
\begin{equation}
\lim_{T \to \infty} \EE\sqa{  W^p_{D_\ell}\bra{ \mu^B_T }} / \bra{ T^{1-p/(d-2)} |D_\ell|}= \c(\cI, d,p).
\end{equation}
with $\c\bra{\cI, d,p}$ as in \Cref{thm:main-bm-iid}.
\end{proposition}

\begin{proof} 
 The strategy is to couple the given Brownian motion on $\T^d$ with a family of independent Brownian motions on $\R^d$, with common initial law sufficiently close to $\tilde{e}_{D_1}$, and use \Cref{thm:main-bm-iid}. Of course the crux of the argument is to take into account the several error terms due to this approximation. We split the proof into several steps.
\newcounter{step}

\stepcounter{step} \noindent \emph{Step \thestep \, (Time splitting).} We introduce two additional parameters $\gamma_\rho$, $\gamma_\sigma$, such that
\begin{equation}\label{eq:gamma-ell-rho-first-condition}
  0< \gamma_\sigma< \gamma_\rho  < \gamma(d-2)
 \end{equation}
 and to be specified below: infact, we are going to collect several constraints and only check at the final step that they can be all satisfied. To guide the intuition, we may think of $\gamma_\sigma$ and $\gamma_\rho$ to be both very close to $0$. We define the quantities
 \begin{equation}
 n:= \lfloor T^{1-\gamma_\rho} \rfloor, \quad \sigma := T^{\gamma_\sigma}, \quad \rho := \frac{T}{n} - \sigma.
 \end{equation}
 Notice that, by our choices of the parameters, it holds
\begin{equation}
1 \ll \sigma \ll \rho \ll \ell^{2-d} \quad \text{ and } \quad  T = n \bra{\rho+\sigma}.
\end{equation}
%
For  $i=0, \ldots, n-1$, we consider the intervals
\begin{equation}
I_i= [i(\rho+ \sigma), i(\rho+\sigma)+ \rho), \quad J_i = [i(\rho+\sigma) +\rho, (i+1)(\rho+\sigma)),
\end{equation}
and decompose
\begin{equation}
\mu_T  = \sum_{i=0}^{n-1} \bra{ \int_{I_i} \delta_{B_t} dt + \int_{J_i} \delta_{B_t} dt}  
 = \sum_{i=0}^{n-1} \bra{ \mu_{\rho}^{\theta_{i(\rho+\sigma)}B} + \mu_{\sigma}^{\theta_{i(\rho+\sigma)+\rho}B}} 
\end{equation}

All the shifted processes appearing in the expression above are stationary Brownian motions on the torus, hence have the same law as $B$ (but they are not independent). Therefore, by \eqref{lem:expectation-bm-stationary}, we have
\begin{equation}\label{eq:step-1-first-error-bound}\begin{split}
\diam(D_\ell)^p \EE\sqa{ 
\sum_{i=0}^{n-1}\mu_{\sigma}^{\theta_{i(\rho+\sigma)+\rho}B} (D_\ell) } & \le {n \sigma} \ell^{p+d}  \quad \le \bra{ T^{1-\gamma_\rho+\gamma_\sigma}}T^{-p\gamma} \ell^d \\
&  \ll T^{1-p/(d-2)} \ell^d,
\end{split}
\end{equation}
provided that  the following condition is satisfied:
 \begin{equation}\label{eq:q:gamma-ell-rho-second-condition}
\gamma_\rho - \gamma_\sigma   > p\bra{1/(d-2) - \gamma}.
 \end{equation}
Starting from \eqref{eq:step-1-first-error-bound}, we easily deduce that
\begin{equation}\label{eq:step-1-same-asymptotics}
\abs{ \EE\sqa{  W^p_{D_\ell}\bra{ \mu^B_T }} - \EE\sqa{  W^p_{D_\ell}\bra{\sum_{i=1}^{n} \mu_{\rho}^{\tilde{B}^i} }}} \ll T^{1-p/(d-2)} \ell^d. 
\end{equation}
Indeed, by \eqref{eq:sub} and \eqref{eq:trivial-wass}, we obtain that
\begin{equation}
\EE\sqa{ W_{D_\ell}^p(\mu_T^B)} - \EE\sqa{ W_{D_\ell}^p\bra{ \sum_{i=0}^{n-1}  \mu_{\rho}^{\theta_{i(\rho+\sigma)}B}}} \ll  T^{1-p/(d-2)} \ell^d.
\end{equation}
Similarly, by \eqref{eq:same-asymptotics-trivial}, for every $\eps>0$,
\begin{equation}
 \EE\sqa{ W_{D_\ell}^p\bra{ \sum_{i=0}^{n-1}  \mu_{\rho}^{\theta_{i(\rho+\sigma)}B}}} - (1+\eps) \EE\sqa{ W_{D_\ell}^p(\mu_T^B)}\ll_\eps  T^{1-p/(d-2)} \ell^d,
\end{equation}
so letting $T\to \infty$ and then $\eps\to 0$, we  see that \eqref{eq:step-1-same-asymptotics} holds. Hence, the thesis follows if we prove that 
\begin{equation}\label{eq:step-1-end}
\lim_{T \to \infty} \EE\sqa{  W^p_{D_\ell}\bra{\sum_{i=1}^{n} \mu_{\rho}^{\tilde{B}^i} }} / \bra{ T^{1-p/(d-2)} |D_\ell|}= \c(\cI, d,p),
\end{equation}
where we write, for brevity,
\begin{equation}\label{eq:tilde-b-i}
\tilde{B}^i := (B_{(i-1)(\rho+\sigma) + t})_{t \in [0, \rho]}.
\end{equation}

%
%

\stepcounter{step} \noindent \emph{Step \thestep \, (Breaking dependence).} The processes $(\tilde{B}^i)_{i=1}^n$ are stationary Brownian motions on $\T^d$, but of course they are not independent. 
In this step we argue that in \eqref{eq:step-1-end} is actually possible to replace  them  with $n$ independent (and stationary) Brownian motions $(B^i)_{i=1}^n$ on $\T^d$. This will follow from the bound
\begin{equation}\label{eq:step-2-bound}
   \abs{ \EE\sqa{  W^p_{D_\ell}\bra{\sum_{i=0}^{n-1}  \mu_{\rho}^{\tilde B^i}} } -  \EE\sqa{ W^p_{D_\ell}\bra{\sum_{i=1}^n \mu_{\rho}^{B^i}}}} \ll  T^{1-p/(d-2)}|D_\ell|, 
\end{equation}
which in turn is proved by repeated applications of \eqref{eq:markov-chain-tv} and \eqref{eq:convergence-tv-any-law}.
%
%
 Precisely, consider the following  non-negative function of $n$ continuous curves $((x^i_t)_{t \in [0, \rho]})_{i=1}^n$ on $\T^d$:
\begin{equation}
(x^i)_{i=1}^n \mapsto F \bra{(x^i)_{i=1}^n} =  W_{D_\ell}^p \bra{ \sum_{i=1}^{n}\mu_{\rho}^{x^i} }.
\end{equation}
By \eqref{eq:trivial-wass},
\begin{equation}\label{eq:uniform-bound-big-F}
 \sup_{ (x^i)_{i=1}^n } F \bra{(x^i)_{i=1}^n} \le n \rho \ell^p.
 \end{equation} 
%
Then, \eqref{eq:step-2-bound} follows from the inequality 
\begin{equation}\label{eq:breaking-dependence}
\abs{ \EE\sqa{ F\bra{ (\tilde{B}^i )_{i=1}^n }} - \EE\sqa{ F( (B^i)_{i=1}^n )}} \les n e^{-c \sigma } n \rho \ell^p 
\ll T^{1-p/(d-2)} \ell^d,
\end{equation}
where $c>0$ is as in \eqref{eq:convergence-tv-any-law}, and the second inequality easily follows since $\gamma_\sigma>0$.

To obtain \eqref{eq:breaking-dependence}, we use telescopic summation  on the following set of inequalities, valid for $k=1, \ldots, n-1$,
\begin{equation}\label{eq:breaking-dependence-iteration}
\abs{ \EE\sqa{ F\bra{ B^1, \ldots, B^{k-1}, \tilde{B}^{1},\ldots \tilde{B}^{n-k+1}}} - \EE\sqa{ F( B^1, \ldots, B^{k}, \tilde{B}^{1}, \ldots \tilde{B}^{n-k})} } \le e^{-c \sigma } n \rho \ell^p.
\end{equation}
For each $k$, the above inequality is an application of  \eqref{eq:markov-chain-tv}, with the variables
\begin{equation}
U = (B^1, \ldots, B^{k-1}, \tilde{B}^{1}), \quad V = (\tilde{B}^{1}_{\rho}, \tilde{B}^{2}_0), \quad W=( \tilde{B}^{2}, \ldots, \tilde{B}^{n-k+1}).
\end{equation}
Notice that they define a Markov chain because, conditionally upon
\begin{equation}
V = (\tilde{B}^{1}_{\rho+\sigma},  \tilde{B}^2_0) = (B_{\rho},  B_{\rho+ \sigma}),
\end{equation} $U$ and $V$ are independent.
By \eqref{eq:convergence-tv-any-law}, we estimate
\begin{equation}
\TV\bra{\mathbb{P}_{(B_{\rho}, B_{\rho+\sigma})}, \mathcal{L}^d_{\T^d} \otimes \mathcal{L}^d_{\T^d}} \les e^{-c \sigma}.
\end{equation}
Thus, using \eqref{eq:markov-chain-tv}, we can turn $U$ (and in particular $\tilde{B}^1$) and $W=(\tilde{B}^2, \ldots, \tilde{B}^{n-k+1})$ independent, and let the initial distribution of $\tilde{B}^2$ be uniform on $\T^d$,  effectively replacing the joint law of $W$ with that of $(\tilde{B}^1, \ldots, \tilde{B}^{n-k})$, with an error in total variation distance that is $\les e^{-c \sigma}$. Because of  \eqref{eq:uniform-bound-big-F}, we see that \eqref{eq:breaking-dependence-iteration} follows.

\stepcounter{step} \noindent \emph{Step \thestep \, (Series of visits).} Our aim is  to prove 
\begin{equation}\label{eq:step-1-end}
\lim_{T \to \infty} \EE\sqa{  W^p_{D_\ell}\bra{\sum_{i=1}^{n} \mu_{\rho}^{B^i} }} / \bra{ T^{1-p/(d-2)} |D_\ell|}= \c(\cI, d,p),
\end{equation}
where $(B^i)_{i=1}^n$ are independent and stationary Brownian motions on $\T^d$. 
In this step, we decompose each measure $\mu_\rho^{B^i}$ into a series of contributions due to the sequence of ``visits'' to $D_\ell$. 
We introduce an additional parameter 
\begin{equation}\label{eq:q:gamma-ell-rho-third-condition}
0<\gamma_L < \gamma
\end{equation}
and set $L = T^{-\gamma_L}$. 
For each $i=1, \ldots, n$, we introduce the sequence of hitting times
letting $\tau_{1,\ell}^i = \tau_{D_\ell} B^i$, and iteratively
\begin{equation}\begin{split}
\tau_{k,L}^i & := \inf \cur{ t \ge \tau_{k-1,\ell}^i \, : \, B_t^i \in D_L^c }\\
\tau_{k+1,\ell}^i  & := \inf \cur{ t\ge \tau_{k,L}^i \, : \, B_t^i \in D_\ell}.
\end{split}
\end{equation}
Notice that between the times $\tau_{k, \ell}^i$ and $\tau_{k, L}^i$, $B^i$ can be identified with a Brownian motion on $\R^d$, if $T$ is sufficiently large so that $L < 1/4$ (we assume this in what follows).
%
We  write
\begin{equation}
K^i := \inf\cur{k\, : \, \tau^i_{k, \ell} > \rho } 
\end{equation}
so that
\begin{equation}\label{eq:good-event}
\mu_{\rho}^{B^i} \restr D_\ell = \sum_{k=1}^{K^i-1} \int_{\tau_{k, \ell}^i}^{\tau_{k, L}^i \land \rho } \delta_{B^i_s} ds,
\end{equation}
hence we have the series representation
\begin{equation}
\sum_{i=1}^n \mu_{\rho}^{B^i} \restr D_\ell = \sum_{k=1}^\infty \lambda_k,
\end{equation}
where each term is defined as follows:
\begin{equation}
\lambda_k = \sum_{i\in S_k } \int_{\tau_{k, \ell}^i}^{\tau_{k, L}^i \land \rho } \delta_{B^i_s} ds
\end{equation}
and $S_k =\cur{ i\, : K_i >k}$. In the next steps, we prove that the main contribution in the sum above comes from the term $\lambda_1$.  Actually, we further analyze by splitting also the contribution of $\lambda_1$ into two parts: we let $S\subseteq S_1$ denote the (random) subset of $i\in \cur{1, \ldots, n}$ such that the event
\begin{equation}
A^i := \cur{ 0<\tau_{1,\ell}^i \le \rho } 
\end{equation}
holds, and set $S_0 = S_1\setminus S_0$. We write  $\lambda_1 = \lambda_0 + \tilde{\lambda}_1$, with
\begin{equation}
\lambda_0 = \sum_{i \in S_0} \int_{\tau_{1, \ell}^i}^{\tau_{1, L}^i \land \rho } \delta_{B^i_s} ds, \quad \tilde{\lambda}_1 = \sum_{i \in S} \int_{\tau_{1, \ell}^i}^{\tau_{1, L}^i \land \rho } \delta_{B^i_s} ds.
\end{equation}
In the following steps, our aim is to prove that
%
%
\begin{equation}\label{eq:reduction-single-visit}
\abs{ \EE\sqa{ W_{D_\ell}^p\bra{ \sum_{i=1}^n \mu^{B^i}_\rho }} - \EE\sqa{W_{D_\ell}^p\bra{ \tilde{\lambda}_1} }} \ll T^{1-p/(d-2)} \ell^d.
\end{equation}

\stepcounter{step} \noindent \emph{Step \thestep \, (Upper and lower bounds for visiting times).} We notice first that,  by \Cref{prop:hitting-times},
\begin{equation}
  \PP(0< \tau_{1, \ell}^i < \rho/2) \gtrsim \rho \ell^{d-2}.
\end{equation}
By identifying $\tilde{B}^i = \theta_{\tau_{1,\ell}^i}B^i$ up to the time $\tau_{D_L^c} = \tau_{1, L}^i-\tau_{1, \ell}^i$ with a Brownian motion on $\R^d$, and using the strong Markov property and \Cref{rem:scaling}, it holds
\begin{equation}
\EE\sqa{ I_{A^i} \int_{\tau_{1, \ell}^i}^{\rho \land \tau_{1, L}^i} I_{\cur{B_s\in D_\ell}} ds} \ge \EE\sqa{ I_{ \cur{0< \tau_{1, \ell}^i < \rho/2} } \mu_{(\rho/2) \land  \tau_{D_L^c} }^{ \tilde{B}^i  }(D_\ell) } \gtrsim \rho \ell^{d-2} \cdot \ell^2. 
\end{equation}
Summation upon $i=1,\ldots, n$ leads to the lower bound
\begin{equation}
\EE\sqa{ \tilde{\lambda}_1(D_\ell) } \gtrsim n \rho \ell^d \gtrsim T \ell^d.
\end{equation}
On the other side, again by \Cref{cor:hitting-twice} and the strong Markov property, but using \Cref{lem:transient-bm-integral-bounds}, it holds for every $q \ge 1$, the upper bound
\begin{equation}
\EE\sqa{ \bra{ I_{A^i} \int_{\tau_{k, \ell}^i}^{\rho \land \tau_{k, L}^i} I_{\cur{B_s\in D_\ell}} ds}^q } \les  \rho \ell^{d-2} \cdot \ell^{2q}.
\end{equation}
By Rosenthal inequality \eqref{eq:rosenthal-centered}, it follows that, for every  $q \ge 2$,
\begin{equation}\begin{split}
\EE\sqa{  \abs{ \tilde{\lambda}_1(D_\ell) - \EE\sqa{ \tilde{\lambda}_1(D_\ell) }}^q } & \les n \cdot \rho \ell^{d-2} \cdot \ell^{2q} +   \bra{n \rho \ell^{d-2} \cdot \ell^{4}}^{q/2} \\
& \les (T \ell^{d-2})^{q/2} \ell^{2q} ,
\end{split}
\end{equation}
having used that $T \ell^{d-2} \to \infty$. Such concentration bound yields that, for some constant $c=c(d)>0$, the event 
\begin{equation}\label{eq:good-g1}
G_1 = \cur{ \tilde{\lambda}_1(D_\ell) \ge c T \ell^d  }
\end{equation}
has large probability. Precisely,
\begin{equation}\label{eq:probability-good-g1}
\PP(G_1^c) \le \PP\bra{  \tilde{\lambda}_1(D_\ell)- \EE\sqa{\tilde{\lambda}_1(D_\ell)} \le c' T \ell^d } \les_q \frac{ (T \ell^{d-2})^{q/2} \ell^{2q}}{(T \ell^{d-2})^q \cdot  \ell^{2q}} = (T \ell^{d-2})^{-q/2} \ll_\alpha T^{-\alpha},
\end{equation}
for every $\alpha>0$, provided that we choose $q$ sufficiently large.

We argue similarly that, for every $k\ge 2$ and any $\delta\in (0,1)$ the event
\begin{equation}
G_k  = \cur{ \lambda_k(D_\ell) \le T \ell^{d} \delta }
\end{equation}
has large probability: for every $\alpha>0$, it holds
\begin{equation}
 \PP(G_k^c) \les_{k,\delta,\alpha} T^{-\alpha},
 \end{equation} Indeed,  by \Cref{cor:hitting-twice}, the strong Markov property and \Cref{lem:transient-bm-integral-bounds} it holds, for every $q \ge 1$,
\begin{equation}
\EE\sqa{ \bra{ I_{\cur{i \in S_k}} \int_{\tau_{k, \ell}^i}^{\rho \land \tau_{k, L}^i} I_{\cur{B_s\in D_\ell}} ds}^q } \les_k  (\rho \ell^{d-2})^k \cdot \ell^{2q}.
\end{equation}
By Rosenthal's inequality for positive variables \eqref{eq:rosenthal-positive},
\begin{equation}\begin{split}\label{moment-lambda-k}
\EE\sqa{ \lambda_k(D_\ell)^q } & =  \EE\sqa{ \bra{ \sum_{i\in S_k} \int_{\tau_{k, \ell}^i}^{\rho \land \tau_{k, L}^i} I_{\cur{B_s\in D_\ell}} ds}^q }  \\
& \les_k n  (\rho \ell^{d-2})^k \cdot \ell^{2q} + \bra{  n (\rho \ell^{d-2})^k \cdot \ell^2 }^{q} \les T \ell^{d-2} (\rho \ell^{d-2})^{k-1} \ell^{2q},
\end{split}
\end{equation}
having used that
\begin{equation}
n (\rho \ell^{d-2})^k \les  T \ell^{d-2} (\rho \ell^{d-2})^{k-1} \ll 1
\end{equation}
provided that we impose the condition
\begin{equation}
\label{eq:gamma-condition-i-forgot}
1-\gamma(d-2) - [(d-2) \gamma - \gamma_\rho] <0.
\end{equation}
Hence, 
\begin{equation}\label{eq:good-gk}
\PP(G_k^c)  \les_{k,\delta} \frac{ T \ell^{d-2} (\rho \ell^{d-2})^{k-1} \ell^{2q}}{ (T \ell^{d-2})^q \ell^{2q}} = (T \ell^{d-2})^{1-q} (\rho \ell^{d-2})^{k-1} \les_{\alpha} T^{-\alpha},
\end{equation}
provided that $q$ chosen is large enough (here we use that $\rho \ll \ell^{2-d}$). Notice that if $q$ is chosen if \eqref{eq:good-hk} holds for some $k$, then it also holds for $h \ge k$ (but the implicit constant may degenerate).

The (possible) dependence upon $k$ in the bounds \eqref{moment-lambda-k} and \eqref{eq:good-gk} is taken into account as follows. 
Set
\begin{equation}\label{eq:good-hk}
H_k  = \cur{ \sum_{h \ge k} \lambda_h(D_\ell)  \le T \ell^{d} \delta },
\end{equation}
and apply \Cref{cor:hitting-twice}, for every $q \ge 1$,
\begin{equation}
\EE\sqa{ \bra{ I_{\cur{i \in S_k}} \sum_{h \ge k} \int_{\tau_{h, \ell}^i}^{\rho \land \tau_{h, L}^i} I_{\cur{B_s\in D_\ell}} ds}^q } \les \PP( \tau_{k, \ell}^i \le \rho ) \rho^q \les_k (\rho \ell^{d-2})^k \cdot \rho^q,
\end{equation}
having bounded from above the total mass by $\rho$. By Rosenthal's inequality for positive variables \eqref{eq:rosenthal-positive},
\begin{equation}\label{eq:moment-sum-h-gtrk}
\EE\sqa{ \bra{ \sum_{h \ge k} \lambda_h(D_\ell) }^q } \les n \cdot  (\rho \ell^{d-2})^k \cdot \rho^q + \bra{ n (\rho \ell^{d-2})^k \cdot \rho }^q \les T \ell^{d-2}  (\rho \ell^{d-2})^{k-1} \rho^q.
\end{equation}
Hence, by Markov inequality and simply choosing $q=1$,
\begin{equation}
\begin{split}
\PP(H_k^c) &  \les_{k,\delta} \frac{   T \ell^{d-2}  (\rho \ell^{d-2})^{k-1} \rho }{ T \ell^{d}} =  (\rho \ell^{d-2})^{k-1} \rho/\ell^2 \\
& \les T^{-k\sqa{\gamma(d-2) - \gamma_\rho} + \gamma d}  \les T^{-\alpha}
\end{split}
\end{equation}
for every $\alpha>0$, provided that we now pick $k = k(\alpha)$ sufficiently large (recall that by assumption $\gamma(d-2) >\gamma_\rho$).

To conclude this step, we argue that, for any $\delta\in (0,1)$ the event
\begin{equation}
G_0  = \cur{ \lambda_0(D_\ell) \le T \ell^{d} \delta }
\end{equation}
has large probability: for every $\alpha>0$, it holds
\begin{equation}
 \PP(G_0^c) \les_{\delta,\alpha} T^{-\alpha},
 \end{equation} Indeed,  by stationarity and \Cref{lem:transient-bm-integral-bounds} it holds, for every $q \ge 1$,
\begin{equation}
\EE\sqa{ \bra{ I_{\cur{ \tau_{1,\ell}^i = 0}} \int_{0}^{\rho \land \tau_{1, L}^i} I_{\cur{B_s\in D_\ell}} ds}^q } \les  \ell^d \cdot \ell^{2q}.
\end{equation}
By Rosenthal's inequality for positive variables \eqref{eq:rosenthal-positive},
\begin{equation}\begin{split}\label{moment-lambda-0}
\EE\sqa{ \lambda_0(D_\ell)^q } & =  \EE\sqa{ \bra{ \sum_{i\in S_0} \int_{0}^{\rho \land \tau_{1, L}^i} I_{\cur{B_s\in D_\ell}} ds}^q }  \\
& \les n  \ell^d \cdot \ell^{2q} + \bra{  n  \ell^d \cdot \ell^2 }^{q} \ll n \ell^d \ell^{2q}.
\end{split}
\end{equation}
having used
\begin{equation}
n \ell^{d} \sim \frac{T}{\rho} \ell^d = T^{1-\gamma_\rho - d \gamma } \ll 1,
\end{equation}
which holds true if we assume that
\begin{equation}
\label{eq:condition-gamma-4-condition}
1- \gamma_\rho - d \gamma < 0.
\end{equation}
By Markov inequality,
\begin{equation}\label{eq:good-gk}
\PP(G_0^c)  \les_{\delta} \frac{ T\rho^{-1} \ell^d \ell^{2q} }{ (T \ell^{d})^q } = \frac{T\rho^{-1} \ell^d}{(T \ell^{d-2})^q } \les T^{-\alpha}  
\end{equation}
provided that $q$ chosen is large enough.

\stepcounter{step} \noindent \emph{Step \thestep \, (Removal of bad event).} In this step we argue that it is sufficient to show \eqref{eq:reduction-single-visit} where expectations are restricted to the event
\begin{equation}
A = G_1 \cap G_{0,\delta} \cap  \bigcap_{k=2}^{\bar{k}-1} G_{k,\delta} \cap H_{\bar {k}, \delta},
\end{equation}
for a suitable choice of $\bar{k}$. Indeed, since we can always bound the transportation costs by $T \ell^p$, the expectation over the event $A^c$ will contribute by a quantity not larger than
\begin{equation}
\PP(A^c) \cdot T \ell^p \ll T^{1-p/(d-2)} \ell^d 
\end{equation}
provided that 
\begin{equation}
 \PP(A^c) \ll T^{-p/(d-2)}\ell^{d-p} = T^{-p/(d-2) - \gamma(d-p)},
 \end{equation} 
 which can be achieved via the union bound and the inequalities \eqref{eq:good-g1}, \eqref{eq:good-gk} and \eqref{eq:good-hk} obtained in the previous step. Precisely, if we fix some $\alpha > p/(d-2) + \gamma(d-p)$, then any $\bar{k}$ such that \eqref{eq:good-hk} holds will be a possible choice. Then, for each $k < \bar{k}$, we can apply \eqref{eq:good-hk}.

\stepcounter{step} \noindent \emph{Step \thestep \, (Proof of \eqref{eq:reduction-single-visit} on $A$).} If $A$ holds, then we are in a position to apply \Cref{lem:same-asymptotics} on $\Omega = D_\ell$,  $\mu = \tilde{\lambda}_1$ and
\begin{equation}
\lambda =  \lambda_0 + \sum_{k=2}^{\bar{k}-1} \lambda_k + \bra{\sum_{h \ge \bar{k}} \lambda_h}\end{equation}
To keep the exposition simple, let us argue in the case $p>d/(d-1)$ (the case $p\le d/(d-1)$ is discussed in a separate step below). Taking expectation, we see that to obtain \eqref{eq:reduction-single-visit} it sufficient to show, with $\nu =\lambda_0$, $\nu = \lambda_k$ for $k\in \cur{2, \ldots, \bar{k}-1}$, and $\nu = \sum_{h \ge \bar{k}} \lambda_h$,
\begin{equation}\label{eq:condition-nu-small-contribution}
\EE\sqa{ I_A \nu(D_\ell) \bra{\frac{ \nu(D_\ell)}{\delta \tilde{\lambda}_1(D_\ell)/\bar{k}}}^{p/d}} \le \ell^p \EE\sqa{ I_A\nu(D_\ell) \bra{\frac{ \nu(D_\ell)}{\delta c T \ell^{d}/\bar{k}}}^{p/d}} \ll T^{1-p/(d-2)} \ell^d,
\end{equation}
 having used that $\mu(D_\ell) = \tilde{\lambda}_1(D_\ell) \ge c T \ell^d$ on $A$. 
 If $\nu= \lambda_k$ with $k \ge 2$, then by \eqref{moment-lambda-k} with $q= 1+p/d$, we have
 \begin{equation} 
 T^{-p/d} \EE\sqa{ \lambda_k(D_\ell)^{1+p/d} } \les T^{1-p/d}\ell^{d} (\rho \ell^{d-2})^{k-1} \ell^{2p/d} \ll T^{1-p/(d-2)} \ell^d
 \end{equation}
 provided that we impose the condition
 \begin{equation}
 \label{eq:condition-gamma-5}
 (k-1)[ \gamma (d-2) - \gamma_\rho] + 2\gamma p/d > 2 p/(d(d-2)).
 \end{equation}
 Notice that since the square bracket is positive, it is sufficient to ensure it holds for $k=2$.
 
  If instead $\nu = \lambda_0$, by \eqref{moment-lambda-0} again with $q= 1+p/d$, we have
 \begin{equation}
   T^{-p/d} \EE\sqa{ \lambda_0(D_\ell)^{1+p/d} } \les T^{1-p/d}\rho^{-1}\ell^{d}  \ell^{2+2p/d} \ll T^{1-p/(d-2)} \ell^d
 \end{equation}
provided that the additional inequality holds:
\begin{equation} \label{eq:condition-gamma-lambda-6}
 -\gamma_\rho +2 \gamma + 2 \gamma p/d >2 p/(d(d-2)).
\end{equation}

Finally, if $\nu = \sum_{h\ge \bar{k}} \lambda_h$, by \eqref{eq:moment-sum-h-gtrk} with $q=1+p/d$, we have
\begin{equation}
 T^{1-p/d} \EE\sqa{ \bra{ \sum_{ h \ge \bar{k}} \lambda_h(D_\ell)}^{1+p/d}} \les T \ell^{d-2} (\rho \ell^{d-2})^{\bar{k}-1} \rho^{1+p/d} \ll T^{1-p/(d-2)} \ell^d
\end{equation}
provided that we impose
\begin{equation}
(\bar{k}-1) (\gamma (d-2) -\gamma_\rho) - 2 \gamma - \gamma_\rho(1+p/d)  > p/(d-2),
\end{equation}
which can be satisfied by choosing $\bar{k}$ sufficiently large. This settles the validity of \eqref{eq:reduction-single-visit}.

\stepcounter{step} \noindent \emph{Step \thestep \, (Adjusting for longer first visits).} Our next aim is to replace the measure $\tilde{\lambda}_1$, where the first visit is interrupted at time $\rho$, with the measure
\begin{equation}
\sum_{i \in S} \int_{\tau_{1, \ell}^i}^{\tau_{1, L}^i} \delta_{B_s^i} ds = \tilde{\lambda}_1 + \nu,
\end{equation}
where we let
\begin{equation}
\nu = \sum_{i\in S}\int_{\rho\land \tau_{1, L}^i} ^{\tau_{1,L}^i} \delta_{B_s^i} ds.
\end{equation}
This is achieved by proving
\begin{equation}\label{eq:adjusted-single-visits}
\abs{ \EE\sqa{  W^p_{D_\ell}\bra{ \tilde{\lambda}_1} } - \EE\sqa{ W_{D_\ell}^p\bra{\sum_{i \in S} \int_{\tau_{1, \ell}^i}^{\tau_{1, L}^i} \delta_{B_s^i} ds  }}} \ll  T^{1-p/(d-2)} \ell^d,
 \end{equation}
again as a consequence of \Cref{lem:same-asymptotics}. We first argue that, for every $q\ge 1$,
\begin{equation}\label{eq:moment-bound-long-visits}
\EE\sqa{ I_{A^i} \bra{ \int_{\rho\land \tau_{1, L}^i} ^{\tau_{1,L}^i} I_{\cur{B_s^i \in D_\ell}} ds}^q} \les L^{2} |\log \ell | \ell^{d-2} \cdot \ell^{2q}.
\end{equation}
Indeed, let us introduce a parameter $0\le \tilde{\rho}\le \rho$ and split the event into two alternatives:
\begin{equation}\label{eq:alternatives}
A^i = \cur{0<\tau_{1,\ell}^i \le \rho} = \cur{0< \tau_{1, \ell}^i \le \tilde{\rho} } \cup  \cur{ \tilde{\rho} < \tau_{1, \ell}^i  \le \rho}
\end{equation}
In the case the first alternative holds, then it must be $\tau_{1, L}^i-\tau_{1, \ell}^i > \rho -\tilde{\rho}$ in order to have a non null contribution. Hence, by using \Cref{prop:hitting-times} and the strong Markov property, writing $\tilde{B}^i = \theta_{\tau_{1,\ell}^i} B^i$ we bound from above
\begin{equation}\begin{split}
& \EE\sqa{  I_{\cur{0< \tau_{1, \ell}^i \le \tilde{\rho}}}\bra{  \int_{\rho\land \tau_{1, L}^i} ^{\tau_{1,L}^i} I_{\cur{B_s^i \in D_\ell}} ds}^q} \\
& \qquad \le (\rho \ell^{d-2})  \EE\sqa{ I_{ \cur{\tau_{D_L^c} > \rho -\tilde{\rho} }} \bra{\int_{0}^{\tau_{D_L^c}  } I_{\cur{\tilde B_s^i \in D_\ell}} ds}^q}\\
& \qquad \les (\rho \ell^{d-2})  \PP\bra{ \tau_{D_L^c} > \rho -\tilde{\rho} }^{1/2} \EE\sqa{ \bra{\int_{0}^{\tau_{D_L^c}  } I_{\cur{\tilde B_s^i \in D_\ell}} ds}^{2q}}^{1/2}\\
& \qquad \les_q \rho \ell^{d-2} \exp\bra{ - c \frac{\rho - \tilde{\rho}}{L^2}} \ell^{2q},
\end{split}
\end{equation}
having used \Cref{lem:exit-time} and \eqref{eq:bm-p-power} in the last line. If instead the second alternative holds in \eqref{eq:alternatives}, we bound from above similarly, but using only that its probability is small -- again by \Cref{lem:exit-time}:
\begin{equation}\begin{split}
\EE\sqa{  I_{\cur{0< \tau_{1, \ell}^i \le \tilde{\rho}}}\bra{  \int_{\rho\land \tau_{1, L}^i} ^{\tau_{1,L}^i} I_{\cur{B_s^i \in D_\ell}} ds}^q} &\le (\rho-\tilde{\rho}) \ell^{d-2}  \EE\sqa{  \bra{\int_{0}^{\tau_{D_L^c}  } I_{\cur{\tilde B_s^i \in D_\ell}} ds}^q}\\
& \les_q (\rho-\tilde{\rho}) \ell^{d-2} \cdot \ell^{2q}.
\end{split}
\end{equation}
Recalling that $L = T^{-\gamma_L}$, letting $\tilde{\rho}=\rho - \tilde c L^{2} \log (1/L)$ for a large enough constant $\tilde{c}$ yields \eqref{eq:moment-bound-long-visits}. Rosenthal's inequality for positive variables \eqref{eq:rosenthal-positive} entails
\begin{equation}\label{eq:moment-nu-q}
\EE\sqa{ \nu(D_\ell)^q } \les n L^{2} |\log \ell | \ell^{d-2} \cdot \ell^{2q} + \bra{n L^{2} |\log \ell | \ell^{d-1}}^q \ell^{2q} \les T \rho^{-1} L^2 |\log \ell| \ell^{d-2} \cdot \ell^{2q},
\end{equation}
having used  
\begin{equation}
n L^{2} |\log \ell | \ell^{d-2} \sim T \rho^{-1} L^2 \ell^{d-2} | \log \ell | \ll 1
\end{equation}
provided that the following condition holds:
\begin{equation}
\label{eq:condition-gamma-7}
1-\gamma_\rho -2 \gamma_L-(d-2) \gamma_\ell<0.
\end{equation}
Starting from \eqref{eq:moment-nu-q}, we now argue similarly as in the previous steps. Given $\delta \in (0,1)$, we introduce the event
\begin{equation}
\tilde{G}_1 = \tilde{G}_{1,\delta} = \cur{ \nu(D_\ell) \le \delta T \ell^d},
\end{equation}
and argue by Markov inequality that
\begin{equation}
\PP\bra{ \tilde{G}_1^c } \les_\delta  \frac{T \rho^{-1} L^2 |\log \ell| \ell^{d-2} \cdot \ell^{2q}}{(T \ell ^{d-2})^q\cdot \ell^{2q}} \ll T^{-\alpha},
\end{equation}
 for every $\alpha>0$, provided that $q$ is chosen large enough. Using this fact, we may reduce  the proof of \eqref{eq:adjusted-single-visits} where the expectations are performed on the event 
 \begin{equation}
 A = G_1 \cap \tilde{G}_1
 \end{equation}
 with $G_1$ as in \eqref{eq:good-g1}. After applying \Cref{lem:same-asymptotics}, everything is once again reduced to prove that
 \begin{equation}
 T^{-p/d} \EE\sqa{ \nu(D_\ell)^{1+p/d}}  \ll T^{1-p/(d-2)} \ell^d,
 \end{equation}
 which is seen to be a consequence of \eqref{eq:moment-nu-q} with $q= 1+p/d$, provided that
\begin{equation}
\label{eq:condition-gamma-8}
 -\gamma_\rho +2 \gamma_L + 2 \gamma p/d >2 p/(d(d-2)).
\end{equation}

\stepcounter{step} \noindent \emph{Step \thestep \, (Application of \Cref{lem:stability})} Having settled \eqref{eq:adjusted-single-visits}, the thesis is reduced to prove that
\begin{equation}
\lim_{T \to \infty} \EE\sqa{  W^p_{D_\ell}\bra{\sum_{i \in S} \int_{\tau_{1, \ell}^i}^{\tau_{1, L}^i} \delta_{B^i_s} ds}} / \bra{ T^{1-p/(d-2)} |D_\ell|}= \c(\cI, d,p).
 \end{equation}
 Let us define $\tilde B^i:= \theta_{\tau^i_{1, \ell}}B^i$ -- beware that this is not the same process as \eqref{eq:tilde-b-i} -- so that we can rewrite it as
 \begin{equation}\label{eq:to-apply-lemma-stability}
\lim_{T \to \infty} \EE\sqa{  W^p_{D_\ell}\bra{\sum_{i \in S} \mu^{\tilde{B}^i}_{\tau_{D_L^c}\tilde{B}^i}}} / \bra{ T^{1-p/(d-2)} |D_\ell|}= \c(\cI, d,p).
 \end{equation}
We condition upon the event $\cur{S = s}$ for a subset $s \subseteq \cur{1, \ldots, n}$, i.e., the event $A_i$ holds if and only if $i \in s$. By \Cref{prop:hitting-times}, the conditional hitting law $\nu_\rho$ of $B^i_{\tau^i_{1, \ell}} = \tilde{B}^i_0$ is quantitatively close to the uniform distribution on $\partial D_\ell$.  Moreover, each $\tilde{B}^i$ is still a Brownian motion on $\T^d$ and independence among the $\tilde{B}^i$'s is preserved.  Furthermore, since we are considering $\tilde{B}^i$ up to $\tau_{D_L^c}\tilde{B}^i$, we can safely lift each of them to a Brownian motion on $\R^d$, without changing the optimal transport cost on $D_\ell$. 
 
 Applying \eqref{eq:bm-p-power} with $q=1$, we easily obtain that
 \begin{equation}
 \EE\sqa{ \sum_{i \in S} \mu^{\tilde{B}^i}_{\tau_{D_L^c}B^i}(D_\ell) } \les n \cdot \rho \ell^{d-2} \ell^2 \les T \ell^d
 \end{equation}
 Thus,
 \begin{equation}
 \EE\sqa{ \sum_{i \in S} \mu^{\tilde{B}^i}_{\tau_{D_L^c}B^i}(D_\ell) } W_{D_\ell}^p\bra{\nu_\rho, \tilde{e}_{D_\ell} } \les T \ell^d \cdot \ell^p (\rho \ell^{d-2} ) |\log \ell| \ll T^{1-p/(d-2)} \ell^d,
 \end{equation}
 provided that we impose
 \begin{equation}\label{eq:condition-gamma-9}
 \gamma p + [\gamma (d-2)- \gamma_\rho ] > p/(d-2).
 \end{equation}
 By applying \Cref{lem:stability}, it follows that we can replace each $\tilde{B}^i$ in \eqref{eq:to-apply-lemma-stability} with a Brownian motion $\bar{B}^i$ with initial law  $\tilde{e}_{D_\ell}$, i.e., uniform on $\partial D_\ell$, and the summation is now performed over a random set $S$, independent of the processes $(\bar{B}^i)_{i=1}^n$.

\stepcounter{step} \noindent \emph{Step \thestep \, (Time enlargement up to $\infty$).} In order to apply \Cref{thm:main-bm-iid}, we need to enlarge the time integration from $\tau^i_{D_L^c}$ to $\infty$. This is achieved along the same lines as before, i.e.\ by isolating a good event and applying \Cref{lem:same-asymptotics} to prove 
\begin{equation}\label{eq:same-asymptotics-almost-final}
\abs{ \EE\sqa{  W^p_{D_\ell}\bra{\sum_{i \in S} \mu^{\bar{B}^i}_{\tau_{D_L^c}\bar B^i}}} - \EE\sqa{W^p_{D_\ell}\bra{\sum_{i \in S} \mu^{\bar{B}^i}_{\infty} }} } \ll T^{1-p/(d-2)} \ell^d. 
\end{equation}
Indeed, if we let
\begin{equation}
\bar{G}_1 = \cur{  \sum_{i \in S} \mu^{\bar{B}^i}_{\tau_{D_L^c}\bar{B}^i} (D_\ell) \ge c T\ell^{d} }
\end{equation}
arguing as in \eqref{eq:probability-good-g1} (it is actually simpler here because $S$ is now independent of the $\bar{B}^i$'s, hence there is no need of the strong Markov property), we obtain that for a suitable $c = c(d)>0$ it holds
\begin{equation}
\PP\bra{ \bar{G}_1^c} \les_\alpha T^{-\alpha}
\end{equation}
for every $\alpha>0$. Moreover, arguing as in the proof of \eqref{moment-lambda-k}, we obtain that, for $q \ge 1$,
\begin{equation}\begin{split}\label{eq:moment-final}
\EE\sqa{ \bra{  \sum_{i \in S} \mu^{ \theta_{\tau_{D_L^c}}\bar {B}^i} _{\infty} (D_\ell) }^q } & \les n \cdot  \rho \ell^{d-2} \cdot \bra{ \frac{\ell}{L}}^{d-2} \cdot \ell^{2q} + \bra{ n \cdot  \rho \ell^{d-2} \cdot \bra{ \frac{\ell}{L}}^{d-2} }^q \cdot \ell^{2q} \\
& \les T \ell^{d-2} \cdot \bra{\frac{\ell}{L}}^{d-2} \cdot \ell^{2q}
\end{split}
\end{equation}
provided that
\begin{equation}
T \ell^{d-2} \cdot \bra{\frac{\ell}{L}}^{d-2} \ll 1,
\end{equation}
which holds true if we impose the condition
\begin{equation}\label{eq:gamma-10}
 \gamma (d-2)  +(d-2) (\gamma_\ell - \gamma_L) >1.
\end{equation}
By Markov inequality, we see that letting 
\begin{equation}
\bar{G}_2 = \cur{ \sum_{i \in S} \mu^{ \theta_{\tau_{D_L^c}}\bar {B}^i} _{\infty} (D_\ell) \le \delta T\ell^d},
\end{equation}
it holds
\begin{equation}
\PP\bra{ \bar{G}_2^c} \les \frac{  T \ell^{d-2} \cdot \bra{\frac{\ell}{L}}^{d-2} \cdot \ell^{2q}}{(T \ell^{d-2})^q \ell^{2d} } \les T^{-\alpha}
\end{equation}
for every $\alpha>0$, provided that $q$ is chosen sufficiently large. Repeating the (by now) usual argument, we reduce first ourselves to the proof of \eqref{eq:same-asymptotics-almost-final} where expectations are performed on $A = \bar{G}_1 \cap \bar{G}_2$, and then applying \Cref{lem:same-asymptotics} we see that, in order to conclude it is sufficient to show
\begin{equation}\label{eq:final-final}
T^{-p/d}\EE\sqa{ \bra{ \sum_{i \in S} \mu^{ \theta_{\tau_{D_L^c}}\bar {B}^i} _{\infty} (D_\ell)  }^{1+p/d}} \ll T^{1-p/(d-2)} \ell^d.
\end{equation}
This easily follows from \eqref{eq:moment-final}, provided that the condition
\begin{equation}\label{eq:gamma-condition-11}
2 \gamma p/d + (d-2) (\gamma - \gamma_L) > 2p/(d(d-2))
\end{equation}
holds true.

\stepcounter{step} \noindent \emph{Step \thestep \, (Application of \Cref{thm:main-bm-iid}).} Thanks to \eqref{eq:same-asymptotics-almost-final}, the thesis is reduced to prove that
\begin{equation}
\lim_{T \to \infty} \EE\sqa{  W^p_{D_\ell}\bra{\sum_{i\in S} \mu_{\infty}^{\bar{B}^i} }}/T^{1-p/(d-2)}|D_\ell| = \c(\cI, d,p),
\end{equation}
where $(\bar{B}^i)_{i=1}^\infty$ are independent Brownian motions on $\R^d$, with common initial law $\tilde{e}_{D_\ell}$.  Recalling also the definition of $S$ (which is now independent of the Brownian motions), we can introduce the function
\begin{equation}
f(m) = \EE\sqa{  W^p_{D_1}\bra{ \sum_{i=1}^{m} \mu_{\infty }^{B^i} } },
\end{equation}
where $(B^i)_{i=1}^\infty$ are independent Brownian motions on $\R^d$ with common initial law $\tilde{e}_{D_1}$, so that a straightforward rescaling gives
\begin{equation}
\EE\sqa{  W^p_{D_\ell}\bra{\sum_{i\in S} \mu_{\tau_L^i}^{\tilde{B}^i} } \bigg| \sharp S = m} = \ell^{p+2}f(m).
\end{equation}
By \Cref{thm:main-bm-iid}, we have
 \begin{equation}
 \lim_{m \to \infty} f(m)/\bra{ \frac{m}{\Cap(D_1)}}^{1-p/(d-2)} = \c(\cI, d,p)|D_1|
 \end{equation}
 Moreover, $\sharp S$ has binomial law with parameters $n=n(T)\to \infty$, $p=p(T)\to 0$ satisfying 
 \begin{equation}
np = T\ell^{d-2}(\Cap(D_1)+o(1)).
 \end{equation}
 By \Cref{rem:debinom}, it follows that
 \begin{equation}
 \lim_{T \to \infty} \EE\sqa{ f(\sharp S)}/\bra{ \frac{\EE\sqa{\sharp S}}{\Cap(D_1)}}^{1-p/(d-2)} = \c(\cI, d,p)|D_1|.
 \end{equation}
 However, we also see that
 \begin{equation}
  \lim_{T \to \infty} T^{1-p/(d-2)} \ell^{d-2-p} / \bra{ \frac{\EE\sqa{\sharp S}}{\Cap(D_1)}}^{1-p/(d-2)} =  1,
 \end{equation}
hence \eqref{eq:final-final} holds.

\stepcounter{step} \noindent \emph{Step \thestep \, (Choice of the parameters).} In order to complete the proof, we need to recall the multiple conditions imposed on $\gamma$, $\gamma_\rho$, $\gamma_\sigma$ and $\gamma_L$ and check that they can be all satisfied. Let us report them here for clarity: we have \eqref{eq:gamma-ell-rho-first-condition} and \eqref{eq:q:gamma-ell-rho-third-condition}, which read
\begin{equation}
  0< \gamma_\sigma< \gamma_\rho  < \gamma(d-2) \quad \text{and} \quad 0<\gamma_L < \gamma.
\end{equation}
Then, we have \eqref{eq:q:gamma-ell-rho-second-condition} 
\begin{equation}
 \gamma_\rho - \gamma_\sigma   > p\bra{1/(d-2) - \gamma},  
 \end{equation}
  and the similar conditions \eqref{eq:gamma-condition-i-forgot}, \eqref{eq:condition-gamma-4-condition}, \eqref{eq:condition-gamma-7}
\begin{equation}
\begin{split}
1-\gamma(d-2) - [ (d-2) \gamma - \gamma_\rho] & <0,\\
\quad 1- \gamma_\rho - d \gamma &< 0,\\ 1-\gamma_\rho -2 \gamma_L-(d-2) \gamma_\ell&<0.
\end{split}
\end{equation}
Next, we have \eqref{eq:condition-gamma-5}, which must be satisfied for $k=2$ (for larger values of $k$ it will automatically follow):
\begin{equation}
 [ \gamma (d-2) - \gamma_\rho] + 2\gamma p/d > 2 p/(d(d-2)).
 \end{equation}
Similar conditions are \eqref{eq:condition-gamma-lambda-6}, \eqref{eq:condition-gamma-8} and  \eqref{eq:gamma-condition-11}:
\begin{equation}\begin{split}
 -\gamma_\rho +2 \gamma_L + 2 \gamma p/d &>2 p/(d(d-2))\\  -\gamma_\rho +2 \gamma_L + 2 \gamma p/d &>2 p/(d(d-2)) \\ 2 \gamma p/d + (d-2) (\gamma - \gamma_L) &> 2p/(d(d-2)).
 \end{split}
\end{equation}
Finally, we have \eqref{eq:condition-gamma-9} and \eqref{eq:gamma-10}
\begin{equation}
 \gamma p + [\gamma (d-2)- \gamma_\rho ] > p/(d-2), \quad  \gamma (d-2)  +(d-2) (\gamma_\ell - \gamma_L) >1.
\end{equation}
It is now an elementary exercise to check that all these conditions can be satisfied by choosing $\gamma$ sufficiently close (but smaller than) $1/(d-2)$, $0<\gamma_\sigma<\gamma_\rho$ sufficiently small, and $\gamma_L$ sufficiently close to $\gamma$. This concludes the proof in the case $p>d/(d-1)$. In the following step we remark how to modify the argument in the case $p\le d/(d-1)$.

\stepcounter{step} \noindent \emph{Step \thestep \, (The case $p\le d/(d-1)$).} In this case, we cannot apply directly \Cref{lem:same-asymptotics} but following \Cref{rem:density-helps} we need also to preliminarily fix an additional parameter $r = r(p,d)>1$ such that $pr >d/(d-1)$, and take into account that all the applications of \Cref{lem:same-asymptotics} will have the additional terms  \eqref{eq:corrections-density-helps}. For brevity, we do not perform explicitly all the computations but show e.g.\ what becomes of the first application of \Cref{lem:same-asymptotics} in Step 6. The bound \eqref{eq:condition-nu-small-contribution} must now be complemented with
\begin{equation}
\ell^p \EE\sqa{I_A \tilde{\lambda}_1(D_\ell)^{1-1/r-p/d}  \nu(D_\ell)^{1/r+p/d}} \ll T^{1-p/(d-2)} \ell^d,
\end{equation}
with $\nu = \lambda_0$, $\nu= \lambda_k$ for $k =2, \ldots, \bar{k}-1$, and $\nu = \sum_{h=\bar{k}}^\infty \lambda_k$. For simplicity, let us focus on the case $\nu = \lambda_k$. If $1/r+p/d\le 1$, then by H\"older inequality and \eqref{moment-lambda-k} with $q=1$ we obtain
\begin{equation}\begin{split}
\ell^p \EE\sqa{I_A \tilde{\lambda}_1(D_\ell)^{1-1/r-p/d}  \lambda_k(D_\ell)^{1/r+p/d}}& \le \ell^p \EE\sqa{ \tilde{\lambda}_1(D_\ell)}^{1-1/r-p/d} \EE\sqa{  \lambda_k(D_\ell)}^{1/r+p/d} \\
\les \ell^p (T\ell^d)^{1-1/r-p/d} (T\ell^d  (\rho \ell^{d-2})^{k-1} )^{1/r+p/d} \\
& = \ell^p T \ell^d (\rho \ell^{d-2})^{(k-1)(1/r+p/d) } \ll  T^{1-p/(d-2)} \ell^d,
\end{split}
\end{equation}
provided that we impose
\begin{equation}
 \gamma p + [\gamma (d-2) - \gamma_\rho] (k-1)(1/r+p/d)  > p/(d-2),
\end{equation}
which can be safely added to all the conditions found so far, again by choosing $\gamma$ sufficiently close (but smaller) to $1/(d-2)$ (recall that $k \ge 2$).  In the case $1/r+p/d >1$, we use the fact that on $A$ we have $\tilde{\lambda}_1(D_\ell) \gtrsim T \ell^d$, hence using \eqref{moment-lambda-k} with $q= 1/r+p/d$,
\begin{equation}\begin{split}
\ell^p \EE\sqa{I_A \tilde{\lambda}_1(D_\ell)^{1-1/r-p/d}  \lambda_k(\Omega)^{1/r+p/d}} & \les \ell^p (T \ell^d)^{1-1/r-p/d}\EE\sqa{ \lambda_k(\Omega)^{1/r+p/d} }\\
& \les  T^{-p/d} (T \ell^d)^{1-1/r} \cdot T \ell^{d-2}  (\rho \ell^{d-2})^{k-1} \ell^{2(1/r+p/d)}\\
& = T^{1-p/d} \ell^d (T \ell^{d-2})^{1-1/r} (\rho \ell^{d-2})^{k-1} \ell^{2p/d} \\
&  \ll T^{1-p/(d-2)} \ell^d,
\end{split}
\end{equation}
provided that we impose the condition
\begin{equation}
2 \gamma p/d - (1-1/r) (1-\gamma(d-2)) + (k-1) [\gamma(d-2) - \gamma_\rho] > 2p/(d(d-2)).
\end{equation}
Since $r$ is fixed and $k\ge 2$ we see that this holds, again provided that we choose $\gamma$ sufficiently close (but smaller) to $1/(d-2)$.
%
%
\end{proof}

We are now in a position to prove \Cref{thm:main-torus}.

\begin{proof}[Proof of \Cref{thm:main-torus}]
 Write for brevity $\mu_T = \mu_T^B$. Given $\bar{\gamma}$ as in \Cref{prop:local-torus}, we choose $\gamma \in (\bar{\gamma}, 1/(d-2))$, set $\ell =  T^{-\gamma}$ and consider the Borel family of sets $(D_\ell(z))_{z \in \T^d}$, which satisfies
\begin{equation}
\int_{\T^d} \chi_{D_\ell(z)} \frac{dz}{|D_{\ell} |} = 1.
\end{equation}
By \Cref{lem:sub} with $E = \T^d$ and $\sigma(dz) = |D_{\ell}|^{-1} dz$, we obtain
\begin{equation}\label{eq:main-sub-additivity}
W^p_{\T^d}\bra{\mu_T}\le (1+\eps)\int_{\T^d} W^p_{D_\ell(z)}\lt(\mu_T \rt) \frac{ d z}{|D_\ell|} + \frac{c}{\eps^{(p-1)^+}}W^p_{\T^d}\lt(\int_{\T^d} \frac{\mu_T(D_\ell(z))}{\abs{D_\ell}}\chi_{D_\ell(z)} \frac{ d z}{|D_\ell|} \rt).
\end{equation}
We are in a position to apply \Cref{prop:smoothed-wasserstein} and deduce that
\begin{equation}
 W^p_{\T^d}\lt(\int_{\T^d} \frac{\mu_T(D_\ell(z))}{\abs{D_\ell}}\chi_{D_\ell(z)} \frac{ d z}{|D_\ell|} \rt) \ll T^{1-p/(d-2)},
\end{equation}

Thus, to obtain the thesis we only need to focus on the first term in the right hand side of \eqref{eq:main-sub-additivity}. By stationarity, the law of $\mu_T \restr D_\ell(z)$ does not depend on $z \in \mathbb{T}^d$, hence
\begin{equation}
\EE\sqa{ \int_{\T^d} W^p_{D_\ell(z)}\lt(\mu_T \rt) \frac{ d z}{|D_\ell|}}  = \int_{\T^d} \EE\sqa{ W^p_{D_\ell(z)}\lt(\mu_T \rt) } \frac{ d z}{|D_\ell|}  = |D_\ell|^{-1} \EE\sqa{ W^p_{D_\ell}\lt(\mu_T \rt) },
\end{equation}
and the thesis follows by \Cref{prop:local-torus}.
\end{proof}

\section{A concentration result}\label{sec:concentration}

In this final section, we establish some concentration properties for the optimal transport cost associated to the occupation measure of a stationary Brownian motion on $\T^d$.

\begin{proposition}\label{prop:concentration}
Let $B$ be a stationary Brownian motion on $\T^d$, $d \ge 3$. Then, for every $p>0$ such that $p<(d-2)/3$ if $d \in \cur{3,4}$ or
\begin{equation}\label{eq:condition-concentration}
 p < (d-2) \cdot \frac{ 1 + p/d}{3+ p/d}, \quad \text{if $d \ge 5$,}
 \end{equation} it holds $\PP$-a.s.\
\begin{equation}\label{eq:concentration}
\lim_{T \to \infty} \abs{ W_{\T^d}^p(\mu_T^B) - \EE\sqa{ W_{\T^d}^p(\mu_T^B) }}/T^{1-p/(d-2)} = 0.
\end{equation}
\end{proposition}

Notice that $\eqref{eq:condition-concentration}$ is satisfied if $p\le (d-2)/3$. 

\begin{proof}
The strategy is to argue as in Step 1 and Step 2 in the proof of \Cref{prop:local-torus}, but on the whole $\T^d$, and then apply a standard concentration argument in the independent case. However, arguing globally we obtain a worse dependence on our bounds, eventually reducing the range of $p$'s for which the result is effective. To keep the derivation simple, we argue only in the case $p \ge 2$ (in particular, $d \ge 5$): to cover remaining cases one has e.g.\ to argue as in Step 12 of the proof of \Cref{prop:local-torus}.

\setcounter{step}{0}

\noindent\stepcounter{step}\emph{Step \thestep \, (Time splitting).} We introduce two parameters $\gamma_\rho$, $\gamma_\sigma$, such that
\begin{equation}\label{eq:gamma-ell-rho-first-condition-concentration}
  0< \gamma_\sigma< \gamma_\rho <1
 \end{equation}
 to be further specified below, and set 
 \begin{equation}
 n:= \lfloor T^{1-\gamma_\rho} \rfloor, \quad \sigma := T^{\gamma_\sigma}, \quad \rho := \frac{T}{n} - \sigma.
 \end{equation}
%
and decompose
\begin{equation}
\mu_T  = \sum_{i=0}^{n-1}  \mu_{\rho}^{\theta_{i(\rho+\sigma)}B} +  \sum_{i=0}^{n-1}\mu_{\sigma}^{\theta_{i(\rho+\sigma)+\rho}B} := \tilde{\mu}_T + \lambda_T 
\end{equation}
We trivially have 
\begin{equation}\label{eq:step-1-first-error-bound-global}
 \tilde{\mu}_T(\T^d) = n \rho \sim T \quad \text{and} \quad  \lambda_T (\T^d)= n \sigma \sim T^{1-\gamma_\rho + \gamma_\sigma},
\end{equation}
Thus,
\begin{equation}
\lambda_T(\T^d)/\tilde{\mu}_T(\T^d) = T^{\gamma_\sigma - \gamma_\rho} \ll 1
\end{equation} 
and we are in a position to apply \Cref{lem:same-asymptotics} with any given $\delta$ provided that $T$ is sufficiently large.
We find, for $\eps \in (0,1)$,
\begin{equation}\label{eq:trivial-concentration-sub}
 -W_{\T^d}^p(\tilde{\mu}_T) + (1+\eps)^{-1} W_{\T^d}^p\bra{\mu_T} \les_\eps  T^{1-(\gamma_\rho-\gamma_\sigma)(1+p/d)} \ll T^{1-p/(d-2)},
 \end{equation}
 provided that we impose
 \begin{equation}\label{eq:condition-gamma-rho-concentration-1}
 \gamma_\rho - \gamma_\sigma > \frac{ p}{d-2} \cdot \frac{d}{d+p}.
 \end{equation}
Using instead  \eqref{eq:same-asymptotics-trivial}, we obtain, for every $\eps$, $\delta$ sufficiently small, for some $c = c(\eps, p, d)$,
\begin{equation}\label{eq:trivial-concentration-sup}
W_{\T^d}^p(\tilde{\mu}_T) - (1+c\delta)(1+\eps) W_{\T^d}^p\bra{ \mu_T}   \les_{\delta, \eps}  T^{1-(\gamma_\rho-\gamma_\sigma)(1+p/d)}\ll T^{1-p/(d-2)},
\end{equation}
again if we impose  \eqref{eq:condition-gamma-rho-concentration-1}. Summing \eqref{eq:trivial-concentration-sub} with the expected value of \eqref{eq:trivial-concentration-sup} yields
\begin{equation}\begin{split}
& -  W_{\T^d}^p\bra{ \tilde{\mu}_T }  +\EE\sqa{ W_{\T^d}^p\bra{\tilde{ \mu}_T } } \\
& \quad + (1+\eps)^{-1} W_{\T^d}^p(\mu_T) -(1+c\delta)(1+\eps) \EE\sqa{W_{\T^d} (\tilde{\mu}_T)} \ll T^{1-p/(d-2)},
\end{split}
\end{equation}
which can be rewritten after some manipulations as
\begin{equation}
 \bra{ W_{\T^d}^p(\mu_T) - \EE\sqa{ W_{\T^d} (\tilde{\mu}_T) }} \les_{\eps} \bra{ W_{\T^d}^p\bra{ \tilde{\mu}_T }  + c \delta \EE\sqa{ W_{\T^d}^p\bra{\tilde{ \mu}_T }}}^+ +\EE\sqa{W_{\T^d} (\tilde{\mu}_T)} + R(T),
\end{equation}
where $R(T) \ll T^{1-p/(d-2)}$. Using \Cref{thm:main-torus} (actually in a weaker form), it holds
\begin{equation}\label{eq:weak-main-theorem-torus}
\EE\sqa{ W_{\T^d}^p(\mu_T)} \les T^{1-p/(d-2)},
\end{equation} 
hence, letting first $T \to \infty$ and then $\delta\to 0$ and finally $\eps \to 0$, we obtain that $\PP$-a.s.\ it holds,
\begin{equation}\begin{split}
& \limsup_{T \to \infty} \bra{ W_{\T^d}^p(\mu_T) - \EE\sqa{ W_{\T^d} (\mu_T) }}/T^{1-p/(d-2)} \\\
& \qquad \le \limsup_{T \to \infty} \bra{ W_{\T^d}^p(\tilde{\mu}_T) - \EE\sqa{ W_{\T^d} (\tilde{\mu}_T) }}^+/T^{1-p/(d-2)}.
\end{split}
\end{equation}
Arguing similarly, but  summing the expectation of \eqref{eq:trivial-concentration-sub} with  \eqref{eq:trivial-concentration-sup} yields, after some manipulations and using \eqref{eq:weak-main-theorem-torus}, that
\begin{equation}\begin{split}
& \limsup_{T \to \infty} \bra{ \EE\sqa{ W_{\T^d} (\mu_T) } - W_{\T^d}^p(\mu_T)}/T^{1-p/(d-2)} \\\
& \qquad \le \limsup_{T \to \infty} \bra{ \EE\sqa{ W_{\T^d} (\tilde{\mu}_T) - W_{\T^d}^p(\tilde{\mu}_T) }}^+/T^{1-p/(d-2)},
\end{split}
\end{equation}
hence, the thesis will follow if we prove that
%
\begin{equation}\label{eq:concentration-tilde}
\lim_{T \to \infty} \abs{ W_{\T^d}^p(\tilde{\mu}_T) - \EE\sqa{ W_{\T^d}^p(\tilde{\mu}_T) }}/T^{1-p/(d-2)} = 0.
\end{equation}

\stepcounter{step} \noindent \emph{Step \thestep \, (Breaking dependence).} We write
\begin{equation}
 \tilde{\mu}_T := \sum_{i=1}^n \mu_{\rho}^{\tilde{B}^i},
 \end{equation}
 where $\tilde{B}^i = \theta_{(i-1)(\rho+\sigma)} B$ are stationary (but not independent) Brownian motions on $\T^d$.  Our aim is to argue that we can replace the processes $\tilde{B}$ with independent (and stationary) Brownian motions $\tilde{B}^{'}$. Before we do so, we argue that we can work on a stronger version of \eqref{eq:concentration-tilde}, namely that for every $q \ge 2$, 
\begin{equation}\label{eq:concentration-lq}
\nor{  W_{\T^d}^p(\tilde{\mu}_T) - \EE\sqa{ W_{\T^d}^p(\tilde{\mu}_T) }}_{L^q} \les \rho T^{[1-p/(d-2)]/2},
\end{equation}
(where the implicit constants depends on $q$) which easily implies \eqref{eq:concentration-tilde} via Borel-Cantelli lemma, provided that we impose
\begin{equation}\label{eq:condition-gamma-rho-concentration-2}
 2 \gamma_\rho  < 1-p/(d-2).
 \end{equation} The interesting feature of \eqref{eq:concentration-lq} is that it can be equivalently restated in terms of an independent copy $\bar{B}$ of $\tilde{B}$:
\begin{equation}
\nor{  W_{\T^d}^p(\tilde{\mu}_T) - W_{\T^d}^p( \bar{\mu}_T) }_{L^q}\les \rho T^{[1-p/(d-2)]/2}
\end{equation}
with the notation
 \begin{equation}
 \bar{\mu}_T := \sum_{i=1}^n \mu_{\rho}^{\bar{B}^i},
 \end{equation}
 and  $\bar{B}^i = \theta_{(i-1)(\rho+\sigma)} \bar B$. 
  
%
%
We then consider the following  non-negative function of $2 n$ continuous curves $(\tilde x, \bar x) = (\tilde{x}^i, \bar{x}^i )_{i=1}^n$ defined on the interval $[0,\rho]$ taking values on $\T^d$:
\begin{equation}
(\tilde{ x}, \bar{x} ) \mapsto F \bra{\tilde x, \bar x} = \abs{  W_{\T^d}^p \bra{ \sum_{i=1}^{n}\mu_{\rho}^{x^i} } -  W_{\T^d}^p \bra{ \sum_{i=1}^{n}\mu_{\rho}^{x^i} }}^q
\end{equation}
By \eqref{eq:trivial-wass}, we have
\begin{equation}\label{eq:uniform-bound-big-F-concentration}
 \sup_{ (\tilde{ x}, \bar{x} ) } F \bra{\tilde{ x}, \bar{x} } \le T^q.
 \end{equation} 
By $2n$ iterated applications of \eqref{eq:markov-chain-tv}, we obtain
\begin{equation}
\TV\bra{ \PP_{(\tilde{B}^i, \bar{B}^i)_{i=1}^n}, \bra{ \PP_{B_{s \in [0,\rho]}} \otimes   \PP_{B_{s \in [0,\rho]}}}^{\otimes n} } \les n e^{-c \sigma}.
\end{equation}
This yields in particular that we can replace the (dependent) Brownian motions $\tilde{B}$ and $\bar{B}$ in the expectation of $F$, with $2n$ stationary and independent Brownian motions, with an error term that is bounded from above by $T^q n e^{-c \sigma} \ll T^{-\alpha}$, for every $\alpha>0$.

\stepcounter{step} \noindent \emph{Step \thestep \, (Concentration via Poincaré inequality).} We now are left with the task of showing that
\begin{equation}\label{eq:concentration-conclusion}
\nor{  W_{\T^d}^p \bra{ \sum_{i=1}^{n}\mu_{\rho}^{B^i} } - \EE\sqa{ W_{\T^d}^p \bra{ \sum_{i=1}^{n}\mu_{\rho}^{x^i} }}}_{L^q} \les \rho T^{[1-p/(d-2)]/2},
\end{equation}
where the Brownian motions $(B^i)_{i=1}^n$ are now independent, and stationary, with values in $\T^d$. Representing each $B^i = Z^i+\tilde{B}^i$, with $Z^i$ independent and uniform variables on $\T^d$ and $\tilde{B}^i$ the projection on $\T^d$ of a (independent) standard Brownian motion on $\R^d$, we can write
\begin{equation}
 W_{\T^d}^p \bra{ \sum_{i=1}^{n}\mu_{\rho}^{B^i} }  = G\bra{(Z^i)_{i=1}^n, (\tilde{B}^i)_{i=1}^n }
\end{equation}
where
\begin{equation}
G((z^i)_{i=1}^n, (\tilde{x}^i)_{i=1}^n ) = W_{\T^d}^p \bra{ \sum_{i=1}^{n}\mu_{\rho}^{z^i+\tilde{x}^i} }. 
\end{equation}
The concentration bound \eqref{eq:concentration-conclusion} follows from the Poincaré inequality on the space $(\T^d)^{n} \times \bra{ C_{0}([0,\rho], \R^d)}^n$, endowed with a suitable product of uniform measure on $\T^d$ and Wiener measure on $C_{0}([0,\rho], \R^d)$. Indeed, since $W_{\T^d}^p\bra{\cdot}$ is expressed as minimization problems, a simple approximation argument (e.g.\ by discretizing the paths $x^i$) yields that
\begin{equation}
(z^i)_{i=1}^n, (\tilde{x}^i)_{i=1}^n \mapsto G((z^i)_{i=1}^n, (\tilde{x}^i)_{i=1}^n )
\end{equation}
is Sobolev (in the sense of Gaussian-Malliavin calculus with respect to the variables $\tilde{x}^i$'s), with (squared) modulus of the gradient estimated by
\begin{equation}\label{eq:squared-modulus-gradient}
\abs{ \nabla G ((z^i)_{i=1}^n, (\tilde{x}^i)_{i=1}^n ) }^2 \les \rho^2 \int_{\T^d} \dist(x,y)^{2(p-1)} d \pi^*(dx,dy) \le \rho^2 \int_{\T^d} \dist(x,y)^{p} d \pi^*(dx,dy)
\end{equation}
where $\pi^*$ denote an optimal transport plan (having used here that $p \ge 2$). It follows that, for every $q \ge 2$,
\begin{equation}
\EE\sqa{ \abs{ \nabla G ((Z^i)_{i=1}^n, (\tilde{B}^i)_{i=1}^n ) }^{q} } \les \rho^{q} \EE\sqa{  \bra{ W_{\T^d}^p \bra{ \sum_{i=1}^{n}\mu_{\rho}^{B^i} }}^{q/2}  }.
\end{equation}
By \eqref{eq:higer-r-wass} with $r = q/2$ and \eqref{eq:weak-main-theorem-torus}, with exponent $pq/2$,
\begin{equation}
\EE\sqa{  \bra{ W_{\T^d}^p \bra{ \sum_{i=1}^{n}\mu_{\rho}^{B^i }}^{q/2}  }} \les T^{q/2-1} \EE\sqa{ W_{\T^d}^{qp/2}\bra{ \sum_{i=1}^{n}\mu_{\rho}^{B^i }} }.
\end{equation}
Next, we use \Cref{prop:upper-bound-sum-brownian-motions}, which yields
\begin{equation}
 \EE\sqa{ W_{\T^d}^{r}\bra{ \sum_{i=1}^{n}\mu_{\rho}^{B^i }} } \les_r (n\rho)^{1-r/(d-2)} \les_r T^{1-r/(d-2)}
\end{equation}
with $r = qp/2$. By Poincaré inequality, we see that \eqref{eq:concentration-conclusion} holds.

\stepcounter{step} \noindent \emph{Step \thestep \, (Conclusion).} To conclude, we need to choose $0<\gamma_\sigma<\gamma_\rho<1$ such that both \eqref{eq:condition-gamma-rho-concentration-1} and \eqref{eq:condition-gamma-rho-concentration-2} holds true:
\begin{equation}
\gamma_\rho- \gamma_\sigma > \frac{p}{d-2} \cdot \frac{d}{p+d}, \quad \text{and} \quad 2 \gamma_\rho  < 1-p/(d-2).
\end{equation}
Since $\gamma_\sigma$ can be arbitrarily small, these conditions can be satisfied provided that \eqref{eq:condition-concentration} holds.
\end{proof}

\begin{remark}[lower bounds]\label{prop:lower-bound}
In the proof of \Cref{thm:main-bm-iid} we did not prove directly that $\c(\cI, p, d)$ is strictly positive, because this can be seen as a consequence of \Cref{thm:main-torus} and the fact that
\begin{equation}
\liminf_{T \to \infty} \EE\sqa{  W^p_{\T^d}\bra{ \mu_T^B}}/  T^{1-p/(d-2)} > 0.
\end{equation}
If $p \ge 1$, this is an immediate consequence of \eqref{eq:mattesini}. For $0<p<1$, one can argue as follows. We start as in Step 1 of the above proof with parameters $\gamma_\rho$, $\gamma_\sigma$, such that
\begin{equation}
  0< \gamma_\sigma< \gamma_\rho <1
 \end{equation}
 to be further specified, and set again
 \begin{equation}
 n:= \lfloor T^{1-\gamma_\rho} \rfloor, \quad \sigma := T^{\gamma_\sigma}, \quad \rho := \frac{T}{n} - \sigma.
 \end{equation}
Decomposing
\begin{equation}
\mu_T  = \sum_{i=0}^{n-1}  \mu_{\rho}^{\theta_{i(\rho+\sigma)}B} +  \sum_{i=0}^{n-1}\mu_{\sigma}^{\theta_{i(\rho+\sigma)+\rho}B} := \tilde{\mu}_T + \lambda_T 
\end{equation}
 from \eqref{eq:same-asymptotics-trivial} with $\eps \to 0$, we have for some constant $c  =c(p)$, that
\begin{equation}
W^p_{\T^d}\bra{ \mu_T^B }  \ge W^p_{\T^d}\bra{ \mu_T^B }  - c T^{1-\gamma_\rho + \gamma_\sigma}.
\end{equation}
Assuming that 
\begin{equation}
\gamma_\rho -\gamma_\sigma > p/(d-2),
\end{equation}
it is then sufficient to argue that
\begin{equation}
 \liminf_{T \to \infty} \EE\sqa{ W^p_{\T^d}\bra{ \sum_{i=1}^n \mu_{\rho}^{\tilde{B}^i} } }/T^{1-p/(d-2)} >0.
\end{equation}
As in Step 2 of the above proof, we can use $n$ applications of \ref{eq:markov-chain-tv} to move to the case where the $\tilde{B}^i$'s are independent (and stationary) Brownian motions (this part uses that $\gamma_\sigma>0$). Hence, by \eqref{eq:trivial-lower-bound}, we have
\begin{equation}
\begin{split}
\EE\sqa{  W^p_{\T^d}\bra{ \sum_{i=1}^n \mu_{\rho}^{\tilde{B}^i} }} & \ge n \rho \int_{\T^d} \EE\sqa{ \min_{i=1, \ldots, n} \min_{t \in [0,\rho]} \dist_{\T^d}(x, \tilde B_t^i)^p } dx\\
& \gtrsim T \EE\sqa{ \min_{i=1, \ldots, n} \min_{t \in [0,\rho]} \dist_{\T^d}(0, \tilde B_t^i)^p } 
\end{split}
\end{equation}
having used stationarity in the last line. By the layer-cake formula and independence,
 \begin{equation}\begin{split}
 \EE\sqa{ \min_{i=1, \ldots, n} \min_{t \in [0,\rho]} \dist_{\T^d}(0, \tilde B_t^i)^p } & \gtrsim  \int_0^1 \PP\bra{ \min_{t \in [0,\rho]} \dist_{\T^d}(0, \tilde B_t^i)^p > s }^{n}  d s \\
 & =  \int_0^{T^{-p/(d-2)}}  \PP \bra{ \tau_{D_{s^{1/p} }} B^i > \rho }^{n}  d s 
 \end{split}
 \end{equation}
Using \Cref{prop:hitting-times}, we obtain that there exists a constant $c>0$ such that, for each $i=1,\ldots, n$, $s\le T^{-p/(d-2)}$, it holds
\begin{equation}
 \PP\bra{ \tau_{D_{s^{1/p} }B^i \ge \rho }} \ge 1- c\rho s^{(d-2)/p} = 1 - \frac{ c T s^{(d-2)/p}}{n},
\end{equation}
having also used that 
\begin{equation}
\rho  s^{(d-2)/p} \les \rho T^{-1} =T^{\gamma_\rho - 1} \ll 1.
\end{equation}
Using this bound, we find
 \begin{equation}
T \int_0^{T^{-p/(d-2)}}  \PP \bra{ \tau_{D_{s^{1/p} }} B^i > \rho }^{n}  d s  \gtrsim T \int_0^{T^{-p/(d-2)}} \bra{1- \frac{ c T s^{(d-2)/p} }{n}}^{n}  d s  \gtrsim T^{1-p/(d-2)},
 \end{equation}
hence the thesis.
\end{remark}
\appendix

\section{Wasserstein asymptotics for stationary measures}\label{app:asymptotics}

In this appendix, we consider a random Borel measure $\nu$ on $\R^d$ satisfying the following conditions:
\begin{enumerate}
\item[i)] (stationarity) For every $x \in \R^d$,  it holds $\tras_x \nu  =  \nu$ in law.
\item[ii)] \label{eq:integrability} (integrability) For every bounded Borel $A \subseteq \R^d$, $\nu(A)$ is integrable.
\item[iii)] (concentration) There exists $\alpha\in [0, d)$ such that, for every $q \ge 1$, there exists $C=C(\nu, q, \alpha)<\infty$ such that, for every Borel $A \subseteq \R^d$ with $\diam(A) \ge 1$, 
\begin{equation}\label{eq:concentration-abstract}
\| \nu(A) - \EE\sqa{ \nu(A) }\|_q \le C \diam(A)^{(d+\alpha)/2}.
\end{equation}

\end{enumerate}

Notice that, by stationarity and integrability, the function $A \mapsto \EE\sqa{ \nu(A)}$ is a translation invariant ($\sigma$-finite) measure hence for some constant $\lambda \in [0, \infty)$ it holds
\begin{equation}\label{eq:expected-measure}
\EE\sqa{ \nu(A) } = \lambda | A|.
\end{equation}
Without loss of generality, we assume in what follows that $\lambda=1$, so that $\EE\sqa{ \nu(A)} = |A|$ (if $\lambda=0$ the statements become trivial).
With this notation, we have the following result.

 \begin{theorem}\label{theo:domain}
Let $\nu$ be a random measure on $\R^d$ satisfying the conditions i), ii), iii) above. Let $p>0$ be such that 
\begin{equation}\label{eq:assumption-p-not-too-large}
r:= d-\alpha-2 \min\cur{p,1} >0.
\end{equation}
Let  $\Omega \subseteq \R^d$ be  a bounded connected domain with $C^2$ boundary (or $\Omega = Q$ a cube). Then, it holds
\begin{equation}
\lim_{n \to \infty}  \EE\sqa{ W_{\Omega}^p( \dil_{n^{-1/d}} \nu ) }/n^{1-p/d} = \c_{\nu, p} |\Omega|,
\end{equation}
where $\c_{\nu, p} \in [0, \infty)$ depends on $\nu$ and $p$ only.
 \end{theorem}

We split the proof in the next subsections, where we recall and slightly generalized ideas and tools from \cite{BaBo, goldman2021convergence, ambrosio2022quadratic, goldman2022optimal}.


%
%
 
 \subsection{The case of a cube}
 
 In this section, we focus on the following result.
 
 \begin{proposition}\label{theo:PoiLeb}
Let $\nu$ be a random measure on $\R^d$ satisfying the conditions i), ii), iii) above. Let $p>0$ be such that 
\begin{equation}
r:= d-\alpha-2 \min\cur{p,1} >0
\end{equation}
Then, the following limit exists:
\begin{equation}
\lim_{L \to \infty} \frac{1}{L^d} \EE\sqa{ W_{Q_L}^p( \nu ) } = \c_{\nu, p} \in [0, \infty).
\end{equation}
Moreover, there exists $C>0$ (depending on $\nu$ and $p$ only) such that for $L\ge 1$, 
\begin{equation}\label{eq:upboundPoiLeb}
 \c_{\nu, p} \le \frac{1}{L^d} \EE\sqa{ W_{Q_L}^p( \nu ) }  +\frac{C}{L^{\frac{r}{2}}}.
\end{equation}
 \end{proposition}
 
 \begin{proof}
  By standard sub-additivity (Fekete-type) arguments, e.g.\ \cite[Lemma 2.12]{goldman2022optimal}, it is sufficient to prove that there exists a constant $C>0$ such that for every $L\ge C$ and $m\in \N$,
  \begin{equation}\label{subad}
\frac{1}{|Q_{mL}|} \EE\sqa{ W_{Q_{mL}}^p( \nu ) } \le  \frac{1}{|Q_L|} \EE\sqa{ W_{Q_L}^p( \nu ) }  +\frac{C}{L^{\frac r 2 }}.
\end{equation}
 Starting from the cube $Q_{mL}$, we construct a sequence of finer and finer partitions of $Q_{mL}$ by rectangles of moderate aspect ratios and side-length given by integer multiples of $L$. 
 To simplify the notation, we define 
\begin{equation}\label{def:frefR}
 \fref(R)= \EE\lt[\frac{1}{|R|}W^p_R(\nu)\rt].
\end{equation}
We rely upon iterated applications of  the following inequality: let $R$ of moderate aspect ratio, let  $\mathcal{R}$ be an admissible partition $R$ into rectangles of moderate aspect ratios and side-lengths given by integer multiples of $L$. For every $\eps\in[0,1)$, we have 
 \begin{equation}\label{onestep}
  \fref(R)\le (1+\eps)\sum_i \frac{|R_i|}{|R|} \fref(R_i) +\frac{C}{\eps^{(p-1)^+}} \frac{1}{|R|^{q/(2d)}},
 \end{equation}
 with $C =C(p)\in (0, \infty)$ and
 \begin{equation}
 q = \begin{cases}  
  d-\alpha-2p &\text{if $0<p<1$,}\\
  p(d-\alpha-2)  & \text{if $p \ge 1$.}
 \end{cases}
 \end{equation}
  Indeed, defining $\kappa=\frac{\nu(R)}{|R|}$, $\kappa_i=\frac{\nu(R_i)}{|R_i|}$, it holds
\begin{equation}
 \EE\sqa{ \kappa_i } = \EE\sqa{ \kappa} = 1
 \end{equation} 
Using \eqref{eq:sub}, we get
\begin{equation}
  \fref(R)\le (1+\eps)\sum_i \frac{|R_i|}{|R|} \fref(R_i) +\frac{C}{\eps^{(p-1)^+}} \EE\lt[\frac{1}{|R|}W_{R}^p\lt(\sum_i \kappa_i \chi_{R_i},\kappa\rt)\rt].
\end{equation}
We then estimate the last term in the right hand side.
In the case $0<p<1$, we use \eqref{eq:less-trivial-wass}, obtaining
\begin{equation}\begin{split}
 \frac{1}{|R|}W^p_{R}\lt(\sum_{i} \kappa_i \chi_{R_i},\kappa\rt)&\les \frac{ |R|^{p/d}}{|R|} \int_{R}\sum_i |\kappa_i-\kappa| \chi_{R_i}\\
 &\les |R|^{p/d}\bra{ |\kappa-\EE\sqa{\kappa}|+\sum_i|\kappa_i-\EE\sqa{\kappa_i}|}.
\end{split}
\end{equation}
By \eqref{eq:concentration-abstract}, we have
\begin{equation}
 \max\cur{\EE\lt[|\kappa-\EE\sqa{\kappa}|\rt], \max_{i} \cur{ \EE\lt[|\kappa_i-\EE\sqa{\kappa_i}|\rt]}}\les |R|^{\frac{(\alpha-d)}{2d}},
\end{equation}
which eventually yields \eqref{onestep}.

If $p\ge 1$, we argue first by Markov inequalit and, \eqref{eq:concentration-abstract}  that for every $n \ge 1$,
\begin{equation}
\PP\bra{ \kappa \le \frac {1} 2}   \le \PP\bra{ \abs{ \kappa - \EE\sqa{ \kappa} } \ge \frac 1 2 }\les \| \nu(R) - \EE\sqa{ \nu(R) } \|_n^n |R|^{-n} = |R|^{n(\alpha-d)/(2d)}.
\end{equation}
Choosing $n$ sufficiently large and using \eqref{eq:trivial-wass}, i.e.,
\begin{equation}
\frac{1}{|R|}W^p_{R}\bra{\sum_{i} \kappa_i \chi_{R_i},\kappa}\les |R|^{p/d} \kappa,
\end{equation}
we may reduce ourselves to the event $\{\kappa\ge \frac{1}{2}\}$. 
Under this condition, by \eqref{eq:estimCZ}, we have 
\begin{equation}
\begin{split}
 \frac{1}{|R|}W^p_{R}\lt(\sum_{i} \kappa_i \chi_{R_i},\kappa\rt)&\les \frac{|R|^{p/d}}{|R|}\int_{R}\sum_i |\kappa_i-\kappa|^p \chi_{R_i}\\
 &\les |R|^{p/d} \lt( |\kappa-\EE\sqa{\kappa}|^p+\sum_i|\kappa_i-\EE\sqa{\kappa_i}|^p\rt).
\end{split}
\end{equation}
By \eqref{eq:concentration-abstract}, we have
\begin{equation}
 \max\cur{\EE\lt[|\kappa-\EE\sqa{\kappa}|^p\rt], \EE\lt[|\kappa_i-\EE\sqa{\kappa_i}|^p\rt]}\les |R|^{\frac{p(\alpha-d)}{2d}},
\end{equation}
which eventually yields \eqref{onestep}.
%
%

Starting from the cube $Q_{mL}$, we next construct a sequence of finer and finer partitions of $Q_{mL}$ inductively as follows. 
 We let $\mathcal{R}_0=\{Q_{mL}\}$. To define $\mathcal{R}_k$, let    $R\in \mathcal{R}_k$. Up to translation we may assume that $R=\prod_{i=1}^d [0, m_i L)$ for some $m_i\in \N$. We then split each interval $[0,m_i L)$ into $[0,\lfloor\frac{m_i}{2}\rfloor L)\cup[\lfloor\frac{m_i}{2}\rfloor L, m_i L)$. 
 It is readily seen that this induces an admissible partition of $R$. Let us point out that when $m_i=1$ for some $i$, the corresponding interval   $[0,\lfloor\frac{m_i}{2}\rfloor L)$ is empty.
 This procedure stops after a finite number of steps $K$ once $\mathcal{R}_K=\{Q_L+z_i, z_i\in [0,m-1]^d\}$. It is also readily seen that $2^{K-1}<m\le 2^K$ and that for every $k\in [0,K]$ and every $R\in \mathcal{R}_k$ we have $|R|\sim (2^{K-k} L)^d$.
 
If $0<p<1$, we prove by  a downward induction that,  for every $k\in [0,K]$ and every $R\in \mathcal{R}_{k}$,
\begin{equation}\label{induction-easy}
 \fref(R)\le \fref(Q_L)+ C L^{-\frac{r}{2}} \sum_{j=k}^K 2^{- (K-j)r/2}.
\end{equation}
This is clearly true for $k=K$. Assume that it holds true for $k+1$. Let $R\in \mathcal{R}_{k}$. If $0<p<1$, we apply \eqref{onestep} with $\eps=0$, obtaining
\begin{equation}
\begin{split}
\fref(R)&\le \sum_{R_i\in \mathcal{R}_{k+1}, R_i\subset R} \frac{|R_i|}{|R|} \fref(R_i) +  C \frac{1}{|R|^{q/(2d)}}\\
& \le \sum_{R_i\in \mathcal{R}_{k+1}, R_i\subset R} \frac{|R_i|}{|R|} \bra{ \fref(Q_L) + C L^{-r/2} \sum_{j=k+1}^K 2^{-(K-j)r/2} } + C L^{-r/2} 2^{-(K-k)r/2}\\
& = \fref(Q_L) + C L^{-r/2} \sum_{j=k}^K 2^{-(K-j)r/2}.
\end{split}
\end{equation}
Applying \eqref{induction-easy} with $k=0$, hence $R = Q_{mL}$ yields \eqref{subad}.

In the case $p \ge 1$, we prove a slightly more involved inequality: there exists a constant $\Lambda<\infty$ (depending on $p$ only) such that, for $R \in \mathcal{R}_k$,
\begin{equation}\label{induction}
 \fref(R)\le \fref(Q_L)+ \Lambda(1+\fref(Q_L))L^{-\frac{r}{2}} \sum_{j=k}^K 2^{- (K-j)r/2}.
\end{equation}
if $\eps= (2^{K-k} L)^{-r/2}\ll1$, we get 
\begin{equation}\begin{split}
 \fref(R)&\le (1+ \eps) \sum_{R_i\in \mathcal{R}_{k+1}, R_i\subset R} \frac{|R_i|}{|R|} \fref(R_i) + \frac{C}{\eps^{p-1}} \frac{1}{|R|^{\frac{pr}{2d}}}\\
 &\stackrel{\eqref{induction}}{\le} (1+\eps) \lt(\fref(Q_L)+ \Lambda(1+\fref(Q_L))L^{-\frac{r}{2}} \sum_{j=k+1}^K 2^{- (K-j)r/2}\rt) +  C (2^{K-k} L)^{-\frac{r}{2}}\\
 &\le  \fref(Q_L)+  \Lambda(1+\fref(Q_L))L^{-\frac{r}{2}}\\
 &\times\lt[\sum_{j=k+1}^K 2^{-(K- j)r/2}+2^{-(K-k)r/2}\lt( \frac{C}{\Lambda}+L^{-\frac{r}{2}} \sum_{j=k+1}^K 2^{- (K-j)r/2}  \rt)\rt].
\end{split}\end{equation}
If $L$ is large enough, then $(\sum_{j=k+1}^K 2^{- (K-j)r/2}) L^{-r/2}\le \frac{1}{2}$. Finally, choosing $\Lambda\ge 2C$ yields \eqref{induction}.\\
Applying \eqref{induction} to $R=Q_{mL}$ and using that $\sum_{j\ge 0} 2^{- jr/2}<\infty$, we get 
\begin{equation}
 \fref(mL)\le \fref(L)+ C(1+\fref(L)) \frac{1}{L^{\frac{r}{2}}}.
\end{equation}
Since $\fref(L)\les L^p$, writing that every $L\ge C$ may be written as $L=mL'$ for some $m\in \N$ and $L'\in[C,C+1]$, we conclude that $\fref(L)$ is bounded and thus \eqref{subad} follows also in this case.
\end{proof}

\begin{remark}
 We point out that as a consequence of the above result, we have
 \begin{equation}\label{eq:boundfR}
 \EE\sqa{ W_Q^p(\nu)}\les |Q|, \quad \text{ for every cube $Q \subseteq \R^d$.}\end{equation}

\end{remark}

%

\subsection{Whitney decomposition}

We recall some result on decomposition of domains. The first one is \cite[Lemma 5.1]{ambrosio2022quadratic}. 


\begin{lemma}\label{lem:decomp}
Let $\Omega\subset \R^d$ be a bounded domain with Lipschitz boundary and let $\cQ = \{Q_i\}_i$ be a Whitney partition of $\Omega$. Then, for every $\delta>0$ sufficiently small, letting $\cQ_\delta=\{Q_i \ : \ \diam(Q_i) \ge \delta\}$, there exists a finite family $\cR_\delta=\{\Omega_j\}_j$ of disjoint open sets such that:
\begin{enumerate}
\item \label{partition-1} $(\Omega_k)_{k=1}^K =  \cQ_\delta \cup \cR_\delta$ is a partition of $\Omega$,
\item \label{partition-2}$ |\Omega_k| \sim \diam(\Omega_k)^d$ for every $k=1, \ldots, K$,
\item \label{partition-3}if $\Omega_k \in \cQ_\delta$, then $\diam(\Omega_k) \sim \dist(x, \Omega^c)$ for every $x \in \Omega_k$,
\item \label{partition-4}if $\Omega_k \in \cR_\delta$, then $\diam(\Omega_k)\sim \delta$ and $\dist(x, \Omega^c) \les \delta$, for every $x \in \Omega_k$.
\end{enumerate} 
Here all the implicit constants depend only on the initial partition $\cQ$ (and not on $\delta$).
\end{lemma}

%
We next collect some useful bounds related to the above construction, generalizing \cite[Lemma 2.2]{goldman2022optimal}. Define, for $\gamma \in \R$, $\delta\in (0,1/2)$, the function
\begin{equation}\label{eq:r-d-alpha}
 r_{\gamma}(\delta):= \begin{cases} 
1 & \text{if $\gamma>0$,}\\
|\log \delta| & \text{if $\gamma =0$,}\\
\delta^{\gamma} & \text{if $\gamma <0$,}
\end{cases}
 \end{equation} 
 and notice that 
 \begin{equation}
 r_{\gamma}(\delta) \sim \sum_{\ell \le |\log_2 \delta |} 2^{- \ell \gamma }.
 \end{equation}
 
Then, we have the following result.

\begin{lemma}\label{lem:bound-partition}
Let $\Omega\subset \R^d$ be a bounded domain with Lipschitz boundary and let $\cQ = \{Q_i\}_i$ be a Whitney partition of $\Omega$. Then, letting $(\Omega_k)_{k=1}^K = \cQ_\delta \cup \cR_\delta$ as in \cref{lem:decomp}, one has that $|\cR_\delta| \les \delta^{1-d}$ and,
\begin{enumerate}
\item  for every $\alpha \in \R$, it holds
 \begin{equation}\label{eq:whitney-general-q}
 \sum_{k=1}^K \diam(\Omega_k)^{\alpha} \les r_{\alpha+1-d}(\delta),
\end{equation}
\item for every $\alpha<0$, $k =1,\ldots, K$, and $x \in \Omega_k$, it holds
\begin{equation}\label{eq:final-bound-min-sum-partition}
\sum_{j} \diam(\Omega_j)^{\alpha} \min\cur{1, \bra{ \frac{\diam(\Omega_j)}{\dist(x, \Omega_j)}}^{d-1}} \les \bra{|\log(\delta)| + |\log(\diam(\Omega_k))|} \delta^{\alpha}.
\end{equation}

\end{enumerate}
In the above inequalities the implicit constants depend upon $\cQ$, $d$ and $\alpha$ only.
\end{lemma}

\begin{proof}
Since the boundary of $\Omega$ is Lipschitz, it follows from the properties of the partition that, for every $x \in \Omega$ and $r \ge s \ge \delta$, 
\begin{equation}\label{eq:uniform-bound-omega-k} \sharp { \cur{k\, : \, \Omega_k\subseteq B(x, r), \diam(\Omega_k) \in [s, 2s)}} \les (r/s)^{d-1},\end{equation}
 with the implicit constant depending on $\cQ$ only.
 It follows that
 $|\cR_\delta| \les \delta^{1-d}$ and, for every $\ell \le |\log_2 \delta|$, the number of cubes $\Omega_k \in \cQ_\delta$ with $\diam(\Omega_k) \in [2^{-\ell}, 2^{-\ell+1})$ is estimated by $ 2^{\ell (d-1)}$. Therefore, for $\alpha \in \R$,
\begin{equation}
\begin{split}  \sum_{k=1}^K \diam(\Omega_k)^{\alpha} & \les \sum_{\Omega_k \in \cQ_{\delta}} \diam(\Omega_k)^{\alpha}  + \sum_{\Omega_k \in \cR_\delta} \diam(\Omega_k)^{\alpha} \\
&  \les  \sum_{\ell \le |\log_2 \delta|} \sharp {\cur{ \Omega_k \in \cQ_{\delta} : \diam(Q_k) \in [2^{-\ell}, 2^{-\ell+1})}}  2^{-\ell \alpha} + |\cR_\delta| \cdot \delta^{\alpha}\\
 & \les \sum_{\ell \le |\log_2 \delta|} 2^{-\ell(\alpha+1-d)} + \delta^{\alpha+1-d} \les r_{\alpha+1-d}(\delta).\end{split} 
\end{equation}
having used that the in the summation $\ell$ is also bounded from below in the summation (e.g.\ by $-|\log_2\diam(\Omega)|$). We thus obtain \eqref{eq:whitney-general-q}.



We now prove \eqref{eq:final-bound-min-sum-partition}. We claim that it follows from the following inequalities, valid for any $\gamma \in \mathbb{N}$:
\begin{equation}\label{eq:before-last-claim}
\sum_{j\, : \, \dist(x, \Omega_j) \le 2^{-\gamma} \diam(\Omega_k)} \diam(\Omega_j)^\alpha \les 2^{-\gamma(d-1)} \diam(\Omega_k)^{d-1}r_{\alpha+1-d}(\delta),
\end{equation}
and
\begin{equation}\label{lastclaim}
   \sum_{j \, : \, \dist(x,\Omega_j)>2^{-\gamma} \diam(\Omega_k)} \frac{\diam(\Omega_j)^\beta }{\dist(x, \Omega_j)^{d-1}} \les |\gamma+ \log \bra{\diam(\Omega_k)}| r_{\beta+1-d}(\delta).
\end{equation}
Indeed, we can split the summation and use \eqref{eq:before-last-claim} and \eqref{lastclaim} to get
\begin{equation}\begin{split}
\sum_{j}  &  \diam(\Omega_j)^{\alpha} \min\cur{1, \bra{ \frac{\diam(\Omega_j)}{\dist(x, \Omega_j)}}^{d-1}} \\
&  \les  \sum_{j\, : \, \dist(x, \Omega_j) \le 2^{-\gamma} \diam(\Omega_k)} \diam(\Omega_j)^\alpha  +  \sum_{j\, : \, \dist(x, \Omega_j) > 2^{-\gamma} \diam(\Omega_k)} \frac{\diam(\Omega_j)^{d-1+\alpha}}{\dist(x, \Omega_j)^{d-1}}\\
& \les 2^{-\gamma(d-1)} \diam(\Omega_k)^{d-1} r_{\alpha+1-d}(\delta) + |\gamma+ \log \bra{\diam(\Omega_k)}| r_{\alpha}(\delta).
\end{split}
\end{equation}
Choosing $\gamma$ so that $2^{-\gamma} \le \delta \le 2^{-\gamma+1}$ yields \eqref{eq:final-bound-min-sum-partition}.

In order to prove \eqref{eq:before-last-claim} and \eqref{lastclaim} we first notice that, given $\Omega_k$, $\Omega_j$ and $x \in \Omega_k$, we have that, for some constant $C = C(\cQ)$,
\begin{equation}\label{eq:omega-j-contained-ball} \Omega_j \subseteq B(x, C \max\cur{\dist(x, \Omega_j), \diam(\Omega_k)}).\end{equation}
Indeed, if $\Omega_j \in \cR_\delta$, then $\diam(\Omega_j)\les \delta\les \diam(\Omega_k)$, hence \eqref{eq:omega-j-contained-ball} holds.  
If instead $\Omega_j \in \cQ_\delta$, then we can find $y \in \Omega_j$ with $|x-y|\le 2 \dist(x,\Omega_j)$, so that, by the triangle inequality,
\begin{equation}
\dist(y, \Omega^c) \le |x-y| + \dist(x, \Omega^c) \les \max\cur{\dist(x, \Omega_j), \diam(\Omega_k)}
\end{equation}
and by property (\ref{partition-3}) in \Cref{lem:decomp} we obtain that $\diam(\Omega_j) \les \max\cur{\dist(x, \Omega_j), \diam(\Omega_k)}$, yielding again the desired inclusion.

Hence, we prove \eqref{eq:before-last-claim} and \eqref{lastclaim}. Let $\ell_k \le |\log_2 \delta|$ be such that $\diam(\Omega_{k}) \in [2^{-\ell_k}, 2^{-\ell_k+1})$.  Combining \eqref{eq:omega-j-contained-ball} and \eqref{eq:uniform-bound-omega-k}, we see that, for every $\ell \le |\log_2\delta|$, there are  at most $2^{(\ell-\ell_k-\gamma)(d-1)}$ sets $\Omega_j$ such that $\dist(x, \Omega_j) \le 2^{-\gamma}\diam(\Omega_k)$ and $\diam(\Omega_j) \in [2^{-\ell}, 2^{-\ell+1})$. Therefore,
\begin{equation}
 \begin{split} \sum_{j\, : \, \dist(x, \Omega_j) \le 2^{-\gamma} \diam(\Omega_k)} \diam(\Omega_j)^{\alpha} & \les \sum_{\ell \le |\log_2\delta|} 2^{-\ell \alpha} 2^{(\ell-\ell_k)(d-1)} \\
& \les 2^{-(\gamma+\ell_k)(d-1)} \sum_{\ell \le |\log_2\delta|}  2^{-\ell(\alpha+1-d)} \\ &   \les  2^{-\gamma(d-1)}\diam(\Omega_k)^{d-1} r_{\alpha+1-d}(\delta).
\end{split}
\end{equation}
This proves \eqref{eq:before-last-claim}. To prove \eqref{lastclaim}, we split dyadically,
\begin{equation}\label{eq:final-claim-step}\begin{split}
 \sum_{j \, : \, \dist(x,\Omega_j)>2^{-\gamma}\diam(\Omega_k)}  \frac{\diam(\Omega_j)^\beta }{d(x,\Omega_j)^{d-1}} & \les \sum_{\ell \le \ell_k+\gamma} \frac{1}{(2^{-\ell})^{d-1}} \sum_{ j\, : \, d(x,\Omega_j) \in [2^{-\ell}, 2^{-\ell+1})} \diam(\Omega_j)^\beta\\ 
 & \stackrel{\eqref{eq:omega-j-contained-ball}}{\les} \sum_{\ell \le \ell_k+\gamma} 2^{\ell(d-1)} \sum_{ \Omega_j \subset B(x,C 2^{-\ell})} \diam(\Omega_j)^\beta.
\end{split}\end{equation}
Let us also notice that, if $\Omega_j \subseteq B(x, C2^{-\ell})$, then necessarily $\delta \le \diam(\Omega_j) \les 2^{-\ell}$ (since $\diam(\Omega_j)^d \sim |\Omega_j|$).  Thus for $\ell'$ with $2^{-\ell'} \sim 2^{-\ell}$,
\begin{equation}\begin{split}
 \sum_{ \Omega_j \subset B(x, C 2^{-\ell})}\diam(\Omega_j)^\beta& \les \sum_{\ell' \le u \le |\log_2 \delta|} 2^{-u \beta} \sharp {\cur{ \Omega_j \subseteq B(x, C 2^{-\ell}) \, : \diam(\Omega_j) \in [2^{-u}, 2^{-u+1}) }}\\
 & \stackrel{\eqref{eq:uniform-bound-omega-k}}{\les} \sum_{\ell' \le u \le |\log_2 \delta|} 2^{-u\beta} \cdot 2^{(u-\ell) (d-1) }= 2^{-\ell(d-1)} \sum_{\ell' \le u \le |\log_2 \delta|} 2^{-u(\beta+1-d)}\\
 & \les 2^{-\ell(d-1)}  r_{\beta+1-d}(\delta)
\end{split}
\end{equation}
using again that $\ell'$ is bounded from below by a constant depending on $\mathcal{Q}$ only.
Plugging this bound in \eqref{eq:final-claim-step}, we conclude that
\begin{equation}\begin{split}
 \sum_{j \, : \, \dist(x,\Omega_j)>2^{-\gamma} \diam(\Omega_k)}  \frac{\diam(\Omega_j)^\beta}{d(x,\Omega_j)^{d-1}}  & \le \sum_{\ell \le \ell_k+\gamma} 2^{\ell(d-1)} \cdot  2^{-\ell(d-1)} r_{\beta+1-d}(\delta) \\
 & \les  \bra{ \gamma + |\log \bra{ \diam(\Omega_k)}|} r_{\beta+1-d}(\delta).
\end{split}
\end{equation}
This concludes the proof of \eqref{lastclaim}.
\end{proof}

As in \cite{goldman2022optimal} in the next result we rely on gradient bounds for 
the Green kernel $(G(x,y))_{x,y\in \Omega}$ of the Laplacian with Neumann boundary conditions:
\begin{equation}\label{eq:green-kernel-bound} \abs{ \nabla _x  G(x,y) } \les |x-y|^{1-d}, \quad \text{for every $x$, $y \in \Omega$,}\end{equation}
where the implicit constant depends uniquely on $\Omega$. 
This condition is satisfied for instance  if  $\Omega$ is $C^2$ or convex, see e.g.\ \cite{wang2013gradient}. Notice that since it is a local condition it also holds for $Q\backslash \Omega$ with $\Omega$ a $C^2$ open set with $d(\partial Q,\partial \Omega)>0$. We have the following bound (see \cite[Lemma 2.6]{goldman2022optimal}).

\begin{lemma}\label{lem:W1pWhitney}
 Let $\Omega\subset \R^d$ be a bounded domain with Lipschitz boundary such that \eqref{eq:green-kernel-bound} holds and for $\delta>0$ sufficiently small, consider a partition $(\Omega_k)_{k=1}^K = \cQ_\delta \cup \cR_\delta$ as in \cref{lem:decomp}. 
For any $(b_k)_{k} \subseteq \R$, $\beta >0$ and $p\ge 1$,
\begin{equation}\label{estimW1pWhitney}
 \lt\|\sum_{k} b_k \chi_{\Omega_k}  - b\rt\|_{W^{-1,p}(\Omega)}\les |\log \delta| \delta^{1-\beta} \cdot \max_{k} \cur{ |b_k|\diam(\Omega_k)^{\beta}},
\end{equation}
with $b = \sum_{k} b_k |\Omega_k|/ |\Omega|$ and the implicit constant depends only on $p$, $d$ and $\mathcal{Q}$ (not on $\delta$ nor $(b_k)_k$.).
\end{lemma}
\begin{proof}
We set
\begin{equation}
 f_k= \chi_{\Omega_k}  - \frac{|\Omega_k|}{|\Omega|}
\end{equation}
and let $\phi_k$ denote the solution to the equation $\Delta \phi_k = f_k$, with null Neumann boundary conditions on $\Omega$ and use as competitor $\xi=\sum_{k=1}^K  b_k \nabla \phi_k$ in the definition of the $W^{-1,p}$ norm.
We get
\begin{equation}\label{normW1pA}
\begin{split}
 \lt\|\sum_{k} b_k \chi_{\Omega_k}  - b\rt\|_{W^{-1,p}(\Omega)}^p & =  \nor{ \sum_{k}  b_k f_k}_{W^{-1,p}(\Omega)}^p  \le \int_{\Omega} \abs{ \sum_{k} b_k {\nabla \phi_k}}^p \\
  & \le \bra{ \max_{k} \cur{ |b_k|\diam(\Omega_k)^{\beta}}}^{p}  \int_{\Omega} \bra{\sum_{k} \diam(\Omega_k)^{-\beta} \abs{\nabla \phi_k}}^p.
  \end{split}
\end{equation}
To bound the last term, we use the integral representation  in terms of the Green's function,
\begin{equation} \phi_k = \int_{\Omega}  G(x,y) f_k(y) d y, \end{equation}
to obtain that, for every $x\in \Omega$,
\begin{equation}\label{boundgradphii}
|\nabla \phi_k(x)|\les \min\cur{\diam(\Omega_k), \frac{|\Omega_k|}{\dist(x,\Omega_k)^{d-1}}} \les \diam(\Omega_k) \min \cur{1, \bra{\frac{ \diam(\Omega_k)}{\dist(x, \Omega_k)}}^{d-1}}.
\end{equation}
Indeed, by \eqref{eq:green-kernel-bound},
\begin{equation}\begin{split}
  |\nabla \phi_k(x)|& \les \int_{\Omega_k} \frac{dy}{|x-y|^{d-1}}+ |\Omega_k| \int_{\Omega} \frac{dy}{|x-y|^{d-1}} 
   \le \int_{\cur{|y|\le \diam(\Omega_k)}} \frac{dy}{|y|^{d-1}} +|\Omega_k|\\
   & \les \diam(\Omega_k).
\end{split}\end{equation}
Moreover, for $x\notin \Omega_k$, we get directly from \eqref{eq:green-kernel-bound},
\begin{equation}
 |\nabla \phi_k(x)|\les \frac{|\Omega_k|}{\dist(x,\Omega_k)^{d-1}}.
\end{equation}
For any $k=1, \ldots, K$ and $x \in \Omega_k$, we then estimate
\begin{equation}
 \begin{split} \sum_{j=1}^K \diam(\Omega_j)^{-\beta} \abs{\nabla \phi_j(x)} & \stackrel{\eqref{boundgradphii}}{\les}  \sum_{j=1}^K \diam(\Omega_j)^{1-\beta} \min \cur{1, \bra{\frac{ \diam(\Omega_j)}{\dist(x, \Omega_j)}}^{d-1}}\\
 & \stackrel{\eqref{eq:final-bound-min-sum-partition}}{\les} \bra{ \abs{\log \delta} + \abs{\log \diam(\Omega_k)}} \delta^{1-\beta}. 
 \end{split}
 \end{equation}

To conclude, we go back with the integration and bound from above:
\begin{equation}
\begin{split} & \int_{\Omega} \bra{\sum_{j=1}^K \diam(\Omega_j)^{-\beta} \abs{\nabla \phi_j}}^p  = \sum_{k=1}^K \int_{\Omega_k} \bra{ \abs{\log \delta}^p + \abs{\log \diam(\Omega_k)}^p} \delta^{(1-\beta)p}\\
& \qquad \les \bra{ \abs{\log \delta}^p + \sum_{k=1}^K \diam(\Omega_k)^d  \abs{\log \diam(\Omega_k)}^p} \delta^{(1-\beta)p}\\
& \qquad \les \abs{\log \delta}^p \delta^{(1-\beta)p},
\end{split}
\end{equation}
%
by estimating $z^d |\log(z)|^p \les z^{d-1/2}$ and using \eqref{eq:whitney-general-q} with $\alpha = d-1/2$..
%
\end{proof}

We have now all the preliminaries to address the proof of \Cref{theo:domain}.

\subsection{Upper bound}

In this section, under the assumptions of \Cref{theo:domain}, we establish the inequality
\begin{equation}\label{eq:upper-bound}
\limsup_{n \to \infty} n^{-1+p/d} \EE\sqa{W_\Omega^p(\nu_n )} \le \c_{\nu, p} |\Omega|.
\end{equation}

Recall that we assume without loss of generality that $\EE\sqa{\nu(A)} = |A|$. Write also $\nu_n = \dil_{n^{-1/d}} \nu$, i.e., $\nu_n(A) = \nu(n^{1/d} A)$, hence
\begin{equation}
\EE\sqa{\nu_n(A)} = n|A|,
\end{equation}
and, by \eqref{eq:concentration-abstract},
\begin{equation}
\| \nu_n(A) - \EE\sqa{ \nu_n(A) }\|_q \les_q \diam(n^{1/d} A)^{(d+\alpha)/2} = n^{\frac{d+\alpha}{2d}} \diam(A)^{(d+\alpha)/2},
\end{equation}
provided that $\diam(n^{1/d} A) \ge 1$, a condition that surely holds if $n$ is sufficiently large and  
\begin{equation}\label{eq:delta-condition}
\gamma_\delta <  \frac 1 d.
\end{equation}
This leads to the concentration inequality
\begin{equation}\label{eq:concentration-abstract-nu-n}
\PP\bra{ \abs{ \frac{ \nu_n(A)}{n|A|}  - 1 } > \bra{ n^{1/d} \diam(A)}^{-\beta/2}  } \les_q \bra{  n^{\frac{\alpha-d +\beta}{2d}} \frac{ \diam(A)^{(d+\alpha+\beta)/2}}{|A|}}^{q},
\end{equation}
for every $q \ge 1$ and $\beta>0$. Moreover, \eqref{eq:boundfR} becomes
\begin{equation}\label{eq:cube-upper-bound-nu-n}
\frac{1}{|Q|} \EE\sqa{ W_Q^p( \nu_n ) } \les n^{1-p/d}.
\end{equation}

%
%
  We begin by fixing a Whitney decomposition
 $\cQ =(Q_i)_i$ of $\Omega$. 
 and a positive parameter $\gamma_\delta>0$, to be further specified below. By \cref{lem:decomp} with $\delta = n^{-\gamma_\delta}$, if $n$ is large enough, we  have a finite Borel partition of $\Omega = \bigcup_{i} \Omega_i$, whose elements are collected into the two disjoint sets $\cQ_{\delta}$, $\cR_{\delta}$.

By \eqref{eq:sub}, we write
\begin{equation}\label{eq:exp-sub-explicit-2} \EE\sqa{ W_{\Omega}^p\bra{ \nu_n}}  \le (1+\eps)\sum_{i}   \EE\sqa{ W_{\Omega_i}^p\bra{ \nu_n }  } +   n \frac{C}{\eps^{(p-1)^+}} \EE\sqa{ W_{\Omega}^p\bra{ \sum_i \kappa_i \chi_{\Omega_i}}},
\end{equation}
having defined $\kappa_i = \nu_n(\Omega_i)/(n|\Omega_i|)$. For each $\Omega_i\in \cR_\delta$  we use the the trivial bound  \eqref{eq:trivial-wass},
\begin{equation}
\EE\sqa{ W_{\Omega_i}^p\bra{\nu_n}} \les \diam(\Omega_i)^p  \EE\sqa{\nu_n(\Omega_i)}\les \delta^{p+d} n,
\end{equation}
so that their contribution is
\begin{equation}
\sum_{\Omega_i \in \cR_\delta}  \EE\sqa{ W_{\Omega_i}^p\bra{ \nu_u }  } \les \delta^{1-d} \delta^{p+d} n = \delta^{1+p} n \ll n^{1-p/d}.
\end{equation}
provided that 
\begin{equation}\label{eq:delta-first-condition}
\delta^{1+p} \ll n^{-p/d}, \quad \text{i.e.,} \quad \gamma_\delta > \frac{p}{(p+1)d}.
\end{equation}
If $\Omega_i\in \cQ_\delta$, then $\Omega_i$ is a cube, hence we have by  \eqref{eq:cube-upper-bound-nu-n} that
\begin{equation}
\EE\sqa{ W_{\Omega_i}^p( \nu_n ) } \les |\Omega_i| n^{1-p/d},
\end{equation}
where the constant does not depend on $\Omega_i$. Using this domination and \Cref{theo:PoiLeb}, we deduce that
 \begin{equation}\label{eq:main-contribution-G-delta}
 \limsup_{n \to \infty} n^{-1+p/d} \sum_{\Omega_i \in \cQ_{\delta(n)}}  \EE\sqa{  W_{\Omega_i}^p\bra{\nu_n}}
  \le  \c_{\nu, p} \sum_{ i \in \mathcal{Q}} |\Omega_i|  = \c_{\nu, p} | \Omega|.\end{equation}
Therefore, in order to conclude, it is sufficient to focus on the remaining term in right-hand side of \eqref{eq:exp-sub-explicit-2} and argue that
 \begin{equation}\label{eq:remainder-contribution-npd}
  \EE\sqa{ W_{\Omega}^p\bra{ \sum_i \kappa_i \chi_{\Omega_i}}} \ll n^{-p/d}. 
  \end{equation} 
  provided that $\gamma_\delta$ is suitably chosen. Writing also $\kappa = \nu_n(\Omega)/(n|\Omega|)$, we introduce an auxiliary parameter $\beta$ satisfying
  \begin{equation}
 0< \beta< d-\alpha
 \end{equation}
 and the event
  \begin{equation}
  A = \cur{  \abs{\kappa-1} \le (n^{1/d} \diam(\Omega) )^{-\beta/2}  }  \cap \bigcap_{i} \cur{ \abs{\kappa_i-1} \le  (n^{1/d} \diam(\Omega_i) )^{-\beta/2}  },
  \end{equation}
  whose probability is bounded from above, via union bound and an application of \eqref{eq:concentration-abstract-nu-n} (recalling also that $|\Omega_i| \sim \diam(\Omega_i)^d$):
  \begin{equation}\begin{split}
  \mathbb{P}(A^c)  \les_q   n^{(\alpha-d +\beta)q/(2d)} \bra{ 1 +   \sum_{i} \diam(\Omega_i)^{(\alpha-d+\beta)q/2}} \les_q n^{(\alpha-d +\beta)q/(2d)}
  \end{split}
  \end{equation}
provided that $q$ is chosen sufficiently large.
By \eqref{eq:trivial-wass}, H\"older inequality and \eqref{eq:concentration-abstract-nu-n},
\begin{equation}\label{eq:bad-event}
\EE\sqa{  W_{\Omega}^p\bra{ \sum_i \kappa_i \chi_{\Omega_i}} I_{A^c} } \les \EE\sqa{ \kappa I_{A^c} } \les \EE\sqa{ \kappa^2}^{1/2} \mathbb{P}(A^c)^{1/2} \les \mathbb{P}(A^c)^{1/2} \ll n^{-p/d}
\end{equation}
provided that $q$ is sufficiently large. Therefore, we are reduced to bound $W_{\Omega}^p\bra{ \sum_i \kappa_i \chi_{\Omega_i}}$ on the event $A$. If $n$ is large enough, on $A$ it holds $\kappa \ge 1/2$, hence by \Cref{lem:peyre} we obtain
\begin{equation}
  W_{\Omega}^p\bra{ \sum_i \kappa_i \chi_{\Omega_i}}  \les   \nor{ \sum_{i} (\kappa_i - \kappa) \chi_{\Omega_i}}_{W^{-1,p}(\Omega)}^p 
\end{equation}
Then, by the identity
\begin{equation}
\sum_{i} (\kappa_i - \kappa) \chi_{\Omega_i} = \sum_{i} \bra{\kappa_i - 1} \bra{ \chi_{\Omega_i} - \frac{ |\Omega_i|}{|\Omega|}},
\end{equation}
still on the event $A$ we apply \Cref{lem:W1pWhitney} with $(b_k)_{k} = (\kappa_i - 1)_i$ and $\beta/2$ instead of $\beta$,   so that
\begin{equation}
 \sup_{i} |\kappa_i -1| \diam(\Omega_i)^{\beta/2} \le n^{-\beta/(2d)},
 \end{equation} 
 obtaining
\begin{equation}
\nor{ \sum_{i} (\kappa_i - \kappa) \chi_{\Omega_i}}_{W^{-1,p}(\Omega)}^p \les \bra{ \abs{\log \delta} \delta^{1-\beta/2} n^{-\beta/(2d)}}^{p}.
\end{equation}
We choose $\gamma_\delta$ so that
\begin{equation}
|\log \delta| \delta^{1-\beta/2} n^{-\beta/(2d)} \ll n^{-1/d}
\end{equation}
which is ensured if
\begin{equation}
(\beta/2-1) \gamma_\delta  < \frac{\beta/2 -1}{d}, \quad\text{i.e.,} \quad \gamma_\delta<\frac 1 d,
\end{equation}
which is a condition  we already found in \eqref{eq:delta-condition}. Recalling also that $\gamma_\delta$ must satisfy \eqref{eq:delta-first-condition} we see that indeed indeed one can always choose such a $\gamma_\delta$. This concludes the proof of \eqref{eq:remainder-contribution-npd}, hence we settled \eqref{eq:upper-bound}.

\subsection{Lower bound}

Under the assumptions of \Cref{theo:domain}, we argue now that the corresponding lower bound holds:
\begin{equation}\label{eq:lower-bound}
\liminf_{n \to \infty} n^{-1+p/d} \EE\sqa{W_\Omega^p(\nu_n )} \ge \c_{\nu, p} |\Omega|,
\end{equation}
thus concluding the proof of \Cref{theo:domain}. The main idea dates back to \cite[Theorem 24]{BaBo}. We fix a cube $Q=Q_L$ with $L$ so large that $\Omega \subseteq Q_{L-1}$. We set $\Omega_1 = \Omega$ and, for $i=2,\ldots, K$ let  $\Omega_i$ be the connected components of $Q\backslash \Omega$ so that  $Q\backslash \Omega=\cup_{i=2}^K\Omega_i$. Notice that, for every $i$,  either $\partial \Omega_i$ is $C^2$ or is the union of $\partial Q$ and a $C^2$ surface. 
In particular each $\Omega_i$ satisfies 
\eqref{eq:green-kernel-bound}.  By \Cref{theo:PoiLeb} and \eqref{eq:sub}  we obtain that
\begin{equation}\begin{split} 
\c_{\nu, p} |Q| & = \lim_{n \to \infty} n^{-1+p/d}\EE\sqa{ W_Q^p(\nu_n )} \\
& \le (1+\eps) \liminf_{n \to \infty} n^{-1+p/d}\EE\sqa{ W_\Omega^p(\nu_n )} \\
& \quad + (1+\eps) \sum_{i=2}^K \limsup_{n \to \infty} n^{-1+p/d}\EE\sqa{ W_{\Omega_i}^p(\nu_n )} + \frac{C}{\eps^{(p-1)^+}} \limsup_{n \to \infty} n^{p/d}\EE\sqa{ W_{\Omega_i}^p( \kappa_i)},
\end{split}
\end{equation}
with $\kappa_i = \nu_n(\Omega_i)/(n|\Omega_i|)$. Using \eqref{eq:upper-bound} for $i=2, \ldots K$, we obtain
\begin{equation}
\c_{\nu, p} |Q| \le \liminf_{n \to \infty} n^{-1+p/d}\EE\sqa{ W_\Omega^p(\nu_n )} + \sum_{k=2 }^K \c_{\nu, p} |\Omega_i| +   \limsup_{n \to \infty} n^{p/d}\EE\sqa{ W_{\Omega_i}^p( \kappa_i)}.
\end{equation}
Since $|\Omega| = |Q| - \sum_{i=2}^K |\Omega_i|$, inequality \eqref{eq:lower-bound}, hence the thesis, follows if we argue that
\begin{equation}
\EE\sqa{ W_{\Omega_i}^p( \kappa_i)} \ll n^{-p/d}.
\end{equation}
We set $\kappa = \nu_n(Q)/(n|Q|)$. If $p<1$, we use \eqref{eq:less-trivial-wass} to obtain
\begin{equation}\begin{split}
\EE\sqa{ W_{\Omega_i}^p( \kappa_i) } &\le  \EE\sqa{ \diam(\Omega)^p\sum_{i=1}^K |\Omega_i| \abs{\kappa_i - \kappa}} \\
& \les  \EE\sqa{ \abs{\kappa-1}} + \sum_{i=1}^K \EE\sqa{\abs{ \kappa_i - 1} } \les n^{(\alpha-d)/(2d} \ll n^{-p/d}
\end{split}
\end{equation}
by the assumption \eqref{eq:assumption-p-not-too-large}. If $p \ge 1$, we proceed along the lines as the proof of \eqref{eq:remainder-contribution-npd}, but here it is actually simpler since the partition is fixed (previously it depended upon $n$ through $\delta$). First, we set $\kappa = \nu_n(Q)/(n|Q|)$ and introduce the event
\begin{equation}
A =  \cur{  \abs{\kappa-1} \le (n^{1/d} \diam(Q) )^{-\beta/2}  }  \cap \bigcap_{i=1}^K \cur{ \abs{\kappa_i-1} \le  (n^{1/d} \diam(\Omega_i) )^{-\beta/2}  },
\end{equation}
for some $0<\beta < d-\alpha$. Arguing as in the proof of \eqref{eq:bad-event}, we are easily reduced to prove
\begin{equation}
\EE\sqa{ W_{\Omega_i}^p( \kappa_i) I_A } \ll n^{-p/d}.
\end{equation}
If $n$ is sufficiently large, we have $\kappa \ge 1/2$ on $A$, hence by \eqref{eq:estimCZ-Lp} we obtain
\begin{equation}
\EE\sqa{ W_{\Omega_i}^p( \kappa_i)  I_A} \les \sum_{i=1}^K \EE\sqa{ \abs{\kappa_i-1}^p I_A} \les n^{-p \beta /(2d)},
\end{equation}
which is $\ll n^{-p/d}$ using \eqref{eq:assumption-p-not-too-large} in this case.

\section{De-Poissonization}\label{app:depoisson}
In this section we discuss a De-Poissonization argument in order to transfer limit results from the case of the sum of a random number $N_\lambda$ of measures, and that of a deterministic number $\EE\sqa{N_\lambda} \approx \lambda$.

\begin{proposition}\label{prop:depoisson}
Let $(\mu^i)_{i=1}^\infty$ be i.i.d.\ random Borel measures on $\Omega \subseteq \R^d$ such that
\begin{equation}
\EE\sqa{W_{\Omega}^p(\mu^1)} < \infty,
\end{equation}
and set
\begin{equation}
 f(n) := \EE\sqa{ W_{\Omega}^p\bra{ \sum_{i=1}^n \mu^i }}
 \end{equation} 
Let $N_\lambda$  denote a (further) independent Poisson random variable with mean $\lambda$. Then, for every $\alpha \in \mathbb{R}$,
\begin{equation}
 \liminf_{n \to \infty} n^{\alpha} f(n) \ge \liminf_{\lambda \to \infty}  \lambda^\alpha \EE\sqa{ f(N_\lambda)},
 \end{equation}
 and 
 \begin{equation}
 \limsup_{n \to \infty} n^{\alpha} f(n) \le \limsup_{\lambda \to \infty}  \lambda^\alpha \EE\sqa{ f(N_\lambda) }. 
 \end{equation}
 In particular, the two limits
 \begin{equation}
 \lim_{n \to \infty} n^\alpha f(n)  \quad \text{and} \quad \limsup_{\lambda \to \infty}  \lambda^\alpha \EE\sqa{ f(N_\lambda) }
 \end{equation}
 exist and coincide, whenever  one is known to exist.
\end{proposition}
%
%
%
%
%
%

\begin{proof}
Given $1\le m\le n$, the identity
\begin{equation}
\frac 1 n \sum_{i=1}^n \mu^i =  {n \choose m}^{-1} \sum_{\substack{ I \subseteq \cur{1, \ldots, n} \\  \sharp I = m}}  \frac 1 m \sum_{i \in I} \mu^i
\end{equation}
 in combination with  \eqref{eq:sub} yields
\begin{equation}\begin{split} \frac 1 n W_{\Omega}^p\bra{  \sum_{i=1}^n \mu^i } & = W_{\Omega}^p\bra{{n \choose m}^{-1} \sum_{\substack{ I \subseteq \cur{1, \ldots, n} \\ \sharp I = m}} \frac 1 m \sum_{i \in I} \delta_{\mu^i}} \\
&  \le {n \choose m}^{-1} \sum_{\substack{ I \subseteq \cur{1, \ldots, n} \\ \sharp I = m}}   \frac 1 m W_{\Omega}^p\bra{ \sum_{i \in I} \mu^i}. 
\end{split}
\end{equation}
Taking expectation, we obtain that
\begin{equation}\label{eq:monotonicity-f-n}
\frac {f(n)}{n} \le \frac {f(m)}{m},
\end{equation}
since $(\mu^i)_{i=1}^m$ have the same law of $(\mu^i)_{i\in I}$, if $I$ contains $m$ elements.

Given $\eps>0$ and $\lambda>0$ introduce 
\begin{equation}
A  =  A_{\eps, \lambda} = \cur{ \lambda(1-\eps) \le N_\lambda \le \lambda(1+\eps)},
\end{equation}
which Markov inequality  satisfies, for every $q \ge 1$, the inequality
\begin{equation}\label{eq:expectation-bad-event}
\EE\sqa{ N_\lambda I_{A^c}} \les_q \lambda^{-q }.
\end{equation}
Using \eqref{eq:sub} and the trivial bound
\begin{equation}
W_\Omega^p\bra{ \sum_{i=1}^{n} \mu^i } \le \sum_{i=1}^{n}  W_\Omega^p\bra{ \mu^i }, 
\end{equation}
we find that
\begin{equation}\label{eq:error-term-depoissonization}
\begin{split}
\EE\sqa{ f(N_\Lambda) I_{A^c}} & \le \EE\sqa{  \sum_{i=1}^{N_\lambda}  W_\Omega^p\bra{ \mu^i } I_{A^c}} =\EE\sqa{ N_\lambda I_{A^c}} \EE\sqa{ W_\Omega^p\bra{ \mu^1} } \\
& \les_q \lambda^{-q},
\end{split}
\end{equation}
which is infinitesimal (even after multiplying by $\lambda^\alpha$, if $q$ is sufficiently large).
Thus, we are reduced to bound $W_{\Omega}^p\bra{\sum_{i=1}^{N_\lambda} \mu^i }$ on $A$, for which we use \eqref{eq:monotonicity-f-n}
\begin{equation}
\frac{f(\lfloor \lambda(1+\eps) \rfloor)}{\lfloor \lambda(1+\eps)\rfloor}  \le \frac{f(N_\lambda)}{N_\lambda} \le  \frac{f(\lfloor \lambda(1-\eps) \rfloor)}{\lfloor \lambda(1-\eps)\rfloor}.
\end{equation}
Multiplying both sides by $N_{\lambda}$ we obtain (still on $A$) that
\begin{equation}
\frac{\lfloor \lambda(1-\eps)\rfloor f(\lfloor \lambda(1+\eps) \rfloor)}{\lfloor \lambda(1+\eps)\rfloor} \le f(N_\lambda) \le \frac{\lfloor \lambda(1+\eps)\rfloor f(\lfloor \lambda(1-\eps) \rfloor)}{\lfloor \lambda(1-\eps)\rfloor}.
\end{equation}
Taking also expectation with respect to $N_{\lambda}$, we obtain
\begin{equation}
\frac{\lfloor \lambda(1-\eps)\rfloor f(\lfloor \lambda(1+\eps) \rfloor)}{\lfloor \lambda(1+\eps)\rfloor} \mathbb{P}(A) \le \EE\sqa{ f(N_\lambda) I_{A}} \le \frac{\lfloor \lambda(1+\eps)\rfloor f(\lfloor \lambda(1-\eps) \rfloor)}{\lfloor \lambda(1-\eps)\rfloor} \mathbb{P}(A).
\end{equation}
Multiplying both sides by $\lambda^\alpha$ and letting $\lambda \to \infty$, we obtain that
\begin{equation}
 \limsup_{\lambda \to \infty} \frac{\lambda^\alpha \lfloor \lambda(1-\eps)\rfloor f(\lfloor \lambda(1+\eps) \rfloor)}{\lfloor \lambda(1+\eps)\rfloor} \mathbb{P}(A)  \le \limsup_{\lambda \to \infty} \lambda^\alpha \EE\sqa{f(N_\lambda)} 
 \end{equation} 
 and
 \begin{equation}
 \liminf_{\lambda \to \infty} \lambda^\alpha \EE\sqa{f(N_\lambda)} \le \liminf_{\lambda \to \infty} \frac{\lambda^\alpha \lfloor \lambda(1-\eps)\rfloor f(\lfloor \lambda(1+\eps) \rfloor)}{\lfloor \lambda(1+\eps)\rfloor}. 
\end{equation}
The thesis then follows by noticing that
\begin{equation}\begin{split}
\limsup_{n \to \infty} n^\alpha f(n) &= \limsup_{\lambda \to \infty} \lfloor \lambda (1+\eps)\rfloor^\alpha f( \lfloor \lambda(1+\eps) \rfloor ) \\
& = \frac{ (1+\eps)^{\alpha-1} }{1-\eps} \limsup_{\lambda \to \infty} \frac{\lambda^\alpha \lfloor \lambda(1-\eps)\rfloor f(\lfloor \lambda(1+\eps) \rfloor)}{\lfloor \lambda(1+\eps)\rfloor} \mathbb{P}(A) 
\end{split}
\end{equation}
and similarly
\begin{equation}
\liminf_{n \to \infty} n^\alpha f(n)  = \frac{ (1-\eps)^{\alpha-1} }{1+\eps} \liminf_{\lambda \to \infty} \frac{\lambda^\alpha \lfloor \lambda(1-\eps)\rfloor f(\lfloor \lambda(1+\eps) \rfloor)}{\lfloor \lambda(1+\eps)\rfloor} \mathbb{P}(A) ,
\end{equation}
and letting $\eps\to 0$.
\end{proof}

\begin{remark}\label{rem:debinom}
We notice that the argument above does not depend very much on the fact that $N_\lambda$ has Poisson law, but rather than \eqref{eq:expectation-bad-event} holds. In particular, if we replace it with a binomial variable $N_T$ with parameters $m=m(T)\to \infty$ and $p=p(T) \to 0$ such that $\lambda := mp \to \infty$, and  we assume that
\begin{equation}
\lim_{n \to \infty} n^\alpha f(n)
\end{equation}
exists, then also the limit
\begin{equation}
\lim_{T \to \infty} \lambda^\alpha \EE\sqa{ f(N_T) }
\end{equation}
exists and coincides with the first limit.
\end{remark}

\section{Upper bounds for the optimal transportation cost}\label{sec:upper}

In this section we collect two upper bounds for transportation cost on the torus. The arguments we employ are simple variants of those contained in \cite{huesmann2022wasserstein} in order to establish the upper bound in \eqref{eq:mattesini}.

We begin with the following result, which is employed in the proof of \Cref{thm:main-torus}. It states that by ``smoothing'' the occupation measure of a stationary Brownian motion on $\T^d$ up to time $T$ over balls whose radius is infinitesimal, but larger than the critical threshold $T^{-1/(d-2)}$,  the Wasserstein costs becomes asymptotically smaller than the sharp rates \eqref{eq:mattesini}. For simplicity, we only argue in the range of dimensions and exponents where our main results \Cref{thm:main-torus} apply.

\begin{proposition}\label{prop:smoothed-wasserstein}
 Let $d \in \cur{3,4}$ and $p \in (0, (d-2)/2)$ or $d \ge 5$ and $p>0$,  and let  $0<\gamma < 1/(d-2)$. For $T >0$, set $\ell = T^{-\gamma}$ and let $B = (B_t)_{t \ge 0}$ be a stationary Brownian motion on on $\T^d$. Then, it holds
\begin{equation}
 \limsup_{T \to \infty} \EE\sqa{  W^p_{\T^d}\lt(\int_{\T^d} \frac{\mu_T^B(D_\ell(z))}{\abs{D_\ell}}\chi_{D_\ell(z)} \frac{ d z}{|D_\ell|} \rt)}/ T^{1-p/(d-2)} = 0.
\end{equation}
\end{proposition}

\begin{proof}
 We argue first the case $d \ge 5$ and  first collect some general facts.  We notice that by \eqref{eq:higer-r-wass} it is sufficient to argue in the case $p \in \mathbb{N}$ even. Given any (positive) measure $\nu$ on $\T^d$ with $\nu(\T^d) =T$, we write
\begin{equation}\label{eq:nu-ell}
 \nu_\ell(x) := \int_{\T^d} \frac{\nu(D_\ell(z))}{\abs{D_\ell}}\chi_{D_\ell(z)} (x)\frac{ d z}{|D_\ell|}.
\end{equation}

Then, by  \Cref{lem:peyre} with $\lambda = T \mathcal{L}^d_{\restr \T^d}$, we obtain for any $p \ge 1$ the inequality
\begin{equation}\label{eq:peyre-app-c}
  W^p_{\T^d}\lt( \nu_\ell \rt)  \les_p T^{1-p} \nor{ \nu_\ell - T}_{W^{-1,p}}^p \les_p T^{1-p}\nor{ \nabla \Delta^{-1} (\nu_\ell - T)}_{L^p}^p
\end{equation}
having used the boundedness of the Riesz transform in the second inequality.
%
By the Fourier inversion formula we find that
\begin{equation}
 \nabla \Delta^{-1} (\nu_\ell - T) (x) = \sum_{\xi \in \mathbb{Z}^d \setminus \cur{0}} \frac{i \xi}{2 \pi|\xi|^2 } \hat{\nu_\ell}(\xi) \exp\bra{2 \pi i \xi \cdot x}
\end{equation}
and therefore
\begin{equation}\label{eq:negative-sobolev-abstract}
 \nor{ \nabla \Delta^{-1} (\nu_\ell - T)}_{L^p}^p = \sum_{ \substack{ \xi_1, \ldots, \xi_p \in \mathbb{Z}^d \setminus \cur{0} \\ \xi_1 + \ldots + \xi_p = 0} } \prod_{j=1}^p \frac{i \xi_j}{2 \pi|xi_j|^2 } \hat{\nu_\ell}(\xi_j)
\end{equation}
Next, we notice that
\begin{equation}\label{eq:nu_ell-convolution}
 \nu_\ell= \mu * {  \frac{ \chi_{D_\ell} }{|D_\ell|} } * {  \frac{ \chi_{D_\ell} }{|D_\ell|} },
\end{equation}
where $D_\ell = D_\ell(0)$ and $*$ denotes the convolution operation on $\T^d$. Taking the Fourier transform, we obtain for every $\xi \in \mathbb{Z}^d$ that
\begin{equation}\hat{\nu}_\ell(\xi) = \int_{\T^d} \exp\bra{ -2\pi i \xi \cdot x} d\nu_\ell(x) =    \bra{ \frac{ \widehat{\chi_{D_\ell} }(\xi)}{|D_\ell|}}^2 \hat{\nu}(\xi).
\end{equation}
We can assume that $\ell = T^{-\gamma}<1/2$, hence we identify $D_\ell \subseteq \T^d$ with the Euclidean ball $D_\ell \subseteq \R^d$ and write
\begin{equation}\begin{split}
  \frac{ \widehat{\chi_{D_\ell} }(\xi)}{|D_\ell|} & =  \frac{1}{|D_\ell|} \int_{D_\ell}   \exp\bra{ -2\pi i \xi \cdot x}  dx = \frac{1}{|D_1|} \int_{D_1}   \exp\bra{ -2\pi i (\ell \xi) \cdot x}  dx \\
  & = \frac{ \widehat{\chi_{D_1} }( \ell \xi)}{|D_1|}.
  \end{split}
\end{equation}
The Fourier transform of $\chi_{D_1}$ can be described in terms of the Bessel functions of the first kind $J_\alpha$,
\begin{equation}
  \widehat{\chi_{D_1} } ( \xi) = | \xi|^{-d/2} J_{d/2}( |\xi |),
\end{equation}
yielding e.g.\ by \cite{brauer1963asymptotic} the upper bound,  for every $\xi \in \R^d$,
\begin{equation} \label{eq:bound-bessel}
  \abs{  \widehat{\chi_{D_1} } ( \xi) } \les_d \frac{1}{ (1 +|\xi|)^{(d+1)/2} }.
\end{equation}

Let us now specialize to the random case $\nu := \mu_T^B$. Taking the expectation in \eqref{eq:peyre-app-c} and \eqref{eq:negative-sobolev-abstract} and using \eqref{eq:bound-bessel} we obtain
\begin{equation}\begin{split}
                  &  \EE\sqa{  W^p_{\T^d}\lt(\int_{\T^d} \frac{\mu_T^B(D_\ell(z))}{\abs{D_\ell}}\chi_{D_\ell(z)} \frac{ d z}{|D_\ell|} \rt)} \les   \\
                  & \quad \quad  \les \sum_{ \substack{ \xi_1, \ldots, \xi_p \in \mathbb{Z}^d \setminus \cur{0} \\ \xi_1 + \ldots + \xi_p = 0} } \prod_{j=1}^p \frac{1}{ |\xi_j|(1 +|\ell \xi_j|)^{d+1} } \abs{ \EE\sqa{ \prod_{j=1}^p \widehat{\mu_T}(\xi_j)}}.
                \end{split}
\end{equation}
Next, repeating almost verbatim the steps 1 and 2 of the proof of \cite[proposition 3.6]{huesmann2022wasserstein} in the case $H=1/2$ but replacing with the function $(1+\ell |\xi|)^{1-d}$ the function $\exp\bra{ - \eps |\xi|^2}$ -- indeed, in those steps one uses only that the function is positive -- we see that our thesis is reduced to show that, for every $q \in \mathbb{N}$, $2 \le q \le p$, it holds
\begin{equation}\label{eq:sum-star}
 T \sum_{\xi_1, \ldots, \xi_q}^*\prod_{j=1}^q \frac{1}{ |\xi_j|(1 +|\ell \xi_j|)^{d+1} } \prod_{j=1}^{q-1}  \frac{1}{\abs{\sum_{i=1}^j \xi_{i}}^2} \ll T ^{(1-1/(d-2) )q},
\end{equation}
where the symbol $\sum^*$ denotes the summation restricted upon $\xi_1, \ldots, \xi_q \in \mathbb{Z}^d \setminus \cur{0}$ such that $\sum_{i=1}^j \xi_i \neq 0$ for every $j=1, \ldots, q-1$ and $\sum_{i=1}^q \xi_i = 0$. In the case $q=2$, this simply reduces to
\begin{equation}\label{eq:sum-star-p-2-upper-bound}
 T \sum_{\xi \in \mathbb{Z}^d \setminus \cur{0}} \frac{1}{ |\xi|^4 (1 +|\ell \xi|)^{2(d+1)} } \les  T \int_0^\infty  \frac{r^{d-1} dr}{ r^{4} (1+\ell r)^{2(d+1)} }  \les  T \ell^{4-d},
\end{equation}
having used a simple integral comparison. Actually, it is convenient for later use to collect the general inequality, valid for $\alpha \in (0,d)$ and $\beta \ge  0$ such that $\alpha+\beta>d$:
\begin{equation}\label{eq:upper-bound-integral-alpha-beta}
 \sum_{\xi \in \mathbb{Z}^d \setminus \cur{0}}  \frac{1}{ |\xi|^{\alpha} (1 +|\ell \xi|)^{\beta} }  \les \int_0^\infty  \frac{r^{d-1} dr}{ r^{\alpha} (1+\ell r)^{\beta} } \les  \ell^{\alpha-d},
\end{equation}
where the implicit constant depends on $\alpha$, $\beta$ and $d$. Notice that if instead $\alpha >d$ we can trivially bound from above for every $\beta \ge 0$,
\begin{equation}
 \sum_{\xi \in \mathbb{Z}^d \setminus \cur{0}}  \frac{1}{ |\xi|^{\alpha} (1 +|\ell \xi|)^{\beta} } \le \sum_{\xi \in \mathbb{Z}^d}  \frac{1}{ (1+|\xi|)^{\alpha}} \les 1.
\end{equation}

Back to \eqref{eq:sum-star}, for the general case $q \ge 3$ we bound from above
\begin{equation}\label{eq:upper-bound-appendix-c}\begin{split}
  \sum_{\xi_1, \ldots, \xi_q}^*\prod_{j=1}^q & \frac{1}{ |\xi_j|(1 +|\ell \xi_j|)^{d+1} } \prod_{j=1}^{q-1}  \frac{1}{\abs{\sum_{i=1}^j \xi_{\sigma(i)}}^2} \le \\
 & \le \sum_{ \xi_1, \ldots, \xi_{q-1} \in \mathbb{Z}^d\setminus \cur{0}}  \frac{1}{ |\xi_1|^3 (1 +|\ell \xi_1|)^{d+1} }  \cdot \\
 & \quad \quad \cdot \prod_{j=2}^{q-1} \frac{1}{ |\xi_j| (1 +|\ell \xi_j|)^{d+1} } \prod_{j=2}^{q-2}  \frac{1}{\bra{ 1+ \abs{\sum_{i=1}^j \xi_{i}} }^2} \cdot \\
 & \quad \quad \cdot \frac{1}{ |\sum_{i=1}^{q-1} \xi_{i}|^3 (1 +|\ell \sum_{i=1}^{q-1} \xi_{i}|)^{d+1} }.
 \end{split}
\end{equation}
To further bound from above, we apply Young's convolution inequality in the form of \cite[lemma 3.5]{huesmann2022wasserstein} with $q-1$ instead of $p$ and a suitable choice of exponents (using the notation therein), namely
\begin{equation}
 \lambda_1 =1, \quad \lambda_j :=  \frac{d}{d-2+\eps},  \quad \Lambda_j = \lambda_j' = \frac{ \lambda_j}{\lambda_j -1} \quad \text{if $2 \le j \le q-2$,}
\end{equation}
where $\eps=\eps(p, \gamma)>0$ is (sufficiently small) to be specified below, and finally
\begin{equation}
 \lambda_{q-1} := 4, \quad \Lambda_{q-1} := \lambda_{q-1}' = \frac{4}{3},
\end{equation}
Such a choice yields that $\Lambda_j > d/2$ for every $j \in \cur{2, \ldots, q-2}$ and  $\lambda_{q-1}<d$ as well as $3 \Lambda_{q-1} < d$ (because $d \ge 5$). Using several times \eqref{eq:upper-bound-integral-alpha-beta} (one checks easily that the conditions $\alpha<d$ and $\alpha+\beta>d$ are satisfied) we find that \eqref{eq:upper-bound-appendix-c} is bounded from above by the product
\begin{equation}
 \begin{split}
 &\bra{ \sum_{\xi \in \mathbb{Z}^d \setminus \cur{0}}\frac{1}{ |\xi|^3 (1 +|\ell \xi|)^{d+1} }} \cdot \bra{ \sum_{\xi \in \mathbb{Z}^d \setminus \cur{0}}\frac{1}{ |\xi|^{\lambda_2} (1 +|\ell \xi|)^{(d+1)\lambda_2} }}^{(p-3)/\lambda_2} \cdot\\
 & \quad \cdot \bra{ \sum_{\xi \in \mathbb{Z}^d \setminus \cur{0}}\frac{1}{(1 +|\xi|)^{2 \Lambda_2} }}^{(q-3)/ \Lambda_2 } \cdot  \\
 & \quad \cdot  \bra{ \sum_{\xi \in \mathbb{Z}^d \setminus \cur{0}}\frac{1}{ |\xi|^{4} (1 +|\ell \xi|)^{4(d+1)} }  }^{1/4} \cdot \bra{ \sum_{\xi \in \mathbb{Z}^d\setminus \cur{0}}\frac{1}{ |\xi|^{4} (1 +|\ell \xi|)^{4(d+1)/3} }}^{3/4 }\\
 & \les \ell^{3-d} \cdot \ell^{ (\lambda_2 -d)(q-3)/\lambda_2} \cdot 1 \cdot  \ell^{4-d}  = \ell^{(3-d) + (3-d-\eps)(q-3) + (4-d)} \\
 & \les  \ell^{(q-1)(2-d) +q - \eps(q-3)}
 \end{split}
\end{equation}
We notice that replacing $(q-3)$ with $(q-3)^+$ the above bound recovers also the case $q=2$ in  \eqref{eq:sum-star-p-2-upper-bound}.
Thus, recalling that $\ell = T^{-\gamma}$ with $\gamma<1/(d-2)$, we obtain, for every $q \in \cur{2, \ldots, p}$
\begin{equation}\begin{split}
 T \sum_{\xi_1, \ldots, \xi_q}^*\prod_{j=1}^q \frac{1}{ |\xi_j|(1 +|\ell \xi_j|)^{d+1} } \prod_{j=1}^{q-1}  \frac{1}{\abs{\sum_{i=1}^j \xi_{\sigma(i)}}^2}  & \les T^{1-\gamma\sqa{ (q-1)(2-d) +q - \eps(q-3)^+ } } \\
 & \ll T^{(1-1/(d-2))q},
 \end{split}
\end{equation}
provided that
\begin{equation}
 1-\gamma\sqa{ (q-1)(2-d) +q - \eps(q-3)^+ } < (1-1/(d-2))q.
\end{equation}
We claim that the condition can be always satisfied, provided that $\eps = \eps(p, \gamma)>0$ is small enough. Indeed, it is sufficient to show that strict inequality holds if we set $\eps = 0$. We find, using that $1/(d-2)- \gamma >0$,
\begin{equation}
 1-\gamma\sqa{ (q-1)(2-d) +q } < (1-1/(d-2))q  \quad \leftrightarrow \quad \frac{q}{q-1} < d-2,
\end{equation}
which is always satisfied since $q \ge 2$ and $d \ge 5$. This establishes the thesis in the case $d \ge 5$, $p>0$.

For the case $d \in \cur{3,4}$, since $p<(d-2)/2<1$, we need to modify the argument. Back to the case of general measure $\nu$ with $\nu(\T^d) = T$, we use \eqref{eq:peyre-concave}, yielding, for every $q<p$,
\begin{equation}\begin{split}\label{eq:wp-concave-appendix-c}
  W^p_{\T^d}\lt( \nu_\ell \rt)  & \les_p \nor{ \Delta^{-q/2} (\nu_\ell - T)}_{L^2} \\
  & \les \sqrt{ \sum_{\xi  \in \mathbb{Z}^d \setminus \cur{0}} \frac{ \abs{ \widehat \nu_\ell(\xi)}^2} {|\xi|^{2q} }} \les \sqrt{ \sum_{\xi  \in \mathbb{Z}^d \setminus\cur{0} } \frac{ \abs{ \widehat \nu(\xi)}^2} {|\xi|^{2q} (1+|\ell \xi|)^{2(d+1)} }}
  \end{split}
\end{equation}
having used Parseval's identity and \eqref{eq:bound-bessel} in the last line. Specializing then to the random case $\nu := \mu_T$ and taking expectation, we find after an application of H\"older's inequality
\begin{equation}\begin{split}\label{eq:expected-wp-concave-appendix-c}
 \EE\sqa{ W^p_{\T^d}\lt( \nu_\ell \rt)  } & \le \sqrt{ \sum_{\xi  \in \mathbb{Z}^d \setminus\cur{0} } \frac{  \EE\sqa{ \abs{ \widehat \nu(\xi)}^2 }} {|\xi|^{2q} (1+|\ell \xi|)^{2(d+1)} }}\\
 & \les \sqrt{ \sum_{\xi  \in \mathbb{Z}^d \setminus\cur{0} } \frac{ T} {|\xi|^{2q+2} (1+|\ell \xi|)^{2(d+1)} }}.
 \end{split}
\end{equation}
having used \cite[lemma 3.2]{huesmann2022wasserstein} in the case $H = 1/2$ to deduce that, for $\xi \neq 0$,
\begin{equation}
  \EE\sqa{ \abs{ \widehat \nu(\xi)}^2 } \les \frac{T}{|\xi|^2}.
\end{equation}
To conclude, we apply \eqref{eq:upper-bound-integral-alpha-beta} with $\alpha = 2q+2<d$ and $\beta = 2(d+1)$, yielding
\begin{equation}
 \sqrt{ \sum_{\xi  \in \mathbb{Z}^d \setminus\cur{0} } \frac{ T} {|\xi|^{2q+2} (1+|\ell \xi|)^{2(d+1)} }} \les \sqrt{T} \ell^{q+(2-d)/2} = T^{(1 +\gamma)/2 - \gamma q} \ll T^{1-p/(d-2)},
\end{equation}
by choosing $q$ sufficiently close to $p$ and using that $\gamma<1/(d-2)$.
\end{proof}

The second result in this section is employed in the proof of \Cref{prop:concentration}:  we extend the rates for transportation cost of the occupation measure from the case of a single stationary Brownian motion on the torus to that of a sum of independent ones.

\begin{proposition} \label{prop:upper-bound-sum-brownian-motions}
Let $d \ge 3$ and $(B^i)_{i=1}^\infty$ be independent stationary Brownian motions on $\T^d$. Then, for every $n \ge 1$, $T > 0$, it holds
\begin{equation}
 \EE\sqa{  W^p_{\T^d}\bra{  \sum_{i=1}^n \mu_T^{B^i} }}  \les (nT)^{1-p/(d-2)}, \quad \text{if $d \ge 5$ and $p>0$,}
 \end{equation}
where the implicit constant depends on $p$ and $d$ only, and
 \begin{equation}
 \EE\sqa{  W^p_{\T^d}\bra{  \sum_{i=1}^n \mu_T^{B^i} }}  \les (nT)^{1-q/(d-2)}, \quad \text{if $d \in \cur{3,4}$, $p \in (0,(d-2)/2)$ and $q<p$,}
 \end{equation}
 where the implicit constant depends on $p$, $d$  and $q$ only. 
\end{proposition}

\begin{proof}
 Consider first the case $d \ge 5$, so that by \eqref{eq:higer-r-wass} we can assume $p \ge 2$. Then, we essentially follow through the arguments from \cite[theorem 1.1]{huesmann2022wasserstein} in the case $H=1/2$, only for the upper bound part and replacing $T$ with $nT$.  The only difference appears in the application of \cite[proposition 3.6]{huesmann2022wasserstein}, where one needs now to show the upper bound
 \begin{equation}\label{eq:upper-bound-n}
  \EE\sqa{ \nor{\sum_{i=1}^n \nabla u^i_\eps }_{L^p}^p }\les (nT\sqrt{\eps_n})^p,
 \end{equation}
 where $\nabla u^i_\eps$ are i.i.d.\ copies of a single $\nabla u_\eps$ defined  in \cite[eq. (3.1)]{huesmann2022wasserstein}  for which \cite[proposition 3.6]{huesmann2022wasserstein} applies, and $\sqrt{\eps} = \sqrt{\eps_n} := (nT)^{-1/(d-2)}$. This inequality can be obtained by Rosenthal's bound \eqref{eq:rosenthal-centered}, since for every $x \in \mathbb{T}^d$ one has $\EE\sqa{\nabla u_\eps(x)} = 0$, hence
 \begin{equation}\label{eq:appendix-c-rosenthal}
  \EE\sqa{ \nor{\sum_{i=1}^n \nabla u^i_\eps   }^p_{L^p} } \les n \EE\sqa{ \nor{ \nabla u_\eps  }^p_{L^p} } + n^{p/2} \EE\sqa{ \abs{ \nabla u_\eps  }^2_{L^2} }^{p/2}.
 \end{equation}
Then, one needs to recall that  in \cite[proposition 3.6]{huesmann2022wasserstein} (for $d \ge 5$, $p \ge 2$) the smoothing parameter $\eps$ is already fixed as $\sqrt{\eps_1} := T^{-1/(d-2)}$: if one defines instead $\sqrt{\eps_n}:= (nT)^{-1/(d-2)} = \sqrt{\eps_1} n^{-1/(d-2)}$, following through the proof yields the upper bound
\begin{equation}
 \EE\sqa{ \nor{ \nabla u_\eps  }^p_{L^p} } \les (T\sqrt{\eps_1})^{p} n^{(p-1) -p/(d-2)} = \frac{1}{n} \cdot (nT \sqrt{\eps_n})^p
\end{equation}
where the implicit constant depends on $p$ and $d$ only. Using this inequality (written also in the case $p=2$) in \eqref{eq:appendix-c-rosenthal} we obtain exactly \eqref{eq:upper-bound-n}. This settles the case $d\ge 5$, $p>0$.

In the case $d \in \cur{3,4}$, $p<(d-2)/2$, we essentially  repeat the argument  of \Cref{prop:smoothed-wasserstein}. Let us first argue for a general measure $\nu$ on $\T^d$ with $\nu(\T^d) = nT$ and then specialize to the random case $\nu = \sum_{i=1}^n \mu^{B^i}_T$. We use the triangle inequality
\begin{equation}
W^p_{\T^d}\bra{ \nu } \le W^p_{\T^d}\bra{\nu, \nu_\ell } + W^p_{\T^d}\bra{ \nu_\ell},
\end{equation}
where we write $\nu_\ell$ as in \eqref{eq:nu-ell} and set $\ell := (nT)^{-1/(d-2)}$. Then, from \eqref{eq:nu_ell-convolution},  it follows easily that
\begin{equation}
 W^p_{\T^d}\bra{\nu, \nu_\ell } \les (nT) \ell^p \les (nT)^{1-p/(d-2)},
\end{equation}
hence it is sufficient to bound the second term, for which we employ \eqref{eq:wp-concave-appendix-c}. Indeed, when specialized to the random case it is sufficient to argue as in \eqref{eq:expected-wp-concave-appendix-c}, with the only difference that instead of \cite[lemma 3.2]{huesmann2022wasserstein}, we have the upper bound, for $\xi \in \mathbb{Z}^d \setminus \cur{0}$,
\begin{equation}
\EE\sqa{ \abs{ \sum_{i=1}^n \widehat{ \mu_T^{B_i}}(\xi)}^2 }  =n \EE\sqa{ \abs{ \widehat{ \mu_T^{B_1}}(\xi)}^2 } \les   \frac{ n T}{|\xi|^2},
\end{equation}
which follows easily from independence among the $B_i$'s and then an application of \cite[lemma 3.2]{huesmann2022wasserstein}.
\end{proof}

\printbibliography

\end{document}